\documentclass[10pt,reqno]{amsart}
\usepackage{amssymb,mathrsfs, mathtools}%,pxfonts}
\usepackage{wrapfig} 
\usepackage{tikz}
\usepgflibrary{fpu}
\usetikzlibrary{arrows,calc,decorations.pathreplacing}
\definecolor{light-gray1}{gray}{0.90}
\definecolor{light-gray2}{gray}{0.80}

\mathtoolsset{showonlyrefs}

\newcommand{\A}{\mathcal{A}}
\newcommand{\B}{\mathcal{B}}

\newcommand{\E}{\mathcal{E}}

\newcommand{\J}{\mathcal{J}}
\newcommand{\LL}{\mathcal{L}}

\newcommand{\HH}{\mathcal{H}}

\newcommand{\cS}{\mathcal{S}}
\newcommand{\Sp}{\mathbb{S}}

\newcommand{\N}{\mathbb{N}}
\newcommand{\R}{\mathbb{R}}
\newcommand{\Z}{\mathbb{Z}}

\newcommand{\al}{\alpha}
\newcommand{\be}{\beta}
\newcommand{\ga}{\gamma}
\newcommand{\de}{\delta}
\newcommand{\e}{\varepsilon}
\newcommand{\fy}{\varphi}
\newcommand{\om}{\omega}
\newcommand{\la}{\lambda}

\newcommand{\ta}{\tau}

\newcommand{\x}{\xi}

\newcommand{\De}{\Delta}
\newcommand{\Om}{\Omega}

\newcommand{\p}{\partial}

\newcommand{\supp}{\operatorname{supp}}

\newcommand{\I}{\infty}

\newcommand{\ti}{\widetilde}
\newcommand{\ba}{\overline}

\newcommand{\ds}{\displaystyle}
\newcommand{\ang}[1]{\left\langle{#1}\right\rangle}
\newcommand{\abs}[1]{\left\lvert{#1}\right\rvert}

\newcommand{\EQ}[1]{\begin{equation}\begin{split} #1 \end{split}\end{equation}}
\newcommand{\Del}[1]{}

\def\ti{\tilde}

\numberwithin{equation}{section}

\newtheorem{thm}{Theorem}[section]
\newtheorem{cor}[thm]{Corollary}
\newtheorem{lem}[thm]{Lemma}
\newtheorem{prop}[thm]{Proposition}
\newtheorem{claim}[thm]{Claim}

\theoremstyle{remark}
\newtheorem{rem}{Remark}

\newcommand{\mand}{{\ \ \text{and} \ \  }}

\newcommand{\mas}{{\ \ \text{as} \ \ }}

\newcommand{\ud}{\mathrm{d}} 
\newcommand{\eps}{\epsilon}

\begin{document}

\author{R.~C\^{o}te}
\author{C.~E.~Kenig}
\author{A.~Lawrie}
\author{W.~Schlag}

\title[Large energy solutions of the equivariant wave map problem: I]{Characterization of large energy solutions of the equivariant wave map problem: I}
\begin{abstract} We consider $1$-equivariant wave maps from $\R^{1+2} \to \Sp^2$. For wave maps of topological degree zero we prove global existence and scattering for energies below twice the energy of harmonic map, $Q$, given by stereographic projection. We deduce this result via the concentration compactness/rigidity method developed by the second author and Merle. In particular, we establish a classification of equivariant wave maps with trajectories that are pre-compact in the energy space up to the scaling symmetry of the equation. Indeed, a wave map of this type can only be either $0$ or $Q$ up to a rescaling. This gives a proof in the equivariant case of a refined version of the {\em threshold conjecture} adapted to the degree zero theory
 where the true threshold is $2\E(Q)$, not $\E(Q)$. 
The aforementioned global existence and scattering statement can also be deduced by considering the work of Sterbenz and Tataru in the equivariant setting. 

For wave maps of topological degree one, we establish a classification of solutions blowing up in finite time with energies less than three times the energy of $Q$. Under this restriction on the energy, we show that a blow-up solution of degree one is essentially the sum of a rescaled $Q$ plus a remainder term of topological degree zero of energy less than twice the energy of $Q$. This result reveals the universal character of the known blow-up constructions for degree one, $1$-equivariant wave maps of Krieger, the fourth author, and Tataru as well as Rapha\"{e}l and Rodnianski. 
\end{abstract}

\thanks{Support of the National Science Foundation DMS-0968472 for the second author, and  DMS-0617854, DMS-1160817 for the fourth author is gratefully acknowledged. 
This first author wishes to thank the University of Chicago for its hospitality during the academic year 2011-12, and acknowledges support from the European Research Council through the project BLOWDISOL. The authors thank Jacek Jendrej for pointing out a gap in the proof of Theorem 1.1 and for his fix of this gap, which is included in Appendix B}



\maketitle

 \section{Introduction}
 Wave maps are defined formally as critical points of the Lagrangian 
 \begin{align*}
 \LL(U, \p U) = \frac{1}{2} \int_{\R^{1+d}} \eta^{\al \be} \ang{ \p_{\al} U\, , \, \p_{\be}U  }_g \, dt \, dx.
 \end{align*}
Here $U: (\R^{1+d}, \eta) \to (M, g)$ where $\eta= \textrm{diag}(-1, 1,\dots, 1)$ is the Minkowski metric on $\R^{1+d}$ and $M$ is a Riemannian manifold with metric $g$. Critical points of $\LL$ satisfy the Euler-Lagrange equation
\begin{align*}
 \eta^{\al \be} D_{\al} \p_{\be} U = 0,
 \end{align*} 
 where $D$ is the pull-back covariant derivative on $U^*TM$.  In local coordinates on $(M, g)$, the Cauchy problem for wave maps is  given by
  \begin{align} \label{cp i}
  &\Box U^k =  - \eta^{\al \be} \Gamma^k_{i j}(U) \p_{\al} U^i \p_{\be} U^j\\
  &(U, \p_t U)\vert_{t=0}=(U_0, U_1), \notag
  \end{align}
  where $\Gamma_{i j}^k$ are the Christoffel symbols on $TM$. Equivalently, we can consider the extrinsic formulation for wave maps.  If $M \hookrightarrow \R^N$ is embedded, critical points are characterized by 
\begin{align*}
\Box U \perp T_U M.
\end{align*}
Here, the Cauchy problem becomes 
  \begin{align*}
  &\Box U = \eta^{\al \be} S(U)( \p_{\al} U, \p_{\be} U)\\
   &(U, \p_t U)\vert_{t=0}=(U_0, U_1),
\end{align*}
where $S$ is the second fundamental form of the embedding.  One should note  that harmonic maps from $\R^d \to M$ are wave maps that do not depend on time.  

Wave maps exhibit a conserved energy,  
\begin{align}\label{en e}
\E(U, \p_t U)(t) = \int_{\R^d}  ( \abs{\p_t U}^2_g +  \abs{\nabla U}_g^2 ) \, dx = \textrm{const.},
 \end{align}
and are invariant under the scaling 
\begin{align*} 
(U(t, x), \p_t U(t, x)) \mapsto  (U(\la t, \la x), \la \p_tU(\la t, \la x)).
\end{align*}
The scaling invariance implies that the Cauchy problem is $\dot{H}^s \times \dot{H}^{s-1}(\R^d)$ critical for $s= \frac{d}{2}$, energy critical when $d=2$, and energy supercritical for $d>2$.  For a recent review of  some of the main developments in the area we refer the reader to Krieger's survey \cite{Kr}. 	

\subsection{Equivariant wave maps}
In the presence of symmetries, such as when the target manifold~$M$ is a surface of revolution, one often singles out a special class of such maps called equivariant wave maps. As an example, for the sphere $M=\Sp^{d}$ one requires that $U\circ \rho=\rho^\ell \circ U$ where the equivariance class, $\ell$, is a positive integer and 
$\rho\in SO(d)$ acts on $\R^d$ and on $\Sp^d$ by rotation, in the latter case about a fixed axis.

Here  we consider energy critical equivariant wave maps. We restrict out attention to the corotational case $\ell=1$, and study maps $U: (\R^{1+2}, \eta) \to (\Sp^2, g)$, where  $g$ is the round metric on $\Sp^2$. In spherical coordinates, $$(\psi, \om) \mapsto (\sin \psi \cos \om , \sin \psi \sin \om, \cos \psi),$$ on $\Sp^2$, the metric $g$ is given by the matrix $g = \textrm{diag}( 1, \sin^2(\psi))$. In the $1$-equivariant setting, we thus require our wave map, $U$, to have the form 
\begin{align*}
U(t, r, \om) = (\psi(t, r), \om) \mapsto (\sin \psi(t, r) \cos \om ,\, \sin \psi(t, r) \sin \om, \,\cos \psi(t, r)),
 \end{align*}
 where $(r, \om)$ are polar coordinates on $\R^2$. In this case, the Cauchy problem \eqref{cp i} reduces to 
\begin{align} \label{cp} 
&\psi_{tt} - \psi_{rr} - \frac{1}{r} \psi_r + \frac{\sin(2\psi)}{2r^2} = 0\\
&(\psi, \psi_t)\vert_{t=0} = (\psi_0, \psi_1). \notag
\end{align} 

We note that equivariant wave maps to surfaces of revolution such as the sphere have been extensively studied, and we refer the reader to the works of Shatah \cite{Sh}, Christodoulou, Tahvildar-Zadeh \cite{CTZ}, Shatah, Tahvildar-Zadeh \cite{STZ1,STZ}, Struwe  \cite{St}, and the book by Shatah, Struwe \cite{SS} for a summary of these developments. 

In this equivariant setting, the conservation of energy becomes
\begin{align}
\E(U, \p_tU)(t)=\E(\psi, \psi_t)(t) = \int_0^{\infty} \left(\psi_t^2 + \psi_r^2 + \frac{\sin^2(\psi)}{r^2}\right) \, r\, dr = \textrm{const.} 
\end{align}
Any $\psi(r,t)$  of finite energy  and continuous dependence on $t\in I:= (t_{0},t_{1})$ must satisfy $\psi(t, 0) = m\pi$ and  $\psi(t,\infty)=n\pi$ for all $t\in I$, where $m, n$ are fixed integers.  This requirement splits the energy space into disjoint classes according to this topological condition. The wave map evolution preserves these classes.

 In light of this discussion, the natural spaces in which to consider Cauchy data for \eqref{cp} are the energy classes
 \begin{align}\label{Hnm}
 \HH_{m, n} := \{ (\psi_0, \psi_1) \,\vert\,  \E(\psi_0, \psi_1) < \infty \quad \textrm{and} \quad \psi_0(0) =m\pi, \, \psi_0(\infty) = n\pi\}. 
 \end{align}
We will mainly consider the spaces $\HH_{0, n}$ and we denote these by $\HH_n := \HH_{0, n}$. In this case we refer to $n$ as the degree of the map. We also define $\HH= \bigcup_{n\in \Z} \HH_{n}$ to be the full energy space. 

In the analysis of $1$-equivariant wave maps to the sphere, an important role is played by the harmonic map, $Q$, given by stereographic projection. In spherical coordinates, $Q$ is given by $Q(r) = 2\arctan(r)$ and is a solution to 
\begin{align}\label{Q}
Q_{rr} + \frac{1}{r}Q_r = \frac{\sin(2Q)}{2r^2}.
\end{align} 
One can show via an explicit calculation that $(Q, 0)$ is an element of $\HH_1$, i.e., $Q$ has finite energy and sends the origin in $\R^2$ to the north pole and spacial infinity to the south pole. In fact, the energy $\E(Q):=\E(Q,0)=4$ is minimal in $\HH_1$ and simple phase space analysis shows that, up to a rescaling,  $(Q, 0)$ is the unique, nontrivial, $1$-equivariant harmonic map to the sphere in $\HH_1$. Note the slight abuse of notation above in that we will denote the energy of the element $(Q, 0) \in \HH_1$ by $\E(Q)$ rather than $\E(Q, 0)$.  

It has long been understood that in the energy-critical setting, the geometry of the target should play a decisive role in determining the asymptotic behavior of  wave maps. For equivariant wave maps, global well-posedness for all smooth data was established by Struwe in \cite{St} in the case where the target manifold does not admit a non-constant finite energy harmonic sphere.  This extended the results of Shatah, Tahvildar-Zadeh \cite{STZ1}, and Grillakis \cite{G}, where global well-posedness was proved for targets satisfying a geodesic convexity condition. Recently, global well-posedness, including scattering,  has been established in the full (non-equivariant), energy critical wave maps problem in a remarkable series of works \cite{KS}, \cite{ST1}, \cite{ST2},  \cite{T3}, for targets that do not admit finite energy harmonic spheres, completing the program developed in \cite{Tat}, \cite{T2}. 

However, finite-time blow-up can occur in the case of compact targets that admit non-constant harmonic spheres. Because we are working in the equivariant, energy critical setting, blow-up can only occur at the origin and in an energy concentration scenario which amounts to a breakdown in regularity. Moreover, in \cite{St}, Struwe showed that if a solution is~$C^{\infty}$ before a regularity breakdown occurs, then such a scenario can only happen by the bubbling off of a non-constant harmonic map. 

In particular, Struwe showed that if a solution, $\psi(t,r)$, with smooth initial data $\vec\psi(0)= (\psi(0), \dot{\psi}(0))$, breaks down at $t=1$, then the energy concentrates at the origin and there  is a sequence of times $t_j \nearrow 1$ and scales $\la_j>0$ with  $\la_j \ll 1-t_j$ so that the rescaled sequence of wave maps $$\vec \psi_j(t, r) := \left(\psi(t_j+ \la_jt,\la_jr), \la_j \dot \psi\left(t_j+ \la_jt, \la_j r\right)\right)$$ converges {\em locally} to $ \pm Q(r/ \la_0)$ in the space-time norm  ${H}^1_{\textrm{loc}}((-1, 1) \times \R^2; \Sp^2)$ for some $\la_0>0$. Further evidence of finite time blow up for equivariant wave maps to the sphere was provided by the first author in \cite{Co}. Recently, explicit blow-up solutions have been constructed in \cite{RS} for equivariance classes $\ell \ge 4$ and in the $1$-equivariant case in \cite{KST}, \cite{KST2} and \cite{RR}. In \cite{KST},  Krieger, the fourth author, and Tataru constructed explicit blow-up solutions  with prescribed blow-up rates $\la(t)= (1-t)^{1 + \nu}$ for $\nu >\frac{1}{2}$ although it is believed that all rates with $\nu >0$ are possible as well. In \cite{KST2}, a similar result is given for the radial, energy critical Yang Mills equation. In \cite{RR}, Rodnianski and Rapha\"{e}l give a description of  stable blow-up dynamics for  equivariant wave maps and the radial, energy critical Yang Mills equation in an open set about $Q$ in a stronger topology than the energy.

Our goal in this paper is twofold. On one hand, we study the asymptotic behavior of solutions to \eqref{cp} with data in the ``zero" topological class, i.e., $\vec \psi(0) \in \HH_0$, below a sharp energy threshold, namely $2\E(Q)$. Additionally, we seek to classify the behavior of wave maps of topological degree one, i.e., those with data $\vec \psi \in \HH_1$,  that blow up in finite time with energies below the threshold $3\E(Q)$. In particular, we show that blow-up profiles exhibited in the works \cite{KST}, \cite{RS} and \cite{RR} are universal in this energy regime in a precise sense described below in Section~\ref{sect bu thm}. 

\subsection{Global existence and scattering for wave maps in $\HH_0$ with energy below $2\E(Q)$}
We begin with a description of our results in the degree zero case. In \cite{St}, Struwe's work implies that  solutions $\vec \psi(t)$ to \eqref{cp} with data $\vec \psi(0) \in \HH_0$ are global in time if $\E(\vec \psi(0)) <2 \E(Q)$. This follows directly from the fact that wave maps in $\HH_0$ with energy below $2\E(Q)$ stay bounded away from the south pole and hence cannot converge, even locally, to a degree one rescaled harmonic map, thus ruling out blow-up. Recently, the first two authors together with Merle, in \cite{CKM}, extended this result to include scattering to zero in the regime, $\vec{\psi}(0) \in \HH_0$ and $\E(\vec \psi) \le \E(Q)+ \de$ for small $\de>0$. It was conjectured in \cite{CKM} that scattering should also hold for all energies up to $2\E(Q)$. This conjecture is a refined version of what is usually called  {\em  threshold conjecture}, adapted to the case of topologically trivial equivariant data. It is implied by the recent work of Sterbenz and Tataru in \cite{ST1}, \cite{ST2} when one considers their results in the equivariant setting with topologically trivial data. Here we give an alternate proof of this  {\em refined threshold conjecture} in  the equivariant setting based on the concentration compactness/rigidity method of the second author and Merle, \cite{KM06}, \cite{KM08}. 
 In particular, we prove the following:

%------------------------------------------------deg 0 g.e. and scatt-----------------------------------------------------------------------%
\begin{thm} [Global Existence and Scattering in $\HH_0$ below $2\E(Q)$]\label{main}  For any smooth data $\vec\psi(0) \in \HH_0$ with $\E(\vec\psi(0)) < 2\E(Q)$, there exists a unique global evolution $\vec \psi \in C^0(\R; \HH_0)$. Moreover, $\vec\psi(t)$ scatters to zero in the sense that the energy of $\vec \psi(t)$ on any arbitrary, but fixed compact region vanishes as $t \to \infty$. In other words, one has 
\begin{align} \label{scat}
 \vec{\psi}(t) = \vec{\fy}(t) + o_{\HH}(1) \quad \textrm{as} \quad t \to \infty
 \end{align}
 where $\vec \fy \in \HH$ solves the linearized version of \eqref{cp}, i.e., 
 \begin{align}\label{lin wave}
 \fy_{tt} - \fy_{rr} - \frac{1}{r} \fy_r + \frac{1}{r^2} \fy = 0
 \end{align}
Furthermore, this result is sharp in $\HH_0$ in sense that $2\E(Q)$ is a true threshold. Indeed for all $\de>0$ there exists data $\vec\psi(0) \in \HH_0$ with $\E(\vec\psi) \le 2\E(Q) + \de$, such that $\vec \psi$ blows up in finite time.  
\end{thm}
%------------------------------------------------------------------------------------------------------------------------------------------%

\begin{rem}  We note that a threshold result as in Theorem~\ref{main} only makes sense in $\HH_0$. Indeed, all initial data in $\HH_1$  have enough energy to blow-up by bubbling off a harmonic map in the sense of Struwe's result in \cite{St}, since $Q$ minimizes the energy in $\HH_1$. The same goes for all higher degrees. In Section~\ref{sharp} we construct a degree zero wave map which blows up in finite time using the explicit {\em degree one} blow up solutions of Krieger, the fourth author and Tataru. This example will also help to illustrate why the twice the energy of the degree one map $Q$ gives the sharp threshold for degree zero maps.  
\end{rem}

\begin{rem}Characterizing the possible dynamics at the threshold, $\vec \psi \in \HH_0$, $\E(\vec \psi)=2\E(Q)$ and above $\E(\vec \psi)>2\E(Q)$, remain open questions.
\end{rem}
\begin{rem}
We briefly remark that Theorem~\ref{main} holds with the same assumptions and conclusions for data $\vec \psi \in \HH_{n, n}$ where $\HH_{n, n}$ is defined as in \eqref{Hnm}. Indeed, the spaces $\HH_{0}$ and  $\HH_{n, n}$ are isomorphic via the map  $(\psi_0, \psi_1) \mapsto (\psi_0+n\pi, \psi_1)$. Also, we can replace the words ``smooth finite energy data" in Theorem~\ref{main} with  just ``finite energy data" using the well-posedness theory for \eqref{cp}, see for example \cite{CKM}. 
\end{rem}
 As mentioned above,  Theorem~\ref{main} is established by the concentration compactness/rigidity method of the second author and Merle in \cite{KM06} and \cite{KM08}. The novel aspect of our implementation of this method lies in the development of a robust rigidity theory for wave maps $\vec U(t)$ with trajectories that are pre-compact in the energy space up to certain time-dependent modulations. We note that the following theorem is independent of both the topological class and the energy of the wave map. 
 
  %-------------------------------------------------Rigidity Thm 1-------------------------------------------------------------% 
 \begin{thm}[Rigidity]\label{rigidity} Let $\vec{U}(t,r, \om) = ((\psi(t,r), \om), (\dot{\psi}(t, r), 0)) \in \HH$ be a solution to \eqref{cp}  and let $I_{\max}(\psi)= (T_-(\psi), T_+(\psi))$ be the maximal interval of existence. Suppose that there exists $A_0>0$ and a continuous function $\la : I_{\max} \to [A_0, \infty)$ such that the set 
 \begin{align} \label{ K}
\ti{K}:= \left\{\,\left (\, U\left(t, \frac{r}{\la(t)}, \om \right),\,  \frac{1}{\la(t)}\p_tU\left( t,\frac{r}{ \la(t)}, \om \right) \right)\, \Big{\vert}\, t \in I_{max} \right\}
 \end{align}
 is pre-compact in $\dot{H}^1 \times L^2(\R^2 ; \Sp^2)$. Then, $I_{\max}= \R$ and either $ U \equiv  0$ or $U: \R^2 \to \Sp^2$ is an equivariant harmonic map,  i.e., $U(t, r, \om) = (\pm Q(r/\ti\la), \om)$ for some  $\ti \la>0$.
 \end{thm}
 %-------------------------------------------------Rigidity Thm 1-------------------------------------------------------------%
\begin{rem} 
To establish  Theorem~\ref{main} we only need a version of Theorem~\ref{rigidity} that deals with data in $\HH_0$ below $2\E(Q)$. This rigidity result in $\HH_0$ is given in Theorem~\ref{rigidity2} below, and states that any solution $\vec \psi\in \HH_0$ with a pre-compact rescaled trajectory must be identically zero.  
The full result in Theorem~\ref{rigidity} is established for its own interest. In fact, we use the conclusions of Theorem~\ref{main} in order to deduce the full classification of pre-compact solutions given in Theorem~\ref{rigidity}.  Alternatively, we can prove Theorem~\ref{rigidity} using the scattering result of \cite[Theorem $1$]{CKM}, and deduce Theorem~\ref{rigidity2} as a corollary. We have chosen the former approach here to illustrate the independence of our stronger rigidity results from the variational arguments given in \cite[Lemma~$7$]{CKM}. 
\end{rem}

\subsection{Classification of blow-up solutions in $\HH_1$ with energies below $3\E(Q)$}\label{sect bu thm}

We now turn to the issue of describing blow-up for wave maps in $\HH_1$, i.e., those maps $\vec \psi(t)$ with $\psi(t, 0)= 0$ and $\psi(t, \infty) = \pi$. From here on out, any wave map that is assumed to blow-up will be also be assumed  to do so at time $t=1$. As mentioned above, the recent works \cite{KST} and \cite{RR} construct explicit blow-up solutions $ \psi(t) \in \HH_1$. In \cite{KST}, the blow up solutions constructed there exhibit a decomposition of the form 
\begin{align}\label{kst dec}
\psi(t, r) = Q(r/ \la(t)) + \epsilon(t, r)
\end{align}
where the concentration rate satisfies $\la(t) = (1-t)^{1+ \nu}$ for $\nu>\frac{1}{2}$, and $\epsilon(t)\in \HH_0$ is small and regular.  Here we consider  the converse problem. Namely, if blow-up does occur for a solution $\vec \psi(t) \in \HH_1$, in which energy regime, and in what sense does such a decomposition always hold?  

The works of Struwe, in \cite{St} for the equivariant case, and Sterbenz, Tataru in \cite{ST2} for the full wave map problem, give a partial answer to this question. As mentioned above, they show that if blow-up occurs, then along a sequence of times, a sequence of rescaled versions of the original wave map converge {\em locally} to $Q$ in the space-time norm ${H}^1_{\textrm{loc}}( (-1, 1)\times \R^2; \Sp^2)$. However working locally removes any knowledge of the topology of the wave map, which is determined by the behavior of the map at spacial infinity. In this paper we seek to strengthen the results in \cite{St} and \cite{ST2} in the equivariant setting by working globally in space in  the {\em energy topology}. Here we are forced to account for the topological restrictions of a degree one wave map, and in fact we use these restrictions, along with our degree zero theory, to our advantage. 

In particular, we make the following observation. If a wave map $\psi(t) \in \HH_1$ blows up at $t=1$ then the local convergence results of Struwe in \cite{St} allow us to extract the blow up profile $\pm Q_{\la_n}:= \pm Q(\cdot/ \la_n)$ at least along a sequence of times $t_n \to 1$. If $\vec \psi$ has energy below $3\E(Q)$ the profile must be $+Q(\cdot/ \la_n)$, and since $Q \in \HH_1$ as well we thus have  $\psi(t_n)- Q_{\la_n} \in \HH_0$. Since this  object should converge locally to zero, the energy of the difference should be roughly the difference of the energies, at least for large $n$.  Hence, if $\psi(t)$ has energy below $3\E(Q)$ the difference $\psi(t_n) - Q_{\la_n}$ is degree zero and has energy below $2\E(Q)$. By Theorem~\ref{main}, we then suspect that the blow-up profile already extracted is indeed universal in this regime and that a decomposition of the form \eqref{kst dec} should indeed hold, excluding the possibility of any different dynamics, such as more bubbles forming. We prove the following result:

%-------------------------------------------------Blow up Theorem-------------------------------------------------------------%

\begin{thm}[Classification of blow-up solutions in $\HH_1$ with energies below $3\E(Q)$]\label{bu cont} Let $\vec \psi(t) \in \HH_1$  be a smooth solution to  \eqref{cp} blowing up at time $t=1$ with  $$\E(\vec \psi) = \E(Q) + \eta < 3\E(Q). $$ Then, there exists a continuous function, $\la:[0,1) \to (0, \infty)$ with $\la(t) = o(1-t)$, a map $\vec \fy=(\fy_0, \fy_1) \in \HH_0$ with $\E(\vec \fy)= \eta$, and a decomposition  
\begin{align}\label{dec1}
\vec \psi(t) =   \vec \fy + \left(Q\left(\cdot/\la(t)\right), 0\right) + \vec\epsilon(t)
\end{align}
such that $\vec \epsilon(t) \in \HH_0$ and $\vec \epsilon(t)\to 0$ in $\HH_0$ as $t \to 1$. %In fact, we show that $\vec \epsilon(t)\to 0$ in $H \times L^2$ as $t \to 1$.
\end{thm}

%-------------------------------------------------Blow up Theorem-------------------------------------------------------------%
\begin{rem}
In the companion work \cite{CKLS2} we address the question of {\em global solutions} $\psi(t) \in \HH_1$ in the regime $\E(\vec \psi)< 3\E(Q)$.  We can show that in this case we have a decomposition and convergence as in \eqref{dec1} with $\la(t) \ll t$ as $t \to \infty$.  This will give us a complete classification of the possible dynamics in $\HH_1$ for energies below $3\E(Q)$. Of course, our results do not give information about the precise rates $\la(t)$.  We also would like to mention the recent results of Bejenaru, Krieger, and Tataru \cite{BKT}, regarding wave maps in $\HH_1$, where they prove asymptotic orbital stability for a co-dimension two class of initial data which is ``close" to $Q_{\la}$ with respect to a stronger topology than the energy. 
\end{rem}

\begin{rem}
Theorem~\ref{bu cont} is reminiscent of the recent results proved by Duyckaerts, the second author, and Merle in \cite{DKM1}, \cite{DKM2}, for the energy critical focusing semi-linear wave equation in $\R^{1+3}$. In fact, the techniques developed in these works provided important ideas  for the proof of Theorem~\ref{bu cont}. The situation for wave maps is somewhat different, however, as the geometric nature of the problem provides some key distinctions. The most notable of these distinctions is that the underlying linear theory for wave maps of degree zero is not nearly as strong as that of a semi-linear wave in $\R^{1+3}$, which causes serious problems. Indeed, as demonstrated in \cite{CKS}, the strong lower bound on the exterior energy in \cite[Lemma $4.2$]{DKM1} {\em fails} for  general initial data in even dimensions. This difficulty is overcome by the fact that there is no self-similar blow-up  for energy critical equivariant wave maps, see e.g.,~\cite{SS}, which can be shown directly due to the non-negativity of the energy density. 

In addition, our degree zero result and the rigid topological restrictions of the problem allow us to extend the conclusions of Theorem~\ref{bu cont} all the way up to $3\E(Q)$ instead of just slightly above the energy of the harmonic map $\E(Q) + \de$, for $\de>0$ small, as is the case in \cite{DKM1}, \cite{DKM2}. This large enegy result is similar in nature to the results for the $3d$ semi-linear radial wave equation in \cite{DKM3}, when, in the notation from \cite{DKM3}, $J_0=1$. 
\end{rem}

\begin{rem}
The results in \cite{DKM1}, \cite{DKM2} have recently been extended by Duyckaerts, the second author, and Merle in \cite{DKM3} and \cite{DKM4}. In \cite{DKM4}, a classification of solutions to the radial, energy critical, focusing semi-linear wave equation in $\R^{1+3}$ of {\em all} energies is given in the sense that only three scenarios are shown to be possible;  $(1)$  type I blow-up; $(2)$ type II blow-up with the solution decomposing into a sum of blow-up profiles arising from rescaled solitons plus a radiation term; or $(3)$ the solution is global and decomposes into a sum of rescaled solitons plus a radiation term as $t \to \infty$. 
\end{rem}

\subsection{Remarks on the proofs of the main results}

In addition to the methods originating in \cite{KM06}, \cite{KM08} and \cite{DKM1}, \cite{DKM2}, the work in this paper rests explicitly on several developments in the field over the past two decades. Here we provide a quick guide to the work on which our results lie: 
\subsubsection{Results used in the proof of Theorem~\ref{main}}
\begin{itemize} 

\item Theory of equivariant wave maps developed in the nineties in the works of Shatah, Tahvildar-Zadeh, \cite{STZ1}, \cite{STZ}, including the use of virial identities to prove energy decay estimates. 
\item The concentration compactness decomposition of Bahouri-G\'{e}rard, \cite{BG}.
\item Lemma $2$ in \cite{CKM} which relates energy constraints to $L^{\infty}$ estimates for equivariant wave maps. In particular, if a degree zero map has energy less than $2\E(Q)$, then the evolution, $\psi(t, r)$, is bounded uniformly below $\pi$. In addition, although only a weaker small data result such as \cite[Theorem $8.1$]{SS} is needed, we use the global existence and scattering result for degree one wave maps with energy below $\E(Q)+ \de$ for small $\de>0$, which was established in \cite[Theorem~$1$]{CKS}. 
\item H\'{e}lein's theorem on the regularity of harmonic maps which says that a weakly harmonic map is, in fact, harmonic, \cite{Hel}. 
\end{itemize}
\subsubsection{Results used in the proof of Theorem~\ref{bu cont}}
\begin{itemize}

\item The virial identity and the corresponding energy decay estimates in \cite{STZ1}. 
\item Struwe's characterization of blow-up, \cite[Theorem $2.2$]{St}, which gives $H^1_{\textrm{loc}}$ convergence along a sequence of times to $Q$ if blow-up occurs. This allows us, a priori, to identify and extract the blow-up profile $Q_{\la_n}$ along a sequence of times, $t_n$, which is absolutely crucial in our argument since we can then work with degree zero maps once $Q_{\la_n}$ has been subtracted from the degree one maps $\psi(t_n)$.   
\item The concentration compactness decomposition of Bahouri-G\'{e}rard, \cite{BG}.
\item The new results on the free radial $4d$ wave equation established by the first, second, and fourth authors in \cite{CKS}. %In particular, we use the fact that for a $4d$ free radial wave $v(t)$ with initial data $(f, 0)$ we have the following lower bound for the exterior energy at time $t$:
%\begin{align*}
%  \| f\|_{\dot{H}^1}\lesssim \|v(t) \|_{\dot{H}^1 \times L^2( r \ge t)}.
%\end{align*}

%In the  proof of Theorem~\ref{bu cont} we are able to relate a sequence of degree zero wave maps with a sequence of $4d$ radial free waves with {\em zero} initial velocities. 
 \item The decomposition of degree one maps which have energy slightly above $Q$ and the stability of this decomposition under the wave map evolution for a period of time inversely proportional to the proximity of the data  to $Q$ in the energy space established by the first author in \cite{Co}.   

\end{itemize}

As we outline in the appendix, the proofs of Theorem~\ref{main}, Theorem~\ref{rigidity},  and Theorem~\ref{bu cont} extend easily to energy critical $1$-equivariant wave maps with more general targets.  In addition, the proofs of Theorem~\ref{rigidity} and Theorem~\ref{main} apply equally well to the equivariance classe $\ell =2$ and the $4d$ equivariant Yang-Mills system after suitable modifications. One should also be able to deduce these results for the equivariance classes $\ell \ge 3$ once a suitable small data theory is established for these equations, which are similar in nature to the even dimensional energy critical semi-linear wave equations in high dimensions treated in \cite{BCLPZ} -- the difficulty here resides in the low fractional power in the nonlinearity.  

However, the method we used to prove Theorem~\ref{bu cont} only works, as developed here, for odd equivariance classes, $\ell=1, 3, 5, \dots$, and does not work  when one considers  even equivariance classes, $\ell = 2, 4, 6, \dots$, or the $4d$ equivariant Yang-Mills system  in this context.  This failure of our technique arises in the linear theory in \cite{CKS} for even dimensions, which provides favorable estimates for our proof scheme only when $\ell$ is odd. Since the $4d$ equivariant Yang-Mills system corresponds roughly to a $2$-equivarant wave map, this falls outside the scope of our current method as well. To be more specific, one can identify the linearized $\ell$-equivariant wave map equation with the $2 \ell + 2$-dimensional free radial wave equation. In the final stages of the proof of Theorem~\ref{bu cont}, and in particular Corollary~\ref{ext en est}, we require the exterior energy estimate 
\begin{align*}
 \| f\|_{\dot{H}^1}\lesssim \|S(t)(f, 0) \|_{\dot{H}^1 \times L^2( r \ge t)} \quad \textrm{for all} \quad t \ge 0
\end{align*}
where $S(t)$ is the the free radial wave evolution operator. In \cite{CKS}, this estimate is shown to be true in even dimensions $4, 8, 12, \dots$, and false in dimensions $2, 6, 10, \dots$. Without this estimate, our proof would  show compactness of the error term in our decomposition in a certain suitable Strichartz space but not in the energy space. Therefore, the full conclusion of Theorem~\ref{bu cont} remains open for the $4d$ equivariant Yang-Mills system and the $\ell$-equivariant wave map equation when $\ell$ is even. 

\subsection{Structure of the paper}
The outline of the paper is as follows. In Section~\ref{prelim} we establish the necessary preliminaries needed for the rest of the work. We include a brief review of the results of Shatah, Tahvildhar-Zadeh, \cite{STZ1} and Struwe \cite{St}. We also recall the concentration compactness decomposition of Bahouri, G\'{e}rard \cite{BG} and adapt their theory to case of equivariant wave maps to the sphere. In particular, we deduce a Pythagorean expansion of the nonlinear wave map energy of such a decomposition at a fixed time. This type of result is crucial in the concentration compactness/rigidity method of  \cite{KM06}, \cite{KM08}. We also establish an appropriate nonlinear profile decomposition. 

In Section~\ref{km} we give a brief outline of the concentration compactness/rigidity method that is used to prove Theorem~\ref{main}. In Section~\ref{rigid} we prove Theorem~\ref{rigidity}, which allows us to complete the proof of Theorem~\ref{main}. 

Finally, in Section~\ref{sect bu} we establish Theorem~\ref{bu cont}, which relies crucially on the linear theory developed in ~\cite{CKS}. 
 
\subsection{Notation and Conventions} We will interchangeably use the notation $\psi_t(t, r)$ and $\dot\psi(t, r)$ to refer to the derivative with respect to the time variable $t$ of the function $\psi(t, r)$.    

The notation $X \lesssim Y$ means that there exists a constant $C>0$ such that $X \le CY$. Similarly,  $X \simeq Y$ means that there exist constants $0<c<C$ so that $cY \le X \le CY$. 
 
\section{Preliminaries}\label{prelim}
We define the energy space  $$\HH= \{\vec U \in \dot{H}^1 \times L^2(\R^2; \Sp^2) \, \vert\, U \circ \rho = \rho \circ U, \,\,\, \forall \rho \in SO(2)\}.$$ $\HH$ is endowed with the norm 
\begin{align}\label{en norm}
\E(\vec U(t)) = \|\vec U(t) \|_{\dot{H}^1 \times L^2(\R^2; \Sp^2)}^2 = \int_{\R^2}  ( \abs{\p_t U}^2_g +  \abs{\nabla U}_g^2 ) \, dx.
\end{align}
As noted in the introduction, by our equivariance condition we can write $U(t, r, \om) = (\psi(t,r), \om)$ and the energy of a wave map becomes 
\begin{align}\label{en psi}
\E(U, \p_tU)(t)=\E(\psi, \psi_t)(t) = \int_0^{\infty} \left(\psi_t^2 + \psi_r^2 + \frac{\sin^2(\psi)}{r^2}\right) \, r\, dr = \textrm{const.} 
\end{align}
We also define the localized energy as follows: Let $r_1, r_2 \in [0, \infty)$. Then we set 
\begin{align*}
\E_{r_1}^{r_2}(\vec \psi(t)):= \int_{r_1}^{r_2}  \left(\psi_t^2 + \psi_r^2 + \frac{\sin^2(\psi)}{r^2}\right) \, r\, dr.
\end{align*}
Following Shatah and Struwe, \cite{SS}, we set
\begin{align} 
G(\psi):= \int_0^{\psi} \abs{\sin \rho} \, d\rho.
\end{align}
Observe that for any $(\psi, 0) \in \HH_n$ and for any $r_1, r_2 \in[0, \infty)$ we have 
\begin{align}\label{b R}
\abs{G(\psi(r_2))-G(\psi(r_1))} &=\abs{\int_{\psi(r_1)}^{\psi(r_2)} \abs{\sin \rho} \, d\rho} 
\\
&= \abs{\int_{r_1}^{r_2} \abs{\sin(\psi(r))} \psi_r(r) \, dr} \notag
 \le\frac{1}{2} \E_{r_1}^{r_2}(\psi, 0)
\end{align}

\subsection{Properties of degree zero wave maps}
As in \cite{CKM}, let $\al \in [0, 2\E(Q)]$ and define the   set $V(\al)\subset \HH_0$: 
 $$V(\al):=\{(\psi_0, \psi_1) \in  \HH_{0}\, \vert \, \E(\psi_0, \psi_1) < \al \}$$  
We claim that for every $\al \in [0, 2\E(Q)]$, $V(\al)$ is naturally endowed with the norm 
\begin{align}\label{HxL^2}  
\|(\psi_0, \psi_1)\|_{H \times L^2}^2 = \int_0^{\infty} \left(\psi_1^2 + (\psi_0)_r^2 + \frac{ \psi_0^2}{r^2} \right) \, r\, dr
\end{align} 
To see this, we recall the following lemma proved in \cite{CKM}.
\begin{lem}\cite[Lemma $2$]{CKM}\label{ckm lem2}
There exists an increasing function $K:[0, 2\E(Q)) \to [0, \pi)$ such that  
\begin{align}\label{pw} 
\abs{\psi(r)} \le K(\E(\vec{\psi})) < \pi \quad \forall \vec{\psi} \in \HH_0 \quad \textrm{with} \quad \E(\psi) <2\E(Q)
\end{align} 
Moreover, for each $\al \in [0, 2\E(Q)]$ we have 
\begin{align}\label{H=E}
\E(\psi_0, \psi_1) \simeq \|(\psi_0, \psi_1) \|_{H\times L^2}
\end{align} 
for every $(\psi_0, \psi_1) \in V(\al)$, with the constant above depending only on $\al$.  
\end{lem}
 
When considering Cauchy data for \eqref{cp} in the class $\HH_{0}$ the formulation in \eqref{cp} can be modified in order to take into account the strong repulsive potential term that is  hidden in the nonlinearity: 
\begin{align*} 
\frac{\sin(2\psi)}{2r^2}  = \frac{ \psi}{r^2} + \frac{ \sin(2\psi) - 2\psi}{2r^2}  =  \frac{ \psi}{r^2} + \frac{O(\psi^3)}{r^2} 
\end{align*}
Indeed, the presence of the strong repulsive potential $\frac{1}{r^2}$ indicates that the linearized operator of \eqref{cp} has more dispersion than the $2$-dimensional wave equation. In fact, it has the same dispersion as the $4$-dimensional wave equation as the following standard reduction shows. 

Setting $\psi = r u$ we are led to this equation for $u$:
\begin{align}\label{4d}
&u_{tt} - u_{rr} -\frac{3}{r} u_r + \frac{\sin(2r u) - 2ru}{2r^3} = 0\\
&\vec u(0)= (u_0, u_1). \notag
\end{align} 
The nonlinearity above has the form $N(u, r) = u^3 Z(ru)$ where $Z$ is a smooth, bounded, even function and the linear part is the radial d'Alembertian in $\R^{1+4}$. The linearized version of \eqref{4d} is just the free radial wave equation in $\R^{1+4}$, namely 
\begin{align} \label{4d free}
&v_{tt} - v_{rr} -\frac{3}{r} v_r =0.
\end{align}
Observe that for $\vec \psi(0) \in \HH_0$ we have that  
\begin{align} \label{2-4}
 \E(\vec\psi(0)) \le \|\vec \psi\|_{H \times L^2}^2:=\int_0^{\infty}\left(\psi_t^2+  \psi_r^2 + \frac{\psi^2}{r^2}\right) \, r \, dr = \int_0^{\infty}(u_t^2+ u_r^2) \, r^3 \, dr.
\end{align} 
If, in addition, we assume that $\E(\vec\psi(0))<2\E(Q)$ then, by Lemma~\ref{ckm lem2} we also have the opposite inequality
\begin{align} \label{4d equiv}
\|\vec u(0)\|_{\dot{H}^1 \times L^2}^2 = \|\vec{\psi}(0) \|_{H\times L^2}^2  \lesssim  \E(\vec{\psi}(0)).
\end{align}
Therefore, when considering initial data  $(\psi_0, \psi_1) \in V(\al)$ for $\al \le 2\E(Q)$ the Cauchy problem \eqref{cp} is equivalent to the Cauchy problem for \eqref{4d} for radial initial data $(r \psi_0, r\psi_1)=:\vec u(0) \in \dot{H}^1 \times L^2 (\R^4)$.  
 
The following exterior energy estimates for the $4d$ free radial wave equation established by the first, second, and fourth authors  in \cite{CKS} will play a key role in our analysis: 
\begin{prop}\cite[Corollary $5$]{CKS}\label{lin ext estimate} Let $S(t)$ denote the free evolution operator for the $4d$ radial wave equation, \eqref{4d free}. There exists $\al_0>0$ such that for all $t  \ge 0$ we have
\begin{align} \label{4d lin ext est}
\|S(t)(f, 0)\|_{\dot{H}^1 \times L^2( r \ge t)} \ge  \al_0 \|f\|_{\dot{H}^1}
\end{align}
for all radial data $(f, 0) \in \dot{H}^1 \times L^2$. 
\end{prop}
The point here is that this same result applies to the linearized version of the wave map equation: 
\begin{align} \label{2d lin}
\fy_{tt} -\fy_{rr} -\frac{1}{r} \fy_r + \frac{1}{r^2} \fy = 0
\end{align} 
with initial data $\vec \fy(0) = ( \fy_0, 0)$. Indeed we have the following: 
\begin{cor}\label{linear ext en} Let $W(t)$ denote the linear evolution operator associated to \eqref{2d lin}. Then there exists $\be_0>0$ such that for all $t \ge 0$ we have 
\begin{align} \label{2d lin ext est}
\|W(t)(\fy_0, 0)\|_{H \times L^2( r \ge t)} \ge  \be_0 \|\fy_0\|_{H}
\end{align} 
for all radial initial data $(\fy_0, 0) \in H \times L^2$. 
\end{cor}
\begin{proof} Let $\vec \fy(t) = W(t)( \fy_0, 0)$ be the linear evolution of the smooth radial data $(\fy_0, 0) \in H \times L^2$.  Define $\vec v(t)$ by $\fy(t, r) = r v(t, r)$. Then $\vec v(t) \in \dot{H}^1 \times L^2 (\R^4) $ and is a solution to \eqref{4d free} with initial data $(v_0, 0) = (\frac{\fy_0}{r}, 0)$. Next observe that for all $A \ge 0$ we have
\begin{align*} 
\|v(t)\|_{\dot{H}^1(r \ge A)}^2 = \int_A^{\infty} v_r^2(t, r) \, r^3\, dr &= \int_A^{\infty} \left( \frac{\fy_r(t, r)}{r} - \frac{\fy(t, r)}{r^2} \right)^2 \, r^3 \, dr\\
&\le 2 \| \fy(t)\|_{H(r\ge A)}^2
\end{align*}
Similarly we can show that $\|\fy(t) \|_{H (r \ge A)}^2 \le  2\|v(t)\|_{ \dot{H}^1( r \ge A)}^2$. Therefore using \eqref{4d lin ext est} on $v(t)$ we obtain 
\begin{align*} 
\|\vec \fy(t) \|_{H \times L^2(r \ge t)}^2  \ge \frac{1}{2} \|v(t)\|_{\dot{H}^1(r \ge t)}^2 \ge \frac{\al_0^2}{2} \|v_0\|_{\dot{H}^1}^2 =  \frac{\al_0^2}{2} \|\fy_0\|_{ H}^2
\end{align*}
which proves \eqref{2d lin ext est} with $\be_0 = \frac{\al_0}{\sqrt{2}}$. 
\end{proof}

 \subsection{Properties of degree one wave maps}
 
Now, suppose $\vec \psi=(\psi_0, \psi_1) \in \HH_1$. This means that  $\psi(0)=0$ and $\psi(\infty) = \pi$. The $H \times L^2$ norm of $\vec \psi$ is no longer finite, but we do have the following comparison: 
\begin{lem}
Let $\vec \psi =(\psi_0, 0) \in \HH_1$ be smooth and let $r_0 \in [0, \infty)$. Then there exists $\al>0$ such that 
\begin{itemize} 
\item[($a$)] If $\E_0^{r_0}( \vec \psi) <\al$, then 
\begin{align}\label{HH1 0}
\|\psi\|_{H(r \le r_0)}^2 \lesssim  \E_0^{r_0}(\vec \psi).
\end{align}
\item[($b$)] If $\E_{r_0}^{\infty}( \vec \psi) <\al$, then 
\begin{align}\label{HH1 inf}
\|\psi(\cdot)- \pi\|_{H(r \ge r_0)}^2 \lesssim  \E_{r_0}^{\infty}(\vec \psi).
\end{align}
\end{itemize}
\end{lem}
\begin{proof}
We prove only the second estimate as the proof of the first is similar. Since $G(\pi)=2$, by \eqref{b R} we have for all $r \in[r_0, \infty)$ that
\begin{align*}
\abs{G(\psi(r)) -2} \le\frac{1}{2} \E_r^{\infty}(\psi, 0) < \frac{\al}{2}.
\end{align*}
Since $G$ is continuous and increasing this means that $\psi(r) \in [\pi- \e(\al), \pi+ \e(\al)]$ where $\e(\rho) \to 0$ as $\rho \to 0$. Hence for $\al$ small enough we have the estimate $\sin^2( \psi(r)) \ge \frac{1}{2} \abs{ \psi(r) - \pi}^2$ for all $r \in [r_0, \infty]$ and the estimate \eqref{HH1 inf} follows by integrating this.
\end{proof}

Let $Q(r):= 2 \arctan(r)$. Note that $(Q, 0)\in \HH_1$ is the unique (up to scaling) time-independent,  solution to \eqref{cp} in $\HH_1$. Indeed, $Q$ has minimal energy in $\HH_1$ and $\E(Q, 0)= 4$. 
One way to see this is to note that $Q$ satisfies $r Q_r(r) = \sin(Q)$ and hence for any $0\le a\le  b < \infty$ we have 
\begin{align} 
G(Q(b))-G(Q(a)) = \int_{a}^{b} \abs{\sin(Q(r))} Q_r(r) \, dr = \frac{1}{2} \E_a^b(Q, 0)
\end{align}
Letting $a \to 0$ and $b \to \infty$ we obtain $\E(Q, 0) = 2G(\pi) = 4$. To see that  $\E(Q, 0)$ is indeed minimal in $\HH_1$, observe that we can factor the energy as follows:
\begin{align*} 
\E(\psi, \psi_t)&=  \int_0^{\infty} \psi_t^2\, r\, dr + \int_0^{\infty} \left(\psi_{r}  - \frac{\sin(\psi)}{r}\right)^2 \, r\, dr + 2\int_0^{\infty} \sin(\psi) \psi_r \, dr\\
&= \int_0^{\infty} \psi_t^2\, r\, dr + \int_0^{\infty} \left(\psi_{r}  - \frac{\sin(\psi)}{r}\right)^2 \, r\, dr + 2\int_{\psi(0)}^{\psi(\infty)} \sin(\rho)  \, d\rho
\end{align*} 
Hence, in $\HH_1$ we have 
\begin{align} \label{var char}
\E(\psi, \psi_t) \ge  \int_0^{\infty} \psi_t^2\, r\, dr + 4 =   \int_0^{\infty} \psi_t^2\, r\, dr + \E(Q)
\end{align}

We shall also require a decomposition from~\cite{Co} which amounts to the coercivity of the energy near to ground state~$Q$,  up to the scaling symmetry. 

\begin{lem}\cite[Proposition $2.3$]{Co}\label{decomposition} There exists a function $\de: (0, \infty) \to (0, \infty)$ such that $\de(\al) \to 0$ as  $\al \to 0$ and such that the following holds: Let $\vec \psi = ( \psi, 0) \in \HH_1$. Define 
\begin{align*}
\al :=\E(\vec \psi)- \E(Q)>0
\end{align*}
Then there exists $\la \in(0, \infty)$ such that 
\begin{align*}
\|\psi - Q( \cdot/ \la)\|_{H} \le \de(\al)
\end{align*}
Note that one can choose  $\la>0$ so that $\E_0^{\la}(\vec \psi) = \E_0^1(Q)= \E(Q)/2$. 
\end{lem}
We will also need the following consequence of Lemma~\ref{decomposition} that is also proved in~\cite{Co}. 
 \begin{cor}\cite[Corollary $2.4$]{Co} \label{Cote} Let  $\rho_n, \sigma_n \to \infty$ be two sequences such that  $\rho_n \ll \sigma_n$. Let $\vec \psi_n(t) \in \HH_1$ be a sequence of wave maps defined on time intervals $[0, \rho_n]$ and suppose that 
 \begin{align*}
   \|\vec \psi_n(0)- (Q, 0)\|_{H \times L^2}  \le   \frac{1}{\sigma_n}.
   \end{align*}
   Then 
    \begin{align*} 
  \sup_{t \in[0, \rho_n]}\|\vec \psi_n( t)- (Q, 0)\|_{H \times L^2}  = o_n(1) \quad \textrm{as} \quad n  \to \infty
   \end{align*}
 \end{cor}
 \begin{rem} 
 We refer the reader to the proof of \cite[Corollary $2.4$]{Co} and the remark immediately following it for a detailed proof of Corollary~\ref{Cote}. We have phrased the above result in terms of sequences of wave maps because this is the form in which it will be applied in Section~\ref{sect bu}.  Also, we note that in~\cite{Co} the notation $\|\cdot\|_H^2$ is used to denote the nonlinear energy, $\E(\cdot)$, of a map, whereas here $\|\cdot\|_H$ is defined as in \eqref{HxL^2}.  Both Lemma~\ref{decomposition} and Corollary~\ref{Cote} hold with either definition. 
 \end{rem}
 \subsection{Properties of blow-up solutions}
 Now let $\vec \psi(t) \in \HH$ be a wave map with maximal interval of existence $I_{\max}( \vec \psi) =(T_-(\vec \psi), T_+( \vec \psi))  \neq~\R$.
 By translating in time, we can assume that $T_+(\vec \psi)=1$.  We recall a few facts that we will need in our argument. From the work  of Shatah and Tahvildar-Zadeh \cite{STZ1}, Jia and and second author~\cite{JK-AJM}, and from Appendix B to this paper, we have the following results:
\begin{lem}\cite[Lemma $2.2$]{STZ1}\label{ext en decay}
Let $\vec \psi(t) \in \HH$ on the interval $[0, 1)$. For any $\la \in(0, 1]$ we have 
\begin{align}
\E_{\la(1-t)}^{1-t} (\vec \psi(t)) = \int_{\la(1-t)}^{1-t} \left(\psi_t^2(t, r) + \psi_r^2(t, r) + \frac{\sin^2(\psi(t, r))}{r^2} \right) \, r\, dr \to 0 \quad \textrm{as} \quad t \to 1
\end{align}
\end{lem}
\begin{proof} 
Lemma~\ref{ext en decay} for smooth wave maps  was proved in~\cite[Lemma 2.2]{STZ1}, and in~\cite[Lemma 2.1]{JK-AJM} for energy class solutions to the $6d$ focusing quadratic wave equation. One may apply a nearly identical argument as the one in~\cite{JK-AJM} to prove Lemma~\ref{ext en decay} here. Alternatively, see Appendix~\ref{a:STZ} for an argument extending this result to $\vec\psi(t) \in \HH$.
\end{proof}

\begin{lem}\cite[Corollary $2.2$]{STZ1}\label{t av decay} Let $\vec{\psi}(t) \in \HH$ be a solution to \eqref{cp} such that $I_{\max}(\vec \psi)$ is a finite interval. Without loss of generality we can assume $T_+(\vec \psi) =1$. Then we have 
\begin{align} \label{dot psi decay} 
\frac{1}{1-t} \int_{t}^1 \int_0^{1-s} \dot \psi^2(s, r) \, r\, dr\, ds \to 0 \quad \textrm{as} \quad t \to 1
\end{align}
\end{lem}
\begin{proof} 
As above Lemma~\ref{ext en decay} for smooth wave maps  was proved in~\cite[Lemma 2.3]{STZ1}. See Appendix~\ref{a:STZ} for an argument extending this result to $\vec\psi(t) \in \HH$. 
\end{proof}

As in \cite{DKM1}, we can use Lemma~\ref{t av decay} to establish the following result. The proof is identical to the argument given in  \cite[Corollary $5.3$]{DKM1} so we do not reproduce it here.

\begin{cor}\cite[Corollary $5.3$]{DKM1}\label{t dec extract} Let $\psi(t) \in \HH$ be a solution to \eqref{cp} such that $T_+(\vec \psi)=1$. Then, there exists a sequence of times $ \{t_n\} \nearrow 1$ such that for every $n$ and for every $\sigma \in (0, 1-t_n)$, we have
\begin{align} 
\frac{1}{\sigma} \int_{t_n}^{t_n+ \sigma} \int_0^{1-t}  \dot \psi^2(t, r) \, r\, dr \, dt &\le \frac{1}{n} \label{sig}\\
 \int_0^{1-t_n}  \dot \psi^2(t_n, r) \, r\, dr & \le \frac{1}{n}\label{t dec} 
 \end{align}
 \end{cor}
Note that \eqref{t dec} follows from \eqref{sig} by letting $\sigma \to 0$ in \eqref{sig} and recalling the continuity of the map $t \mapsto \dot \psi(t, \cdot)$ from $[0,1) \to L^2$.

We now recall a result of Struwe, \cite{St}, which will be essential in our argument for degree~$1$. 

\begin{thm} \cite[Theorem $2.1$]{St} \label{str} Let $\psi(t) \in \HH$ be a smooth solution to \eqref{cp} such that $T_+(\vec \psi)=1$.  Let $\{t_n\}\nearrow 1$ be defined as in Corollary~\ref{t dec extract}. Then there  exists  a sequence  $\{\la_n\}$ with $\la_n=o(1-t_n)$ so that the following results hold: Let 
\begin{align} \label{psi n def}
\vec \psi_n(t, r) := (\psi(t_n+ \la_n t, \la_n r), \la_n \dot \psi(t_n + \la_n t, \la_n r))
\end{align}
be the wave map evolutions associated to the data $\vec \psi_n(r):= \vec \psi(t_n, \la_n r)$. And denote by $U_n(t, r, \om) := (\psi_n(t, r) , \om)$  the full wave maps. Then, 
\begin{align} \label{loc}
U_n(t, r, \om) \to U_{\infty}(r, \om) \quad \textrm{in} \quad{H}^1_{loc}((-1, 1) \times \R^2 ; \Sp^2)
\end{align} 
where $U_{\infty}$ is a smooth, non-constant, $1$-equivariant,  time independent solution to \eqref{cp i}, and hence $U_{\infty}(r, \om)=(\pm Q(r/\la_0), \om)$ for some $\la_0>0$. We further note that after passing to a subsequence, $U_n(t, r, \om) \to U_{\infty}(r, \om)$ locally uniformly in $(-1, 1) \times (\R^2-\{0\})$.
 
 Moreover, with the times $t_n$ and scales $\la_n$ as above, we have 
 \begin{align}\label{sig with la}
 \frac{1}{\la_n} \int_{t_n}^{t_n+ \la_n} \int_0^{1-t}  \dot \psi^2(t, r) \, r\, dr \, dt = o_n(1).
 \end{align}
\end{thm} 

\begin{rem} We note that we have altered the selection procedure by which  the sequence of times $t_n$ is chosen in the proof of Theorem~\ref{str}. In \cite{St}, after defining a scaling factor $\la(t)$, Struwe uses Lemma~\ref{t av decay} to select a sequence of times $t_n$ via an argument involving Vitali's covering theorem, and he sets $\la_n:= \la(t_n)$. Here we do something different. Given Lemma~\ref{t av decay} we use the argument in \cite[Corollary $5.3$]{DKM1} to find a sequence $t_n \to 1$ so that \eqref{sig} and \eqref{t dec} hold. Now we choose the scales $\la(t)$ as in Struwe and for each $n$ we set $\sigma = \la_n:= \la(t_n)$ and we establish  \eqref{sig with la}, which is exactly \cite[Lemma $3.3$]{St}. The rest of the proof of Theorem~\ref{str} now proceeds exactly as in \cite{St}. 
\end{rem}

We will also need the following consequences of Theorem~\ref{str}:

\begin{lem} \label{Hloc conv} Let $\psi(t) \in \HH$ be a solution to \eqref{cp} such that $T_+(\vec \psi)=1$. Let  $\{t_n\}\nearrow 1$ and  $\{\la_n\}$ be chosen as in Theorem~\ref{str}. Define $\psi_n(t, r)$, $\pm Q(r/ \la_0)$ as in \eqref{psi n def}. Then 
\begin{align}\label{Hloc pm}
\psi_n \mp Q( \cdot /\la_0)  \to 0\quad{as} \quad n \to \infty \quad \textrm{in} \quad L^2_t((-1, 1); H_{\textrm{loc}})
\end{align}
where $H$ is defined as in \eqref{HxL^2}. 
\end{lem}
\begin{proof} We prove the case where the convergence in Theorem~\ref{str} is to  $+Q(r/ \la_0)$. Let  $Q_{\la_0}(r)= Q(r/ \la_0)$. By Theorem~\ref{str}, we know that 
\begin{multline} \label{H1 loc}
\int_{\R^{1+2}} \left( \abs{ \p_t\psi_n(t, r)}^2+  \abs{ \p_r(\psi_n(t, r)- Q_{\la_0}(r))}^2 \right)\chi(t, r) \, r \, dr \, dt \\+ \int_{\R^{1+2}}  \abs{\psi_n(t, r) - Q_{\la_0}(r)}^2 \chi(t, r) \, r \, dr \, dt  \longrightarrow 0 \quad \textrm{as} \quad n \to \infty
\end{multline}
for all $ \chi \in C_0^{\infty}((-1, 1) \times \R^2)$, radial in space. Hence to prove \eqref{Hloc pm}, it suffices to show that 
\begin{align}\label{hardy piece}
\int_{\R^{1+2}}  \frac{\abs{\psi_n(t, r)- Q_{\la_0}(r)}^2}{r^2} \chi(t, r) \, r \, dr \, dt \to 0 \quad \textrm{as}  \quad n \to \infty
\end{align}
for all $\chi$ as above. Next, note that if for fixed $\de>0$, $\chi(t, r)$ satisfies $\textrm{supp}( \chi(t, \cdot)) \subset [\de, \infty)$, we have 
\begin{multline*}
\int_{\R^{1+2}}  \frac{\abs{\psi_n(t, r)- Q_{\la_0}(r)}^2}{r^2} \chi(t, r) \, r \, dr \, dt \\ \le \de^{-2}\int_{\R^{1+2}}  \abs{\psi_n(t, r)- Q_{\la_0}(r)}^2 \chi(t, r) \, r \, dr \, dt  \to 0 \quad \textrm{as}  \quad n \to \infty, 
\end{multline*}
with the convergence in the last line following from \eqref{H1 loc}. Hence, from here out we only need to consider $\chi$ with $\textrm{supp} \chi(t, \cdot) \subset [0, 1]$.  Referring to Struwe's argument in \cite[Proof of Theorem $2.1$, (ii)]{St}, we note that by construction, $\la_n$ and $\la_0$ are such that
\begin{align*}
&\E_0^1( \vec \psi_n(t)) < \e_1, \quad
\E_0^1(Q_{\la_0}) < \e_1
\end{align*}
uniformly in $\abs{t} \le1$ and uniformly in $n$, where $\e_1>0$ is a fixed constant that we can choose to be as small as we want. Recalling that for each $t$,  $\psi(t, 0)= Q(0) = 0$ and using  \eqref{b R}, this implies  that 
\begin{align*}
\abs{G( \psi_n(t, r)) } \le \frac{1}{2} \e_1, \quad
\abs{G(Q_{\la_0}(r))} \le \frac{1}{2} \e_1
\end{align*}
for all $r \in [0, 1]$. 
In particular, we can choose $\e_1$ small enough so that 
\begin{align*}
\abs{\psi_n(t, r)}< \frac{\pi}{8} , \quad \abs{Q_{\la_0}(r)} < \frac{\pi}{8}
\end{align*}
for all $r \in[0, 1]$. Using the above line we then can conclude that there exists $c>0$ such that
\begin{align}\label{sin bound}
(\psi_n(t, r)- Q(r/\la_0))(\sin(2 \psi_n(t, r))- \sin(2 Q_{\la_0}(r)))\ge c (\psi_n(t, r)- Q(r/\la_0))^2
\end{align}
for all $r \in[0, 1]$, and $\abs{t} \le 1$.  Consider the equation 
\begin{align*}
(-\p_{tt} + \p_{rr} + \frac{1}{r}\p_r)(\psi_n(t, r)- Q_{\la_0}(r)) = \frac{\sin(2 \psi_n(t, r)) - \sin(2Q_{\la_0}(r))}{r^2}.
\end{align*}
Now, let $\chi \in C_0^{\infty}((-1, 1) \times \R^2)$ satisfy $\textrm{supp}( \chi(t, \cdot)) \subset [0,1]$.  Multiply the above equation by $(\psi_n(t, r)- Q_{\la_0}(r)) \chi(t, r)$, and integrate over $\R^{1+2}$. Then, integrating by parts and using the strong local convergence in \eqref{H1 loc} we can deduce that 
\begin{align*}
\int_{\R^{1+2}}  \frac{( \sin(2\psi_n(t, r))-\sin(2 Q_{\la_0}(r))) (\psi_n(t, r)- Q(r/\la_0))}
{r^2} \chi(t, r) \, r \, dr \, dt \to 0 
\end{align*}
as $n \to \infty$. The lemma then follows by combining the above line with \eqref{sin bound}. 
\end{proof}

\begin{lem}\label{pi seq lem pm}Let $\psi(t) \in \HH$ be a wave map that blows up at time $t=1$. Then, there exists a sequence of times $ \bar t_n \to 1$ and a sequence of points $r_n \in [0, 1- \bar t_n)$ such that 
\begin{align} \label{pi seq pm}
 \psi(\bar t_n, r_n) \to \pm \pi \quad \textrm{as} \quad n \to \infty
\end{align} 
\end{lem} 
\begin{proof} 
If not, then there exists a  $\de_0 >0 $ such that for every time $t \in [0, 1)$  we have $\abs{\psi(t, r) } \in \R -[ \pi- \de_0, \pi+ \de_0]$ for all $r \in [0, 1-t)$. Now let $t_n$, $\la_n$ and $\psi_n(t, r)$  and $\pm Q_{\la_0}$ be as in Theorem~\ref{str} and Lemma~\ref{Hloc conv}.
Choose $0<R_1<R_2 < \infty$ so that $\abs{Q_{\la_0}(r)} > \pi-\frac{\de_0}{2}$ for $r \in [R_1, R_2]$ and choose $N$ large enough so that $[\la_n R_1, \la_n R_2]  \subset [0, 1-t_n - \la_n t) $ for all $t\in [0, 1]$ and for all $n \ge N$.  This implies that 
\begin{align} \label{not close}
\abs{\psi_n(t, r) \mp Q_{\la_0}(r)} \ge \frac{\de_0}{2} \quad \forall n \ge N, \quad \forall r \in[R_1, R_2],
\end{align}
and for all $t \in [0 ,1]$. But this provides an immediate contradiction with the convergence in \eqref{Hloc pm}. 
\end{proof}

\begin{cor}  \label{pi seq lem}Let $\psi(t) \in \HH_1$ be a wave map that blows up at time $t=1$ such that $\E(\vec \psi)<3 \E(Q)$. Recall that $\vec \psi(t) \in \HH_1$ means that $\psi(t, 0)=0, \psi(t,\infty)= \pi$. Then we have
\begin{align}\label{Hloc}
\psi_n - Q( \cdot /\la_0)  \to 0\quad{as} \quad n \to \infty \quad \textrm{in} \quad L^2_t((-1, 1); H_{\textrm{loc}}),
\end{align}
with $\psi_n(t, r)$, $t_n$, and $\la_n$ defined as in Theorem~\ref{str}. In addition, there exists another sequence of times $\bar t_n \to 1$ and a sequence of points $r_n \in [0, 1-\bar t_n)$ such that 
\begin{align} \label{pi seq}
 \psi(\bar t_n, r_n) \to  \pi \quad \textrm{as} \quad n \to \infty
\end{align} 
\end{cor} 
\begin{proof}
We use the energy bound $\E(\vec \psi) < 3 \E(Q)$  to eliminate the possibility that the convergence in Theorem~\ref{str} is to $-Q(r/\la_0)$ instead of to $+ Q(r/ \la_0)$.  Suppose that in fact we had in \eqref{Hloc pm} that $\psi_n +Q(\cdot/ \la_n) \to 0$ in $L^2_t((-1, 1); H_{\textrm{loc}})$. Lemma~\ref{pi seq lem pm} then gives  a sequence of times $\bar t_n \to 1$ and a sequence $r_n \in [0, 1- \bar t_n)$ such that 
\begin{align}\label{-pi}
\psi(\bar t_n, r_n) \to - \pi
\end{align}
as $n \to \infty$. Now recall that $\vec \psi(t) \in \HH_1$. Using the above along with \eqref{b R} we see that 
\begin{align*}
2\E(Q) = 8 \leftarrow 2\abs{G(\psi(\bar t_n, r_n))- 2 } \le \E_{r_n}^{\infty}(\psi(\bar t_n), 0))
\end{align*}
On the other hand, we can use \eqref{-pi} and \eqref{b R} again to see that  
\begin{align*}
\E(Q)=4\leftarrow 2\abs{G(\psi(\bar t_n, r_n))} \le \E_0^{r_n}(\psi(\bar t_n), 0)
\end{align*}
Putting this together we see that we must have $\E(\vec \psi) \ge 3 \E(Q)$ which contradicts our initial assumption on the energy. 
\end{proof}

\subsection{Profile Decomposition} 
Another essential ingredient of our argument is the profile decomposition of Bahouri and Gerard \cite{BG}. Here we restate the main results of \cite{BG} and then adapt these results to the case of $2d$ equivariant wave maps to the sphere of topological degree zero. In fact the results for the $4d$ wave equation stated here first appeared in \cite{Bu} as the decomposition in \cite{BG} was performed only in dimension $3$.  In particular, we recall the following result:

\begin{thm} \cite[Main Theorem]{BG} \label{BaG}\cite[Theorem $1.1$]{Bu}
Consider a sequence of data $\vec u_n \in \dot{H}^1 \times L^2( \R^4)$ such that $ \|u_n\|_{\dot{H}^1 \times L^2} \le C$. Then, up to extracting a subsequence,  there exists a sequence of free $4d$ radial waves $\vec V_L^j  \in \dot{H}^1 \times L^2$,  a sequence of times $\{t_n^j\}\subset \R$, and sequence of scales $\{\la_n^j\}\subset (0, \infty)$, such that for $\vec w_n^k$ defined by 
\begin{align} \label{4d lin prof} 
&u_{n, 0}(r) = \sum_{j=1}^k \frac{1}{\la_n^j}V_L^j( -t_n^j/ \la_n^j, r/ \la_n^j) + w_{n, 0}^k(r)\\
& u_{n, 1}(r)= \sum_{j=1}^k \frac{1}{(\la_n^j)^2}\dot V_L^j( -t_n^j/ \la_n^j, r/ \la_n^j) + w_{n, 1}^k(r)
\end{align}
we have, for any $j \le k$, that 
\begin{align} \label{w weak}
(\la_{n}^j w_n^k( \la_n^j t_n^j,  \la_n^j\cdot) , (\la_n^j)^2 w_n^k( \la_n^j t_n^j,  \la_n^j\cdot)) \rightharpoonup 0\quad \textrm{weakly in} \quad \dot{H}^1 \times L^2(\R^4). 
\end{align}
In addition, for any $j\neq k$ we have
\begin{align} \label{po scales}
\frac{\la_n^j}{\la_n^k} + \frac{\la_n^k}{\la_n^j} + \frac{\abs{t_n^j-t_n^k}}{\la_n^j} + \frac{\abs{t_n^j-t_n^k}}{\la_n^k} \to \infty \quad \textrm{as} \quad n \to \infty.
\end{align}
Moreover, the errors $\vec w_n^k$ vanish asymptotically in the sense that if we let $w_{n, L}^k(t) \in \dot{H}^1 \times  L^2$ denote the free evolution, (i.e., solution to \eqref{4d free}), of the data $\vec w_n^k \in \dot{H}^1 \times L^2$, we have 
\begin{align} \label{w in strich}
\limsup_{n \to \infty} \left\| w_{n, L}^k\right\|_{L^{\infty}_tL^4_x \cap L^3_tL^6_x( \R \times \R^4)}  \to 0 \quad \textrm{as} \quad k \to \infty.
\end{align}
Finally, we have the almost-orthogonality of the $\dot{H}^1 \times L^2$ norms of the decomposition: 
\begin{align} \label{free en ort}
\|\vec u_n\|_{\dot H^1 \times L^2}^2 = \sum_{1 \le j \le k} \| \vec V_L^j( - t_n^j/ \la_n^j) \|_{\dot{H}^1 \times L^2}^2  +  \|\vec w_n^k\|_{\dot H^1\times L^2}^2 + o_n(1)
\end{align} 
as  $n \to \infty$.
\end{thm}

The norms appearing in~\eqref{w in strich}  are dispersive and examples of Strichartz estimates, see Lindblad, Sogge~\cite{LinS} and Sogge's book~\cite{Sogge} for more background and details. 
For our purposes here, it will often be useful to rephrase the above decomposition in the framework of the $2d$ linear wave equation \eqref{lin wave}. Using the right-most equality in \eqref{2-4}  together with the identifications 
\begin{align*} 
& \psi_n(r) =  r u_n(r)\\
& \fy^j_L( -t_n^j/ \la_n^j, r/ \la_n^j) = \frac{r}{\la_n^j}V^j_L( -t_n^j/ \la_n^j, r/ \la_n^j)\\
& \ga_n^k(r)  = r w_n^k,
\end{align*}
we see that Theorem~\ref{BaG} directly implies the following decomposition for  sequences $\vec \psi_n \in \HH_0$ with uniformly bounded $H \times L^2$ norms. In particular, by~\eqref{4d equiv}, the following corollary holds for all sequences $\vec \psi_n \in \HH_0$ with $\E(\vec\psi_n) \le C <2\E(Q)$.  

\begin{cor}\label{bg wm} Consider a sequence of data $\vec \psi_n \in \HH_0$ that is uniformly bounded in $H \times L^2$. Then, up to extracting a subsequence,  there exists a sequence of linear waves $\vec \fy^j_L \in \HH_0$, (i.e., solutions to \eqref{lin wave}),  a sequence of times $\{t_n^j\}\subset \R$, and a sequence of scales $\{\la_n^j\}\subset (0, \infty)$, such that for $\vec \ga_n^k$ defined by 
\begin{align} 
&\psi_{n, 0}(r) = \sum_{j=1}^k \fy^j_L( -t_n^j/ \la_n^j, r/ \la_n^j) + \ga_{n, 0}^k(r)\\
& \psi_{n, 1}(r)= \sum_{j=1}^k \frac{1}{\la_n^j}\dot\fy^j_L( -t_n^j/ \la_n^j, r/ \la_n^j) + \ga_{n, 1}^k(r)
\end{align}
we have, for any $j \le k$, that 
\begin{align} 
(\ga_n^k( \la_n^j t_n^j,  \la_n^j\cdot) , \la_n^j \ga_n^k( \la_n^j t_n^j,  \la_n^j\cdot)) \rightharpoonup 0\quad \textrm{weakly in} \quad H \times L^2. 
\end{align}
In addition, for any $j\neq k$ we have
\begin{align} \label{po}
\frac{\la_n^j}{\la_n^k} + \frac{\la_n^k}{\la_n^j} + \frac{\abs{t_n^j-t_n^k}}{\la_n^j} + \frac{\abs{t_n^j-t_n^k}}{\la_n^k} \to \infty \quad \textrm{as} \quad n \to \infty.
\end{align}
Moreover, the errors $\vec \ga_n^k$ vanish asymptotically in the sense that if we let $\ga_{n, L}^k(t) \in \HH_0$ denote the linear evolution, (i.e., solution to \eqref{lin wave}) of the data $\vec \ga_n^k \in \HH_0$, we have 
\begin{align} 
\limsup_{n \to \infty} \left\|\frac{1}{r} \ga_{n, L}^k\right\|_{L^{\infty}_tL^4_x \cap L^3_tL^6_x( \R \times \R^4)}  \to 0 \quad \textrm{as} \quad k \to \infty.
\end{align}
Finally, we have the almost-orthogonality of the $H \times L^2$ norms of the decomposition: 
\begin{align} \label{ort H} 
\|\vec \psi_n\|_{H \times L^2}^2 = \sum_{1 \le j \le k} \| \vec \fy_L^j( - t_n^j/ \la_n^j) \|_{H \times L^2}^2  +  \|\vec \ga_n^k\|_{H \times L^2}^2 + o_n(1)
\end{align} 
as  $n \to \infty$.
\end{cor}
In order to apply the concentration-compactness/rigidity method developed by  the second author and Merle in \cite{KM06}, \cite{KM08}, we need the following ``Pythagorean decomposition'' of the nonlinear energy \eqref{en psi}: 
\begin{lem} \label{en orth lem}
Consider a sequence $\vec \psi_n \in \HH_0$ and a decomposition as in Corollary~\ref{bg wm}. Then this Pythagorean decomposition holds for the energy of the sequence: 
\begin{align} \label{en orth} 
\E(\vec \psi_n) = \sum_{j=1}^k \E(\vec \fy_L^j(-t_n^j/ \la_n^j)) + \E(\vec \ga_n^k) + o_n(1) 
\end{align}
 as $ n \to \infty$. 
 \end{lem}
 \begin{proof} 
 By \eqref{ort H}, it suffices to show for each $k$ that 
 \begin{align*} 
 \int_0^{\infty} \frac{\sin^2\left( \psi_n\right)}{r} \,  dr =  \sum_{j=1}^k  \int_0^{\infty} \frac{\sin^2\left( \fy^j_L(-t_n^j/ \la_n^j)\right)}{r} \,  dr +  \int_0^{\infty} \frac{\sin^2\left( \ga^k_n\right)}{r} \,  dr + o_n(1).
 \end{align*}
 We will need the following simple inequality: 
 \begin{align} \label{sin}
 \abs{\sin^2(x+y) -\sin^2(x) - \sin^2(y)} &= \abs{-2 \sin^2(x) \sin^2(y) + \frac{1}{2} \sin(2x) \sin(2y)}\\
 & \lesssim \abs{x} \abs{y}. \notag
 \end{align}

 Since at some point we will need to make use dispersive estimates for the $4d$ linear wave equation the argument is clearer if, at this point, we pass back to the $4d$ formulation. Recall that this means we set   
\begin{align*} 
& \psi_n(r) =  r u_n(r)\\
& \fy^j_L( -t_n^j/ \la_n^j, r/ \la_n^j) = \frac{r}{\la_n^j}V^j_L( -t_n^j/ \la_n^j, r/ \la_n^j)\\
& \ga_n^k(r)  = r w_n^k.
\end{align*}
Since we have fixed $k$, we can, by an approximation argument, assume that all of the profiles $V^j(0, \cdot)$ are smooth and supported in the same compact set, say $B(0, R)$. We seek to prove that 
\begin{align*} 
\abs{ \int_0^{\infty} \frac{\sin^2\left( r u_n\right)}{r} \, dr  - \sum_{j=1}^k  \int_0^{\infty} \frac{\sin^2\left( \frac{r}{\la_n^j}V^j_L(-t_n^j/ \la_n^j, r/ \la_n^j)\right)}{r} \,dr -  \int_0^{\infty} \frac{\sin^2\left( r w^k_n\right)}{r} \,  dr } \\
= o_n(1).
 \end{align*}
Using the inequality \eqref{sin} $k-1$ times, we can reduce our problem to showing the following two estimates:
\begin{align} 
& \int_0^{\infty}  \frac{\abs{V_L^j(-t_n^j/ \la_n^j, r/ \la_n^j)}}{(\la_n^j)}  \frac{\abs{V_L^i(-t_n^i/ \la_n^i, r/ \la_n^i)}}{(\la_n^i)}   \, r \, dr =o_n(1) \quad \textrm{for} \, \, i \neq j \label{i2}\\
& \int_0^{\infty}  \frac{\abs{V^j_L(-t_n^j/ \la_n^j, r/ \la_n^j)}}{(\la_n^j)}  \abs{w_n^k(r)}  \, r \, dr =o_n(1) \quad \textrm{for} \, \, j \le k. \label{i4}
\end{align}
From here the proof proceeds on a case by case basis where the cases are determined by which pseudo-orthogonality condition  is satisfied in \eqref{po}.  

\begin{flushleft} {\em Case $1$: $\la_n^i  \simeq \la_n^j.$}  
\end{flushleft}
In this case we may assume, without loss of generality, that $\la_n^j= \la_n^i = 1$ for all $n$. By \eqref{po} we then must have that $\abs{t_n^i-t_n^j} \to \infty$ as $n \to \infty$. This means that either $\abs{t_n^i} $ or $\abs{t_n^j}$, or both tend to $\infty$ as $n \to \infty$.  To prove \eqref{i2} we rely on the $\ang{t}^{-\frac{3}{2}}$ point-wise decay of free waves in $\R^4$. Indeed, we have 
\begin{align*} 
&\int_0^{\infty} \abs{V_L^j(-t_n^j, r)}  \abs{V_L^{i}(-t_n^i, r)} \, r \, dr \\
& \le \left( \int_{0}^{R+ \abs{t_n^j}} \abs{V_L^j(-t_n^j, r)}^2 \, r \, dr \right)^{\frac{1}{2}} \left( \int^{R+ \abs{t_n^i}}_0 \abs{V_L^i(-t_n^i, r)}^2 \, r\, dr\right)^{\frac{1}{2}}\\
&\lesssim \ang{t_n^j}^{-1/2} \ang{t_n^i}^{-1/2} = o_n(1).
\end{align*}
Next we prove \eqref{i4}. First suppose that $\abs{t_n^j} \to \infty$. Then we have 
\begin{align*} 
\int_0^{\infty}  \abs{V_L^j(-t_n^j, r)}\abs{w_n^k(r)}  \, r \, dr& \le \left( \int_{0}^{R+\abs{t_n^j}} \abs{V_L^j(-t_n^j, r)}^{2} \, r \, dr \right)^{\frac{1}{2}}\\
& \quad \times \left( \int^{\infty}_0 \abs{w_n^k(  r)}^2 \, r\, dr\right)^{\frac{1}{2}}\\
&\lesssim \|w_n^k\|_{\dot{H}^1}\ang{t_n^j}^{-\frac{1}{2}} =o_n(1)
\end{align*}
where the second inequality follows from the point-wise decay of free waves in $\R^4$ and Hardy's inequality. Finally consider the case where $\abs{t_n^j} \le C$. Then we can assume, after passing to a subsequence and translating the profile,  that $t_n^j=0$ for  every $n$. In this case, then we know that
$w_n^k \rightharpoonup 0$ weakly in $\dot{H}^1$ and hence $w_n^k\to 0$ strongly in, e.g., $L^3_{loc}(\R^4)$ as $n \to \infty$. And we have 
\begin{align*} 
\int_0^{\infty}  \abs{V_L^j(0, r)}\abs{w_n^k(r)}  \, r \, dr& \le \left( \int_{0}^{R} \abs{V_L^j(0, r)}^{\frac{3}{2}}  \, dr \right)^{\frac{2}{3}} \left( \int^{R}_0 \abs{w_n^k(  r)}^3 \, r^3\, dr\right)^{\frac{1}{3}} \\
& \le C(R) \|w_n^k\|_{L^3(B(0, R))} =o_n(1).
\end{align*}

\begin{flushleft} {\em Case $2$: $\mu_n^{i j}=\frac{\la_n^i}{\la_n^j} \to 0$  and $\frac{\abs{t_n^j}}{\la_n^j} + \frac{\abs{t_n^i}}{\la_n^i} \le C$ as $n \to \infty$.}
\end{flushleft}
We can assume, by translating the profiles, that $t_n^i = t_n^j = 0$ for all $n$. We begin by establishing \eqref{i2}. 

Changing variables we have 
\begin{align*} 
&\int_0^{\infty}  \frac{\abs{V_L^j(0, r/ \la_n^j)}}{(\la_n^j)}  \frac{\abs{V_L^i(0, r/ \la_n^i)}}{(\la_n^i)}   \, r \, dr  = \int_{0}^{R} \abs{V^j(0, r)}\mu^{ij}_n \abs{V^i(0, \mu^{ij}_n r)} \, r\, dr\\
& \le \left( \int_{0}^{R} \abs{V_L^j(0, r)}^2 \, r \, dr \right)^{\frac{1}{2}} \left( \int^R_0( \mu^{ij}_n)^2 \abs{V_L^i(0, \mu^{ij}_n r)}^2 \, r\, dr\right)^{\frac{1}{2}}\\
& \le C \left( \int_0^{R \mu_n^{ij}}  \abs{V_L^i(0, r)}^2 \, r\, dr\right)^{\frac{1}{2}} = o_n(1),
\end{align*}
where the last line follows from the fact that $R \mu_n^{ij} \to 0$ as $n \to \infty$. Next we prove \eqref{i4}. Again, we change variables to obtain
\begin{align*} 
& \int_0^{\infty}  \frac{\abs{V_L^j(-t_n^j/ \la_n^j, r/ \la_n^j)}}{(\la_n^j)}  \abs{w_n^k(r)}  \, r \, dr = \int_0^R \abs{V_L^j(0, r)} \la_n^j \abs{w^k_n( \la_n^j r)} \, r \, dr\\
 & \le \left( \int_{0}^{R} \abs{V_L^j(0, r)}^{\frac{3}{2}} \, r \, dr \right)^{\frac{2}{3}} \left( \int^R_0(\la_n^j)^3 \abs{w_n^k( \la^j_n r)}^3 \, r^3\, dr\right)^{\frac{1}{3}}= o_n(1),
 \end{align*}
 where the last line tends to $0$ as $n \to \infty$ since \eqref{w weak} implies that $\la_n^j w_n^k (\la_n^j \cdot) \to 0$ in $ L^3_{loc}(\R^4)$.  

\begin{flushleft} {\em Cases $3$: $\mu_n^{i j}=\frac{\la_n^i}{\la_n^j} \to 0$ , $\frac{\abs{t_n^j}}{\la_n^j} + \frac{\abs{t_n^i}}{\la_n^i} \to \infty$}\end{flushleft}
This remaining case can be handled by combining the techniques demonstrated in {\em Case $1$} and {\em Case $2$} using either the point-wise decay of free waves or \eqref{w weak} when applicable. We leave the details to the reader.  
\end{proof}

We will state the remaining results in this section in the $4d$ setting for simplicity.  The transition back to the $2d$ setting is straight-forward and is omitted. 

Next,  we exhibit the existence of a non-linear profile decomposition as in \cite{BG}.    We will employ the following notation: For a profile decomposition as in \eqref{4d lin prof} with profiles $\{V^j_L\}$ and parameters $\{t_n^j, \la_n^j\}$ we will denote by $\{V^j\}$ the non-linear profiles associated to $\{V^j_L(-t_n^j/ \la_n^j), \dot{V}^j_L(-t_n^j/ \la_n^j)\}$, i.e., the unique solution to \eqref{4d} such that for all $-t_n^j/ \la_n^j \in I_{\max}(V^j)$ we have
\begin{align*} 
\lim_{n \to \infty} \left\| \vec V^j(-t_n^j/\la_n^j) - \vec V_L^j(-t_n^j/ \la_n^j)\right\|_{\dot{H}^1 \times L^2} =0
\end{align*}
The existence of the non-linear profiles follows immediately from the local well-posedness theory for \eqref{4d} developed in \cite{CKM} in the case that $-t_n^j/ \la_n^j \to \tau_{\infty}^j \in \R$. If $-t_n^j/ \la_n^j  \to \pm \infty$ then the existence of the nonlinear profile follows from the existence of wave operators for \eqref{4d}.

We will make use of the following result on several occasions. 

\begin{prop}\label{nonlin profile} Let $\vec u_n \in \dot{H}^1 \times L^2$ be a uniformly bounded sequence with a profile decomposition as in Theorem~\ref{BaG}. Assume that the nonlinear profiles $V^j$ associated to the linear profiles $V^j_L$  all exist globally and scatter  in the sense that 
\begin{align*} 
\|V^j\|_{L^3_t(\R; L^6_x)} <\infty.
\end{align*}
Let $\vec u_n(t)$ denote the solution of \eqref{4d} with initial data $\vec u_n$.  Then, for $n$ large enough,  $\vec u_n(t, r)$ exists globally in time and scatters with 
\begin{align*} 
\limsup_{n \to \infty} \|u_n\|_{L^3_t(\R ; L^6_x)} < \infty.
\end{align*} 
Moreover, the following non-linear profile decomposition holds: 
\begin{align} 
u_n(t, r) = \sum_{j=1}^k \frac{1}{\la_n^j}V^j\left(\frac{t-t_n^j}{\la_n^j}, \frac{r}{\la_n^j}\right) + w_{n, L}^k(t, r)+ z_n^k(t, r)
\end{align} 
with  $w_{n, L}^k(t, r)$ as in \eqref{w in strich} and 
\begin{align}\label{nonlin error}
\lim_{k \to \infty} \limsup_{n\to \infty} \left( \|z_n^k\|_{L^3_tL^6_x} + \|\vec z_n^k\|_{L^{\infty}_t \dot{H}^1 \times L^2} \right) =0.
\end{align}
\end{prop}

The proof of Proposition~\ref{nonlin profile} is similar to the the proof of \cite[Proposition $2.8$]{DKM1} and we give a sketch of the argument below. In the current formulation, the argument is easier than the one given in \cite{DKM1} since here we make the simplifying assumption that all of the non-linear profiles exist globally and scatter. We also refer the reader to \cite[Proof of Proposition $3.1$]{LS} where the essential elements of the argument are carried out in an almost identical setting. 

The main ingredient in the proof of Proposition~\ref{nonlin profile} is the following non-linear perturbation lemma which we will also make use of later as well. For the proof of the perturbation lemma we refer the reader to \cite[Theorem $2.20$]{KM08}, and \cite[Lemma $3.3$]{LS}. In the latter reference a detailed proof in an almost identical setting is provided which can be applied verbatim here. 

\begin{lem}\cite[Theorem $2.20$]{KM08} \cite[Lemma $3.3$]{LS} \label{pert} There  are continuous functions $\e_0,  C_0: (0, \infty) \to (0, \infty)$ such that the following holds: Let $I\subset \R$ be an open interval, (possibly unbounded), $u, v \in C^0(I; \dot{H}^1(\R^4)) \cap C^1(I; L^2(\R^4)) $ radial functions satisfying for some $A>0$ 
\begin{align*} 
&\|\vec u\|_{L^{\I}(I; \dot{H}^1 \times L^2)}+ \|\vec v\|_{L^{\I}(I; \dot{H}^1 \times L^2)}+ \|v\|_{L^3_t(I; L^6_x)} \le A\\
&\|\textrm{eq}(u)\|_{L^1_t(I; L^2_x)}+\|\textrm{eq}(v)\|_{L^1_t(I; L^2_x)} + \|w_0\|_{L^3_t(I; L^6_x)} \le \e \le \e_0(A)
\end{align*} 
where $\textrm{eq}(u):= \Box u +u^3Z(ru)$ in the sense of distributions, and $\vec w_0(t):= S(t-t_0)(\vec u-\vec v)(t_0)$ with $t_0 \in I$ arbitrary, but fixed and $S$ denoting the free wave evolution operator in $\R^{1+4}$. Then,
\begin{align*} 
\|\vec u -\vec v - \vec w_0\|_{L^{\I}_t(I; \dot{H}^1 \times L^2)} + \|u-v\|_{L^3_tL^6_x} \le C_0(A) \e
\end{align*} 
In particular, $\|u\|_{L^3_t(I; L^6_x)} < \I$. 
\end{lem}

\begin{proof}[Proof of Proposition~\ref{nonlin profile}]
Set 
\begin{align*} 
v_n^k(t, r) = \sum_{j=1}^k \frac{1}{\la_n^j}V^j\left(\frac{t-t_n^j}{\la_n^j}, \frac{r}{\la_n^j}\right) 
\end{align*} 
We would like to apply Lemma~\ref{pert} to $u_n$ and $v_n^k$ for large $n$ and we need to check that the conditions of Lemma~\ref{pert} are satisfied for these choices. First note that $\textrm{eq}(u_n) = 0$. We claim that $\|\textrm{eq}(v_n^k)\|_{L^1_tL^2_x}$ is small for large $n$. To see this, observe that 
\begin{align*} 
\textrm{eq}(v_n^k) = \sum_{j=1}^k N\left( V^j_n(t, r)\right) - N\left( \sum_{j=1}^k V_n^j(t, r)\right)
\end{align*} 
where we have used the notation $V^j_n(t, r):= \frac{1}{\la_n^j}V^j\left(\frac{t-t_n^j}{\la_n^j}, \frac{r}{\la_n^j}\right) $ and $N(v) = v^3 Z(rv)$ as in \eqref{4d}. Using the simple inequality 
\begin{multline} \label{sin 2}
\abs{\frac{\sin(2ru) + \sin(2rv) - \sin(2r(u+v))}{2r^3}} \\= \abs{\frac{2 \sin(2ru) \sin^2(rv) + 2 \sin(2rv) \sin^2 (ru)}{2r^3}}  \lesssim u^2\abs{v} + v^2 \abs{u}
\end{multline} 
together with the pseudo-orthogonality of the times and scales in \eqref{po scales} and arguing as in the proof of Lemma~\ref{en orth lem} we obtain $\|\textrm{eq}(v_n^k)\|_{L^1_tL^2_x} \to 0$ as $n \to \infty$ for any fixed $k$. 
Next it is essential that 
\begin{align} \label{unif in k}
 \limsup_{n \to \infty} \left\|\sum_{j=1}^k V^j_n\right\|_{L^3_tL^6_x} \le A < \infty
 \end{align} 
 {\em uniformly in $k$}, which will follow from the small data theory together with \eqref{free en ort}. The point here is that the sum can be split into one over $1 \le j \le j_0$ and another over $j_0 \le j \le k$ . The splitting is performed in terms of the free energy, with $j_0$ being chosen so that 
 \begin{align*} 
  \limsup_{n \to \infty} \sum_{j_0 <j \le k} \|V^j_L(-t_n^j/ \la_n^j)\|_{\dot{H}^1 \times L^2}^2 < \de_0^2
  \end{align*}
  where $\de_0$ is chosen so that the small data theory applies. Using again \eqref{po scales} as well as the small data scattering theory one now obtains 
  \begin{align*} 
  \limsup_{n \to \infty}  \left\|\sum_{j_0 < j \le k}  V^j_n\right\|^3_{L^3_tL^6_x} &= \sum_{j_0 < j \le k}\|V^j\|_{L^3_tL^6_x}^3 \\
  & \le C \limsup_{n \to \infty} \left( \sum_{j_0 <j \le k} \|V^j_L(-t_n^j/ \la_n^j)\|_{\dot{H}^1 \times L^2}^2 \right)^{\frac{3}{2}}
  \end{align*}
 with an absolute constant $C$. This implies \eqref{unif in k}. Now the desired result follows directly from Lemma~\ref{pert}. 
\end{proof}

In Section \ref{sect bu} we will require a few additional results from \cite{CKS}. We restate these results here for completeness. First, we note that for a profile decomposition as in Theorem \ref{BaG}, the Pythagorean decompositions of the free energy remain valid even after a space localization. In particular we have the following: 
\begin{prop}\cite[Corollary $8$]{CKS} \label{c16}Consider a sequence of radial data $\vec u_n \in \dot{H}^1 \times L^2( \R^4)$ such that $ \|u_n\|_{\dot{H}^1 \times L^2} \le C$, and a profile decomposition of this sequence as in Theorem \ref{BaG}. Let $\{r_n\}\subset (0, \infty)$ be any sequence. Then we have 
\begin{align*}
 \| \vec u_n\|_{\dot{H}^1 \times L^2(r \ge r_n)}^2  = \sum_{1 \le j \le k} \| \vec V_L^j( - t_n^j/ \la_n^j) \|_{\dot{H}^1 \times L^2(r \ge r_n/ \la_n^j)}^2  +  \|\vec w_n^k\|_{\dot H^1\times L^2(r \ge r_n)}^2 + o_n(1)
\end{align*} 
as  $n \to \infty$.
\end{prop}

Next, we will need a fact about solutions to the free $4d$ radial wave equation that is also established in \cite{CKS}. The following result is the analog of \cite[Claim $2.11$]{DKM1} adapted to $\R^4$. In \cite{DKM1} it is proved in odd dimensions only.

\begin{lem}\cite[Lemma $11$]{CKS} \cite[Claim $2.11$]{DKM1} \label{local scat lem} Let $\vec w_n(0)= (w_{n,0}, w_{n,1})$ be a uniformly bounded sequence in $\dot{H}^1 \times L^2(\R^4)$ and let $\vec w_n(t)\in \dot{H}^1 \times L^2( \R^4)$ be the corresponding sequence of radial $4d$ free waves. Suppose that 
\begin{align*}
\|w_n \|_{L^3_tL^6_x} \to 0
\end{align*}
as $n \to \infty$. Let $\chi \in C^{\I}_0(\R^4)$ be radial so that $\chi \equiv 1$ on $\abs{x} \le 1$ and $\textrm{supp} \chi \subset \{ \abs{x} \le 2\}$. Let $\{\la_n\} \subset (0, \infty)$ and consider the truncated  data 
\begin{align*}
\vec v_n(0):=  \fy(r/ \la_n) \vec w_n(0),
\end{align*}
where either $\fy= \chi$ or $\fy= 1- \chi$. Let $\vec v_n(t)$ be the corresponding sequence of free waves. Then 
\begin{align*}
\|v_n \|_{L^3_tL^6_x} \to 0 \quad \textrm{as} \quad n \to \infty.
\end{align*}
\end{lem}
%Lemma~\ref{local scat lem} follows from the similar arguments as \cite[Proof of Claim $2.11$]{DKM1}. In \cite{DKM1} the above result is established in odd dimensions only.  This  result is  extended  to even dimensions, and in particular $d=4$, in \cite{CKS} and we refer the reader to \cite{CKS} for the proof. We also give a brief sketch in the appendix.

%%%%%%%%%%%%%%%--------Section 3---------%%%%%%%%%%%%%%%%%%%%%%%%%%%%

\section{Outline of the Proof of Theorem~\ref{main}}\label{km}

The proof of Theorem~\ref{main} follows from the concentration-compactness/rigidity method developed by the second author and Merle in \cite{KM06},  \cite{KM08}. This method provides a framework for establishing global existence and scattering results for a large class of nonlinear dispersive equations. We begin with a brief outline of the argument adapted to our current situation. For data $\vec \psi(0) \in \HH_0$ denote by $\vec \psi(t)$ the nonlinear evolution to \eqref{cp} associated to $\vec{\psi}(0)$. Define the set 
\begin{align}\label{scat set}
\cS:=\{\vec \psi(0) \in \HH_0 \, \vert\, \vec\psi(t)\, \,   \textrm{exists globally and scatters to zero as}\,\,  t \to \pm \infty\}
\end{align} 
Our goal is then to prove that 
\begin{align*}
\{\vec \psi(0) \in \HH_0 \, \vert\, \E(\vec \psi) < 2 \E(Q)\} \subset \cS
\end{align*}
This will be accomplished by establishing the following three steps. First, we recall the following global existence and scattering result proved in \cite{CKM}, for data in $\HH_0$ with energy $\le \E(Q)$. 
%---------------------------------------------------G.E. Scattering below E(Q)------------------------------------------------------%
\begin{thm}\cite[Theorem $1$ and Corollary $1$]{CKM} \label{CKM main}There exists a small $\de> 0$ with the following property. Let $\vec\psi(0)=(\psi_0, \psi_1) \in \HH_0$ be such that $\E(\vec\psi) < \E(Q)+ \de$. Then, there exists a unique global evolution $\vec\psi \in C^0(\R; \HH_0)$ to \eqref{cp} which scatters to zero in the sense of \eqref{scat}.
\end{thm} 
%-------------------------------------------------------------------------------------------------------------------------------------------%
This shows that $\cS $ is not empty. We remark that Theorem~\ref{CKM main} gives more than what is needed for the rest of the argument. A small data global existence and scattering result such as \cite[Theorem $2$]{CKM} would suffice to show that $\cS$ is not empty. In fact, the proof of Theorem~\ref{main}, and in particular Theorem $4.1$ provide an independent alternative to the proof of scattering below $\E(Q)+ \de$ given in \cite{CKM}.  

Next, we argue by contradiction. Assume that Thereom~\ref{main} fails and  suppose that $\E(Q)<\E^* < 2\E(Q)$ is the minimal energy level at which a failure to the conclusions of Theorem~\ref{main} occurs. We then combine the concentration compactness decomposition given in Corollary \ref{bg wm}, the nonlinear perturbation theory in Lemma~\ref{pert}, and the nonlinear profile decomposition in Proposition~\ref{nonlin profile} to extract a so-called critical element, i.e., a nonzero solution $\vec\psi_* \in C^0( I_{\max}(\vec \psi_*); \HH_0)$ to \eqref{cp} whose trajectory in $\HH_0$ is pre-compact up to certain time-dependent scaling factors arising due to the  scaling symmetry of the equation. Here $I_{\max}(\vec \psi)$ is the maximal interval of existence of $\vec\psi_*$. To be specific, we can deduce the following proposition:
%------------------------------------------Concentration Compactness-------------------------------------------------------%
\begin{prop}\cite[Proposition $2$ and Proposition $3$]{CKM} \label{conc comp}Suppose that Theorem~$1$ fails and let $\E^*$ be defined as above. Then, there exists a nonzero solution $\vec \psi_*(t) \in \HH_0$ to \eqref{cp}, (referred to as a the critical element), defined on its maximal interval of existence $I_{\max}(\vec \psi_*) \ni0$, with $$\E(\vec \psi_*)= \E^* <2\E(Q)$$ Moreover, there exists $A_0>0$, and a continuous function $\la: I_{\max} \to [A_0, \infty) $ such that the set 
\begin{align} \label{K} 
K:=\left\{ \left (\, \psi_*\left(t, \frac{r}{\la(t)} \right),\,  \frac{1}{\la(t)}\dot{\psi}_*\left( t,\frac{r}{ \la(t)}\right) \right)\,\Big{\vert} \, t\in I_{\max}\right\}
\end{align}
is pre-compact in $H \times L^2$. 
\end{prop}
%-------------------------------------------------------------------------------------------------------------------------------------------%
\begin{rem} 
As noted above, the Cauchy problem \eqref{cp}, for data $\vec\psi(0) \in V(\al)$ with $\al\le 2\E(Q)$ is equivalent to  the  Cauchy problem for the $4d$ nonlinear radial wave equation, \eqref{4d}, via the identification $ru = \psi$. Hence, it suffices to carry out the small data global existence and scattering argument, as well as the concentration compactness decomposition and the extraction of a critical element on the the level of the $4d$ equation \eqref{4d} for $u$. We remark that in this setting, scattering in the sense of \eqref{scat} is equivalent to $\|u\|_{\mathcal{X}(\R^{1+4})}<\infty$ where $\mathcal{X}$ is a suitably chosen Strichartz norm. For example, $\mathcal{X} = L^3_tL^6_x$ will do. 
\end{rem} 
 \begin{rem} 
 In the proof of Theorem~\ref{main}, the requirement that $\E(\vec \psi(0))< 2\E(Q)$ arises in the concentration compactness procedure.  Indeed, in order to ensure that the critical element  $\vec\psi_*$ described in Proposition~\ref{conc comp} lies in $\HH_0$ one needs to require that any sequence of data $\{\vec{\psi}_n(0)\}$ with energies converging from below to the minimal energy level $\E_*$, also have uniformly bounded $H\times L^2$ norms. This is only guaranteed when $\E_*<2\E(Q)$ by Lemma~\ref{ckm lem2}.  In this case, one obtains a sequence of data $\vec u_n(0)$, via the identification $ru_n =\psi_n$, that is uniformly bounded in $\dot H^1 \times L^2(\R^4)$ and on which one is free to perform the concentration compactness decomposition as in \cite{BG} and extract a critical element $\vec u_*$ as in  \cite{KM08}, \cite{CKM}. We can then define $\vec \psi_* := r\vec u_*$.
 \end{rem} 
\begin{rem} 
For the proof that the function $\la(t)$ described in Proposition~\ref{conc comp} can be taken to be continuous, we refer the reader to \cite[Lemma $4.6$]{KM08} and \cite[Remark $5.4$]{KM06}. The fact that we can assume that $\la$ is bounded from below follows verbatim from  the arguments given in \cite[Section $6$, Step $3$]{DKM1}. See also, \cite[Proof of Theorem $7.1$]{KM08} and \cite[Proof of Theorem $5.1$]{KM06}. 
\end{rem}

The final step, referred to as the rigidity argument, consists of showing any solution $\vec \psi(t) \in \HH_0$ with the aforementioned compactness properties must be identically zero, which provides the contradiction. This part of the concentration compactness/rigidity method is what allows us to extend the result in \cite{CKM} to all energies below $2\E(Q)$ and we will thus carry out the proof in detail in the next section.

\subsection{Sharpness of Theorem~\ref{main} in $\HH_0$}\label{sharp} Before we begin the rigidity argument, we first show that Theorem~\ref{main} is indeed sharp in $\HH_0$ by demonstrating the following claim: for all $\de>0$ there exist data $\vec\psi(0) \in \HH_0$ with $\E(\psi)\le 2\E(Q) + \de$, such that the corresponding wave map evolution, $\psi(t)$, blows up in finite time. This follows easily from the blow-up constructions of \cite{KST} or \cite{RR}. 
 
Fix $\de_0 >0$.  By \cite{KST} or \cite{RR} we can choose data $\vec u(0) \in \HH_1$ such that 
 \begin{align*} 
 \E(\vec u(0)) \le \E(Q) + \de, \quad \de \ll \de_0
 \end{align*}
such that the corresponding wave map evolution $\vec u(t) \in \HH_1$ blows up at time $t~=~1$. In other words, the energy of $\vec u(t)$ concentrates in the backwards light cone, $K(1,0):=\{(t, r) \in[0,1]\times [0,1] \,\,  \vert\,\,  r\le 1-t\}$, emanating from the point $(1, 0) \in \R\times [0, \infty]$, i.e., 
\begin{align*} 
\lim_{t \nearrow 1}\E_0^{1-t}(\vec u(t)) \ge \E(Q) 
\end{align*}
where $\E_a^b(u, v) = \int_a^b (u_r^2 + v^2 + \frac{\sin^2(u)}{r^2})\, r \, dr$.  Now define $\vec \psi(0) \in  \HH_0$ as follows: 
\begin{align} 
\psi(0,r) = \begin{cases} u(0, r) \quad \, \textrm{if} \quad r \le 2\\ \pi -Q(\la r) \quad \textrm{if}\quad r\ge2\end{cases}
\end{align}
where $\la>0$ is chosen so that $\pi - Q(2\la ) = u(0, 2)$. We note that the existence of such a $\la$ follows form the fact that we can ensure that $u(0, r)>0 $ for $r>1$. To see this, observe that since $\vec u(t)$ blows up at time $t=1$ and thus must concentrate at least $\E(Q)$ inside the light cone we can deduce by the monotonicity of the energy that $\E_{0}^1( \vec u(0)) \ge \E(Q)$. Now choose  $\de<\E(Q)$. If we have $u(0, r) \le 0$ for any $r>1$ we would need at least $\E_r^{\infty}(u(0), 0) \ge \E(Q)$ to ensure that $u(0, \infty)=  \pi$. This follows from the minimality of $\E(Q)$ in $\HH_1$. However $\E_{r}^{\infty}(u(0), 0) \le \de < \E(Q)$. 

Now observe that
\begin{align} 
\E(\vec \psi(0)) = \E_0^2(\vec u(0)) + \E_2^{\infty}(\pi- Q) \le \E(\vec u) + \E(Q) \le 2\E(Q) + \de.
\end{align}
Let $\vec\psi(t)$ denote the wave map evolution of the data $\vec \psi(0)$. By the finite speed of propagation, we have that $\vec\psi(t, r) = \vec u(t, r)$ for all $(t, r) \in K(0,1)$ and hence 
\begin{align}
 \lim_{t \nearrow 1} \E_0^{1-t}(\vec \psi(t)) = \lim_{t \nearrow 1}\E_0^{1-t}(\vec u(t)) \ge \E(Q)
 \end{align}
 which means that $\vec \psi(t)$ blows up at $t=1$ as desired. Note that if one wishes to construct blow-up data in $ \HH_0$ that maintains the smoothness of $ u(0)$, one can simply smooth out $\vec \psi(0, r)$ in a small neighborhood of the point $r=2$ using an arbitrarily small amount of energy. 
 
We again remark that the questions of determining  the possible dynamics at the threshold, $\E(\vec \psi) = 2 \E(Q)$, and above it, $\E(\vec \psi)>2\E(Q)$, are not addressed here and remain open.

%-----------------------------------------------------rigidity section----------------------------------------------------------------%

\section{Rigidity} \label{rigid}

 In this section we prove Theorem~\ref{rigidity} and complete the proof of Theorem~\ref{main}. We begin by establishing a rigidity theory in $\HH_0$ which will allow us to deduce Theorem~\ref{main}.  We then use the conclusions of Theorem~\ref{main} together with the proof of Theorem~\ref{rigidity2} to establish Theorem~\ref{rigidity}.

%-------------------------------------------------Rigidity Thm in $\HH_0$-------------------------------------------------------------% 
 \begin{thm}[Rigidity in $\HH_0$]\label{rigidity2} Let $\vec \psi(t) \in \HH_0$ be a solution to \eqref{cp} and let $I_{\max}(\psi)= (T_-(\psi), T_+(\psi))$ be the maximal interval of existence. Suppose that there exist $A_0>0$ and a continuous function $\la : I_{\max} \to [A_0, \infty)$ such that the set 
 \begin{align} \label{K2}
  K:= \left\{\,\left (\, \psi\left(t, \frac{r}{\la(t)} \right),\,  \frac{1}{\la(t)}\dot{\psi}\left( t,\frac{r}{ \la(t)}\right) \right)\, \Big{\vert}\, t \in I_{max} \right\}
 \end{align}
 is pre-compact in $H \times L^2$. Then, $I_{\max}= \R$ and $\psi\equiv 0$.
 \end{thm}
 %-------------------------------------------------Rigidity Thm in $\HH_0$-------------------------------------------------------------%

We begin by recalling the following virial identity:
\begin{lem}\label{vir} Let $\chi_R(r) =\chi(r/R)\in C_0^{\infty}(\R)$ satisfy $\chi(r)= 1$ on $[-1, 1]$ with $\supp(\chi) \subset [-2, 2]$. Suppose that $\vec \psi$ is a solution to \eqref{cp} on some interval $I \ni 0$. Then, for all $T \in I$ we have 
\begin{align}\label{virial}
 \ang{\chi_R\dot \psi \, \vert \, r \psi_r}\Big{|}_0^T = -\int_0^T \int_0^{\infty} \dot{\psi}^2 \, r\, dr\, dt + \int_0^T O( \E_R^{\infty} (\vec\psi(t)))\, dt.
\end{align}
\end{lem}
 \begin{proof} Since $\vec\psi$ is a solution to \eqref{cp} we have 
 \begin{align*} 
 \frac{d}{dt} \ang{\chi_R\dot \psi \, \vert \, r \psi_r} &= \ang{\chi_R\ddot \psi \, \vert \, r \psi_r} + \ang{\chi_R\dot \psi \, \vert \, r \dot\psi_r} \\
 &= \ang{\chi_R(\psi_{rr} + \frac{1}{r} \psi_r - \frac{\sin(2\psi)}{2r^2})  \, \Big{|} \, r \psi_r} + \ang{\chi_R\dot \psi \, \vert \, r \dot\psi_r} \\
 %&=\int_0^{\infty} \frac{1}{2} \p_r(\psi^2_r) \chi_R \, r^2\, dr + \int_0^{\infty} \psi^2_r \chi_R \, r\, dr\\
% & \quad - \int_0^{\I} \frac{1}{2} \p_r( \sin^2(\psi)) \chi_R \, dr + \int_0^{\I} \frac{1}{2} \p_r(\dot{\psi}^2) \chi_R \, r^2 \, dr\\
 &= - \int_0^{\infty} \dot\psi^2 r\, dr + \int_0^{\I}(1- \chi_R) \dot{\psi}^2 \, r\, dr\\
 &\quad -\frac{1}{2} \int_0^{\I} \left( \dot \psi^2 + \psi^2_r - \frac{\sin^2(\psi)}{r^2}\right) \chi_R^{\prime} \, r^2\, dr.
\end{align*}
 Observe that 
 \begin{align*} 
 \abs{\int_0^{\I}(1- \chi_R) \dot{\psi}^2 \, r\, dr} \lesssim \E_R^{\infty}(\vec \psi).
 \end{align*} 
 Finally, noting that $\chi_R^{\prime}(r) = \frac{1}{R} \chi^{\prime}(r/R)$, we obtain
 \begin{multline*}
 \abs{\int_0^{\I} \frac{1}{2} \left( \dot \psi^2 + \psi^2_r - \frac{\sin^2(\psi)}{r^2}\right) \chi_R^{\prime} \, r^2\, dr} \\\lesssim \int_R^{2R} \left( \dot \psi^2 + \psi^2_r +\frac{\sin^2(\psi)}{r^2}\right) \frac{r}{R} \chi'\left(\frac{r}{R}\right)r\,\,dr
 \lesssim \E_R^{\I}(\vec \psi).
 \end{multline*}
Hence we can conclude that 
\begin{align*}
\frac{d}{dt} \ang{\chi_R\dot \psi \, \vert \, r \psi_r} = - \int_0^{\infty} \dot{\psi}^2 \, r\, dr\,  + O( \E_R^{\infty} (\vec\psi(t))).
\end{align*}
An integration from $0$ to $T$ proves the lemma. 
\end{proof}

With the virial identity \eqref{virial}, we can begin the proof of Theorem~\ref{rigidity2}. This will be done in several steps and is inspired by the arguments in \cite[Proof of Theorem $2$]{DKM1}. To begin, we recall from \cite{KM08} that any wave map with a pre-compact trajectory in $H \times L^2$ as in \eqref{K2} that blows up in finite time is supported on the backwards light cone. 

\begin{lem}\cite[Lemma $4.7$ and Lemma $4.8$]{KM08}\label{comp supp} Let $\vec{\psi}(t) \in \HH_0$ be a solution to \eqref{cp} such that $I_{\max}(\vec \psi)$ is a finite interval. Without loss of generality we can assume $T_+(\vec \psi) =1$. Suppose there exists a continuous function $\la : I_{\max} \to (0, \infty)$ so that $ K$, as defined in \eqref{K2}, is pre-compact in $H \times L^2$. Then 
\begin{align}\label{la lower}
0< \frac{C_0(K)}{1-t} \le \la(t).
\end{align}
And, for every $t \in [0, 1)$ we have 
\begin{align}\label{supp} 
\supp (\vec \psi(t)) \in [0, 1-t).
\end{align}
\end{lem} 

We can now begin the proof of Theorem~\ref{rigidity2}.
\begin{proof}[Proof of Theorem~\ref{rigidity2}]  
\hspace{1in}\begin{flushleft}\emph{Step $1$}: \hspace{.5in}\end{flushleft}
First we show that $I_{\max}(\psi)= \R$. Assume that $T_+(\vec \psi) < \infty$ and we proceed by contradiction. Without loss of generality, we may assume that $T_+(\vec \psi) =1$. By Lemma~\ref{comp supp}, we can deduce  that $0< \frac{C_0(K)}{1-t} \le \la(t)$ and $\supp(\vec \psi(t)) \in [0, 1-t)$. 
In addition, we know, by  Lemma~\ref{ext en decay}, (see \cite{St} or the argument in \cite[Lemma $2.2$]{STZ}), that self similar blow-up for $2d$ wave maps is ruled out; note that the argument in~\cite{STZ} applies  only to smooth wave maps. See Appendix~\ref{a:STZ} for how to extend this result to solutions in~$\HH_0$. This implies that there exists a sequence $\{\tau_n\} \subset (0,1)$ with  $\tau_n \to 1$ such that 
\begin{align} \label{taun}
\frac{1}{ \la(\tau_n)(1-\tau_n)} < 1 \quad \textrm{as} \quad  n \to \infty.
\end{align}

Hence, we can extract a further subsequence $\{t_n\}\to 1$  and apply Corollary~\ref{t dec extract} with $\sigma= \frac{1}{\la(t_n)}$ to obtain,  for every $n$, the bound
\begin{align}\label{1n}
\la(t_n) \int_{t_n}^{t_n + \frac{1}{\la(t_n)}}  \int_0^{\infty} \dot \psi^2(t, r) \, r\, dr\, dt \le \frac{1}{n}.
\end{align}
Note that above we have used the fact that $\textrm{supp} (\vec \psi(t)) \in [0, 1-t)$. Next, with $t_n$ as above, define a sequence in $\HH_0$ by setting 
\begin{align*} 
\vec \psi_n(0) = (\psi_n^0, \psi_n^1):=\left(\psi\left(t_n, \frac{r}{\la(t_n)}\right), \frac{1}{\la(t_n)} \dot\psi\left(t_n, \frac{r}{\la(t_n)}\right)\right).
\end{align*}
The nonlinear evolutions associated to our sequence 
\begin{align*}
\vec \psi_n(t):=\left(\psi\left(t_n+\frac{t}{\la(t_n)}, \frac{r}{\la(t_n)}\right), \frac{1}{\la(t_n)} \dot\psi\left(t_n+ \frac{t}{\la(t_n)}, \frac{r}{\la(t_n)}\right)\right)
\end{align*}
are then solutions to \eqref{cp} with $\E(\vec{\psi}_n) = \E(\vec{\psi})$. Observe  that 
\begin{align} \label{t lim1} 
\int_0^{1} \int_0^{\infty} \dot\psi_n^2(t, r) \, r\, dr\, dt \to 0 \quad \textrm{as} \quad n \to \infty.
\end{align}
Indeed, by \eqref{1n} we have that 
\begin{align*}
\int_0^{1} \int_0^{\infty} \dot\psi_n^2 \, r\, dr\, dt = \la(t_n) \int_{t_n}^{t_n + \frac{1}{\la(t_n)}}  \int_0^{\I} \dot{\psi}^2(t, r) \, r \, dr  \, dt \to 0 \quad \textrm{as} \quad n \to \I.
\end{align*}
We now proceed as follows. By the compactness of $K$ we can find $\vec \psi_{\infty}(0) = ( \psi_{\infty}^0, \psi_{\infty}^1) \in \HH_0$ and a subsequence of $\{\vec{\psi_n}(0)\}$ such that we have strong convergence
\begin{align}
\vec \psi_n(0) \to \vec \psi_{\infty}(0) \quad \textrm{as} \quad n \to \I
\end{align}
in $H \times L^2$. Note that this also implies strong convergence in the energy topology, i.e.,  $\vec{\psi}_n(0) \to \vec{\psi}_{\I}(0)$ in $\HH_0$. In particular, we have 
\begin{align}\label{1E=}
\E(\vec\psi_{\infty}(0)) = \E(\vec \psi_n(0)) = \E(\vec{\psi}).
\end{align}
Now,  let $\vec\psi_{\infty}(t)\in \HH_0$ denote the forward solution to \eqref{cp} with initial data $\vec\psi_{\infty}(0)$ on its maximal interval of existence $[0, T_+(\psi_{\infty}))$. Choose $T_0 \in(0, T_+(\psi_{\I}))$ with $T_0 \le 1$. 

Using Lemma~\ref{pert} for the equivalent $4$-dimensional wave equation \eqref{4d}, the strong convergence of $\vec{\psi}_n(0)$ to $\vec \psi_{\infty}(0)$ in $H \times L^2$  implies that for large $n$, 
the nonlinear evolutions $\vec \psi_n(t)$ and $\vec\psi_{\infty}(t)$ remain uniformly close in $H\times L^2$ for $t\in [0, T_0]$. Indeed, we have  
\begin{align}\label{pert concl1}
\sup_{t \in [0, T_0]}\|\vec \psi_n(t) - \vec \psi_{\I}(t) \|_{H\times L^2}  = o_n(1).
\end{align}
Hence, combining \eqref{t lim1} with \eqref{pert concl1} we have 
\begin{align*}
0 \gets \int_0^{1} \int_0^{\infty} \dot\psi_n^2(t, r) \, r\, dr\, dt  &\ge \int_0^{T_0} \int_0^{\infty} \dot\psi_n^2(t, r) \, r\, dr\, dt\\
& = \int_0^{T_0} \int_0^{\infty} \dot\psi_{\infty}^2(t, r) \, r\, dr\, dt + o_n(1).
\end{align*}
Therefore we have $\dot\psi_{\infty} \equiv 0$ on $[0, T_0]$. Since $\psi =0$ is the unique harmonic map in  $\HH_0$ we necessarily have that $\psi_{\infty} \equiv 0$. But, by \eqref{1E=} we then have $0=\E(\vec\psi_{\infty}) = \E(\vec \psi_n) = \E( \vec \psi)$. Hence $\vec \psi \equiv 0$, which contradicts our assumption that $\psi \neq 0$ blows up at time $t=1$.

\begin{flushleft}{\em Step $2$}: By Step $1$, we have reduced the proof of Theorem~\ref{rigidity2} to the case $I_{\max} =\R$, and hence $\la: \R \to [A_0, \infty)$. By time symmetry we can, without loss of generality, work with nonnegative times only and thus consider  $\la(t):[0, \infty) \to [A_0, \infty)$. \end{flushleft}

First note that since $ K$ is pre-compact in $H \times L^2$  and since $\la(t) \ge A_0$ we have that for all $\e >0$ there exists an $R= R(\e)$ such that for every $t\in [0, \infty)$
\begin{align}\label{tail2} 
\E_{R(\e)}^{\I}(\vec \psi(t)) < \e.
\end{align}
Also, observe that for all $T>0$ we have 
\begin{align} \label{lhs2}
\abs{ \ang{\chi_R\dot \psi \, \vert \, r \psi_r}\Big{|}_0^T }  \lesssim R\E(\vec{\psi}).
\end{align} 
Now, fix $\e>0$ and fix $R$ large enough so that $\sup_{t\ge0}\E_{R}^{\I}(\vec\psi) < \e$. Then, Lemma~\ref{vir} together with \eqref{lhs2} implies that for all $T \in [0, \infty)$ we have 
\begin{align*} 
\frac{1}{T}\int_0^T \int_0^{\infty} \dot{\psi}^2 \, r\, dr\, dt \lesssim \frac{R}{T}\E(\vec \psi) + \e.
\end{align*}
 This shows that 
 \begin{align} \label{cesaro2} 
 \frac{1}{T}\int_0^T \int_0^{\infty} \dot{\psi}^2 \, r\, dr\, dt  \to 0 \quad \textrm{as} \quad T \to \I.
 \end{align}
 Next, we claim that there exists a sequence $\{t_n\}$ with $t_n \to \I$ such that 
 \begin{align} \label{ces2} 
  \lim_{n\to \infty} \la(t_n) \int_{t_n}^{t_n + \frac{1}{\la(t_n)}} \left( \int_0^{\I} \dot{\psi}^2 \, r \, dr \right) \, dt=0.
  \end{align} 
To see this, we begin by defining a sequence $\tau_n$ as follows. Set 
\begin{align*} 
\tau_0 =0,  \quad \tau_{n+1}:= \tau_n + \frac{1}{\la(\tau_n)}= \sum_{k=0}^n \frac{1}{\la(\tau_k)}.
\end{align*}
First we establish that $\tau_n \to \infty$ as $n \to \infty$. If not, then up to a subsequence we would have $\tau_n \to \tau_{\infty} < \I$. This would imply that 
\begin{align*}
\tau_{\infty} = \sum_{k=0}^{\I} \frac{1}{\la(\tau_k)} < \infty
\end{align*}
which means that $\ds{\lim_{k \to \I} \frac{1}{\la(\tau_k)} =0}$. But this is impossible since $\la(\tau_k) \to \la(\tau_{\infty})< \infty$ by the continuity of $\la$. 

Now, suppose that \eqref{ces2} fails for all subsequences $\{t_n\} \subset \{\tau_n\}$. Then there exists $\e >0$ such that for all $k$,
\begin{align*} 
\int_{\tau_k}^{\tau_{k+1}} \left( \int_0^{\I} \dot{\psi}^2 \, r \, dr \right) \, dt \ge \e \frac{1}{\la(\tau_k)}.
\end{align*} 
Summing both sides above from $1$ to $n$ gives
\begin{align*} 
\int_0^{\tau_{n+1}} \left( \int_0^{\I} \dot{\psi}^2 \, r \, dr \right) \, dt \ge \e \sum_{k=1}^n  \frac{1}{\la(\tau_k)} = \e \tau_{n+1}
\end{align*}
which contradicts \eqref{cesaro2}. Hence there exists a sequence $\{t_n\}$ such that \eqref{ces2} holds. Moreover, since $\la(t) \ge A_0>0$ for all $t \ge 0$ we can extract a further subsequence, still denoted  by $\{t_n\}$, such that \eqref{ces2} holds and all the intervals $[t_n, t_n + \frac{1}{\la(t_n)}]$ are disjoint. 

Next, with $t_n$ as above, define a sequence in $\HH_0$ by setting 
\begin{align*} 
\vec \psi_n(0) = (\psi_n^0, \psi_n^1):=\left(\psi\left(t_n, \frac{r}{\la(t_n)}\right), \frac{1}{\la(t_n)} \dot\psi\left(t_n, \frac{r}{\la(t_n)}\right)\right).
\end{align*}
The nonlinear evolutions associated to our sequence 
\begin{align*}
\vec \psi_n(t):=\left(\psi\left(t_n+\frac{t}{\la(t_n)}, \frac{r}{\la(t_n)}\right), \frac{1}{\la(t_n)} \dot\psi\left(t_n+ \frac{t}{\la(t_n)}, \frac{r}{\la(t_n)}\right)\right)
\end{align*}
are then global solutions to \eqref{cp} with $\E(\vec{\psi}_n) = \E(\vec{\psi})$. Observe  that 
\begin{align} \label{t lim2} 
\int_0^{1} \int_0^{\infty} \dot\psi_n^2(t, r) \, r\, dr\, dt \to 0 \quad \textrm{as} \quad n \to \infty.
\end{align}
Indeed, by \eqref{ces2} we have that 
\begin{align*}
\int_0^{1} \int_0^{\infty} \dot\psi_n^2 \, r\, dr\, dt = \la(t_n) \int_{t_n}^{t_n + \frac{1}{\la(t_n)}} \left( \int_0^{\I} \dot{\psi}^2 \, r \, dr \right) \, dt \to 0 \quad \textrm{as} \quad n \to \I.
\end{align*}
We now proceed as follows. By the pre-compactness of $ K$ we can find $\vec \psi_{\infty}(0) = ( \psi_{\infty}^0, \psi_{\infty}^1) \in \HH_0$ and a subsequence of $\{\vec{\psi_n}(0)\}$ such that we have strong convergence
\begin{align}
\vec \psi_n(0) \to \vec \psi_{\infty}(0) \quad \textrm{as} \quad n \to \I
\end{align}
in $H \times L^2$. Note that this also implies strong convergence in the energy topology, i.e.,  $\vec{\psi}_n(0) \to \vec{\psi}_{\I}(0)$ in $\HH_0$. In particular, we have 
\begin{align}\label{E=}
\E(\vec\psi_{\infty}(0)) = \E(\vec \psi_n(0)) = \E(\vec{\psi}).
\end{align}
 
 Now,  let $\vec\psi_{\infty}(t)\in \HH_0$ denote the forward solution to \eqref{cp} with initial data $\vec\psi_{\infty}(0)$ on its maximal interval of existence $[0, T_+(\psi_{\infty}))$. Choose $T_0 \in(0, T_+(\psi_{\I}))$ with $T_0 \le 1$. 

Using Lemma~\ref{pert} for the $4$-dimensional wave equation \eqref{4d}, the strong convergence of $\vec{\psi}_n(0)$ to $\vec \psi_{\infty}(0)$ in $H \times L^2$  implies that for large $n$ 
the nonlinear evolutions $\vec \psi_n(t)$ and $\vec\psi_{\infty}(t)$ remain uniformly close in $H\times L^2$ in $t\in [0, T_0]$. Indeed, we have  
\begin{align}\label{pert concl}
\sup_{t \in [0, T_0]}\|\vec \psi_n(t) - \vec \psi_{\I}(t) \|_{H\times L^2}  = o_n(1).
\end{align}
Hence, combining \eqref{t lim2} with \eqref{pert concl} we have 
\begin{align*}
0 \gets \int_0^{1} \int_0^{\infty} \dot\psi_n^2(t, r) \, r\, dr\, dt  &\ge \int_0^{T_0} \int_0^{\infty} \dot\psi_n^2(t, r) \, r\, dr\, dt\\
& = \int_0^{T_0} \int_0^{\infty} \dot\psi_{\infty}^2(t, r) \, r\, dr\, dt + o_n(1).
\end{align*}
Therefore we have $\dot\psi_{\infty} \equiv 0$ on $[0, T_0]$. Since $\psi =0$ is the unique harmonic map in  $\HH_0$ we necessarily have that $\psi_{\infty} \equiv 0$. But, by \eqref{E=} we then have $0=\E(\psi_{\infty}, 0) = \E(\vec \psi_n) = \E( \vec \psi)$. Hence $\vec \psi \equiv 0$ as desired. 
\end{proof}

 We can now complete the proof of Theorem~\ref{main}. 
 \begin{proof}[Proof of Theorem~\ref{main}] Suppose that Theorem~\ref{main} fails. Then by Proposition~\ref{conc comp} there would exist a nonzero critical element $\vec \psi_*$ that satisfies the assumptions of Theorem~\ref{rigidity2}. But by Theorem~\ref{rigidity2}, $\vec\psi_*\equiv 0$, which is a contradiction. 
 \end{proof} 
 
 To conclude, we prove Theorem~\ref{rigidity}. 
 \begin{proof}[Proof of Theorem~\ref{rigidity}] \hspace{1in}\begin{flushleft}\emph{Step $1$}: First we show that $I_{\max}(\vec U)= \R$. We argue by contradiction. Assume that $T_+(\vec U) < \infty$. Without loss of generality, we may assume that $T_+(\vec U) =1$. \end{flushleft}
 Applying the exact same argument as in Step $1$ of the proof of Theorem~\ref{rigidity2} up to \eqref{t lim1} we can construct a sequence of solutions $\vec U_n(t)\in\dot{H}^1 \times L^2(\R^2; S^2)$ to \eqref{cp} such that 
\begin{align*} 
\vec U_n(0) = (U_n^0, U_n^1):=\left(U\left(t_n, \frac{r}{\la(t_n)}, \om\right), \frac{1}{\la(t_n)} \p_tU\left(t_n, \frac{r}{\la(t_n)}, \om\right)\right)
\end{align*}
with $\E(\vec U_n) = \E(\vec U)$ and 
\begin{align} \label{U t lim} 
\int_0^{1} \int_{\R^2} \p_t U_n^2(t) \, dx\, dt \to 0 \quad \textrm{as} \quad n \to \infty.
\end{align}
From this we obtain the following conclusions:
\begin{itemize} 
\item[(i)]Extracting a subsequence we have $U_n \rightharpoonup U_{\infty}$ weakly in $\dot{H}^1_{\textrm{loc}}([0,1]\times \R^2; S^2)$ and hence $\vec{U}_{\infty}(t)$ is a weak solution to \eqref{cp} on $[0, 1]$. 
\item[(ii)] By the pre-compactness of $\ti K$ we can, in fact, ensure that  $\vec U_n(0) \to \vec{U}_{\infty}(0)$ strongly in $\dot{H}^1\times L^2(\R^2; S^2)$. This implies that \begin{align}\label{EU}
\E(\vec{U}_{\infty}) = \E(\vec U_n) = \E(\vec U)
\end{align}
\item[(iii)] By \eqref{U t lim} we can deduce that $\dot{U}_{\infty} \equiv 0$ on $[0,1]$. 
\end{itemize} 

Putting this all together, we have a  time independent weak solution $\vec U_{\infty} \in \HH $ to \eqref{cp} for $t \in[0,1]$. By H\'{e}lein's Theorem \cite[Theorem $1$]{Hel} we know that  $U_{\infty}$ is, in fact, harmonic. Since $U=0$ and $U=(\pm Q, \om)$ are the unique harmonic maps up to scaling in  $\HH$ we necessarily have that either  $U_{\infty}=0$ or $\vec U_{\infty}(r, \om) =(Q(\ti \la \cdot), \om)$ for some $\ti \la>0$.  Hence, by \eqref{EU}, we can deduce that either $\E(\vec U) = 0$ or $\E(\vec U) = \E(Q, 0)$. The former case implies that $U \equiv 0$. If the latter case occurs, then $U(t)$ can either be an element of $\HH_0$,   $\HH_1$, or of $\HH_{-1}$ since all the higher topological classes, $\HH_n$ for $\abs{n}>1$, require more energy. If $U(t) \in \HH_0$  then it is global in time and scatters by Theorem~\ref{main}. If $U(t) \in\HH_1$ or $\HH_{-1}$ then we have $U(t, r, \om) = (\pm Q(\ti \la r),  \om)$ for some $\ti \la>0$ since $(Q, 0)$, respectively $(-Q, 0)$,  uniquely minimizes the the energy in $\HH_1$, respectively $\HH_{-1}$. In either case, this provides a contradiction to our assumption that $I_{\max} \neq \R$. 

  \begin{flushleft}\emph{Step $2$}: 
   \end{flushleft}
 
 Again we apply the exact same argument given in Step $2$ of the proof of Theorem~\ref{rigidity2} and we  construct a sequence of solutions $\vec U_n(t)\in\dot{H}^1 \times L^2(\R^2; S^2)$ to \eqref{cp} such that 
\begin{align*} 
\vec U_n(0) = (U_n^0, U_n^1):=\left(U\left(t_n, \frac{r}{\la(t_n)}, \om\right), \frac{1}{\la(t_n)} \p_tU\left(t_n, \frac{r}{\la(t_n)}, \om\right)\right)
\end{align*}
with $\E(\vec U_n) = \E(\vec U)$ and 
\begin{align} \label{U t lim2} 
\int_0^{1} \int_{\R^2} \p_t U_n^2(t) \, dx\, dt \to 0 \quad \textrm{as} \quad n \to \infty.
\end{align}
We thus obtain the following conclusions:
\begin{itemize} 
\item[(i)]Extracting a subsequence we have $U_n \rightharpoonup U_{\infty}$ weakly in $\dot{H}^1_{\textrm{loc}}([0,1]\times \R^2; S^2)$ and hence $\vec{U}_{\infty}(t)$ is a weak solution to \eqref{cp} on $[0, 1]$. 
\item[(ii)] By the pre-compactness of $\ti K$ we can extract a further subsequence with $\vec U_n(0) \to \vec{U}_{\infty}(0)$ strongly in $\dot{H}^1\times L^2(\R^2; S^2)$. This implies that \begin{align}\label{EU2}
\E(\vec{U}_{\infty}) = \E(\vec U_n) = \E(\vec U)
\end{align}
\item[(iii)] By \eqref{U t lim2} we can deduce that $\dot{U}_{\infty} \equiv 0$ on $[0,1]$. 
\end{itemize} 

Putting this all together, we have a  time independent weak solution $\vec U_{\infty} \in \HH $ to \eqref{cp} for $t \in[0,1]$. By H\'{e}lein's Theorem \cite[Theorem $1$]{Hel} we know that $U_{\infty}$ is, in fact, harmonic. Since $U=0$ and $U=( \pm Q, \om)$ are the unique harmonic maps up to scaling in  $\HH$ we necessarily have that either  $U_{\infty}=0$ or $\vec U_{\infty}(r, \om) =(\pm Q(\ti \la \cdot), \om)$ for some $\ti \la>0$.  Hence by \eqref{EU2} we can deduce that either $\E(\vec U) = 0$ or $\E(\vec U) = \E(Q, 0)$. The former case implies that $U \equiv 0$. Arguing as in the conclusion to Step $1$, the latter case implies that either $U(t) \in \HH_0$ or $U(t) \in \HH_{\pm 1}$. If $U(t) \in \HH_{\pm1}$, then $U(t, r, \om) = (\pm Q(\ti \la r),  \om)$ for some $\ti \la>0$. If $\vec U(t) \in \HH_0$ with $\E(\vec U) = \E(Q)$,  then  Theorem~\ref{main}  shows that $ \vec U(t)$ is global in time and scatters to $0$ as $t \to \infty$ in $\dot{H}^1 \times L^2(\R^2; S^2)$ in the sense that the energy of $\vec U(t)$ goes to $0$ as $t \to \infty$ on any fixed but compact set $V \subset \R^2$. Finally, we observe that the pre-compactness of $\ti K$ renders such a scattering result impossible. 

We thus conclude that either $U \equiv 0$ or $U(t, r, \om) = ( \pm Q(\ti \la r), \om)$ for some $\ti \la >0$ proving Theorem~\ref{rigidity}.
 \end{proof}

%-------------------------------------------------------Blow up Section-----------------------------------------------------------------%

 \section{Universality of the blow-up profile for degree one wave maps with energy below $3 \E(Q)$}\label{sect bu}

In this section we prove Theorem~\ref{bu cont}. We start  by first  deducing the conclusions of   Theorem~\ref{bu cont} along a sequence of times. To be specific, we establish the following proposition: 

\begin{prop}\label{bu}  Let $\vec \psi(t) \in \HH_1$  be a solution to  \eqref{cp} blowing up at time $t=1$ with  $$\E(\vec \psi) = \E(Q) + \eta < 3\E(Q) $$ Then there exists a sequence of times $t_n \to 1$, a sequence of scales $\la_n = o(1-t_n)$, a map $\vec \fy=(\fy_0, \fy_1) \in \HH_0$, and a decomposition  
\begin{align}\label{dec}
(\psi(t_n),  \dot\psi(t_n)) =   ( \fy_0, \fy_1) + \left(Q\left(\frac{\cdot}{\la_n}\right), 0\right) + \vec\e(t_n)
\end{align}
such that $\vec \e(t_n) \in \HH_0$ and $\vec \e(t_n)\to 0$ in $H \times L^2$ as $n \to \infty$. 
\end{prop}

Most of  this section will be devoted to the proof of Proposition~\ref{bu}. We will proceed in several steps, the first being the extraction of the radiation term.

\subsection{Extraction of the radiation term}

In this subsection we construct what we will call the radiation term, $\vec \fy = (\fy_0, \fy_1)$, in the decomposition \eqref{dec}. 

\begin{lem}\label{psi fy} %Let $\vec \fy$ be defined as in \eqref{fy def}, \eqref{fy def1}. 
There exists $\fy \in \HH_0$ with $\E(\vec \fy) \le \eta< 2\E(Q)$ so that the following holds: Denote by $\vec \fy(t)$ the  wave map evolution of $\vec \fy$. Then $\vec \fy(t) \in \HH_0$ is global in time and scatters to zero as $t \to \pm \infty$ and we have
\begin{align}\label{same} 
\vec \fy(t, r) +\pi= \vec \psi(t, r) \quad \forall \, \, (t, r) \in \{ (t, r) \, \vert\, t \in[0, 1),  r\in (1-t, \infty)\}
\end{align}
\end{lem}

\begin{proof}
To begin, let $\bar t_n \to 1$ and $r_n \in (0,1-\bar t_n]$ be chosen as in Corollary \ref{pi seq lem}. We make the following definition: 
\begin{align}  \label{phi 0 n def}
&\phi_n^0(r) = \begin{cases} \pi - \frac{\pi-\psi(\bar t_n, r_n)}{r_n} r \quad \textrm{if} \quad 0\le r\le r_n\\ \psi(\bar t_n, r) \quad \textrm{if} \quad r_n\le r < \infty \end{cases} \\
&\phi_n^1(r) = \begin{cases}0 \quad \textrm{if} \quad 0\le r\le r_n\\ \dot{\psi}(\bar t_n, r) \quad \textrm{if} \quad r_n\le r < \infty \end{cases} 
\end{align}
We claim that $\vec \phi_n:=(\phi_n^0, \phi_n^1)$ forms a bounded sequence in the energy space $\HH$--in fact, the sequence is in $\HH_{1,1}$ which is defined in \eqref{Hnm}. To see this we start with the claim that 
\begin{align} \label{E phi rn to inf}
\E_{r_n}^{\infty}(\vec \phi_n) = \E_{r_n}^{\infty}(\vec \psi(\bar t_n)) \le \eta + o_n(1).
\end{align} 
Indeed, since $\psi(\bar t_n, r_n) \to \pi$ we have $G(\psi(\bar t_n, r_n)) \to 2= \frac{1}{2}\E(Q)$ as $n \to \infty$. Therefore, by \eqref{b R} have $$\E_0^{r_n}(\psi(\bar t_n), 0)  \ge 2G(\psi(\bar t_n, r_n)) \ge \E(Q)- o_n(1)$$ for large $n$ which  proves \eqref{E phi rn to inf} since $ \E_{r_n}^{\infty}(\vec \psi(\bar t_n)) =  \E_{0}^{\infty}(\vec \psi(\bar t_n))-  \E_{0}^{r_n}(\vec \psi(\bar t_n))$.

%----------------------------------------------------------------figure-----------------------------------------------------------------------------%
\begin{figure}
\begin{tikzpicture}[
	>=stealth',
	axis/.style={semithick},
	coord/.style={dashed, semithick},
	yscale = 1.2,
	xscale = 1.2]
	\newcommand{\xmin}{0};
	\newcommand{\xmax}{8};
	\newcommand{\ymin}{0};
	\newcommand{\ymax}{3.5};
	\newcommand{\xa}{3};
	\newcommand{\fsp}{0.2};
	\draw [semithick] (\xmin-\fsp,0) -- (\xmax,0);
	\draw [semithick,->] (0,\ymin-\fsp) -- (0,\ymax+\fsp);
	\draw (\xa,0.1) -- (\xa,-0.1) node [below] {$r_n$};
	\draw (\xmax,0.1) -- (\xmax,-0.1) node [below] {$r=\infty$};
	\draw [semithick, dashed] (\xmin,pi) node [left] {$\pi$} -- (\xmax,pi);
	\draw [thick, dashed] plot[domain={\xmin}:{\xa}] (\x,{atan(\x*\xmax/(\xmax-\x))/90*pi}); %*\xmax/(\xmax-\x)
	\draw [thick] ((\xmin,pi) -- (\xa, {atan(\xa*\xmax/(\xmax-\xa))/90*pi});
	\draw [thick, /pgf/fpu,/pgf/fpu/output format=fixed] plot[domain={\xa}:{\xmax}, samples=400] (\x,{atan(\x*\xmax/(\xmax-\x))/90*pi + 1/100*sin(100*(\x-\xa)^2)*(\x-\xmax)^2*(\x-\xa)});
	\fill (\xmax,pi) circle (0.05);
\end{tikzpicture}
\caption{\label{fig:4} The solid line represents the graph of the function $\phi_n^0(\cdot)$ for fixed $n$, defined in \eqref{phi 0 n def}. The dotted line is the piece of the function $\psi(\bar t_n,\cdot)$ that is chopped at $r=r_n$ in order to linearly connect to $\pi$, which ensures that $\vec \phi_n \in \HH_{1,1}$.}
\end{figure}
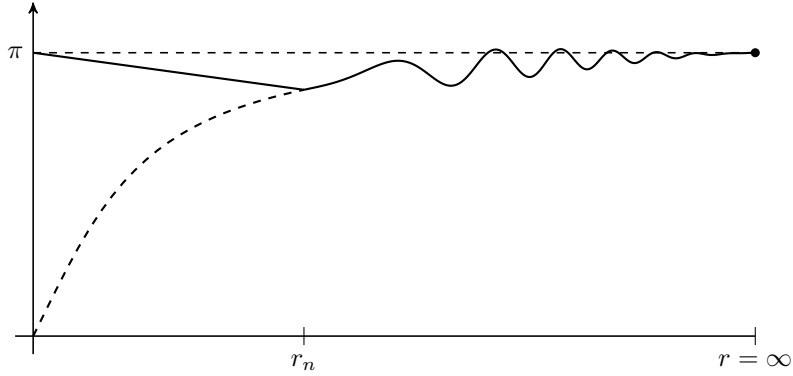
%----------------------------------------------------------------figure-----------------------------------------------------------------------------%

We can also directly compute $\E_{0}^{r_n} (\phi_n^0,0)$. Indeed,
\begin{align*} 
\E_{0}^{r_n} (\phi_n^0,0)&= \int_0^{r_n} \left( \frac{\pi-\psi(\bar t_n, r_n)}{r_n}\right)^2\, r\, dr + \int_0^{r_n} \frac{\sin^2\left( \frac{\pi-\psi(\bar t_n, r_n)}{r_n} r\right)}{r} \, dr\\ \\
& \le C \abs{\pi- \psi(\bar t_n, r_n)}^2  \to 0 \quad \textrm{as} \quad   n \to \infty.
\end{align*} 
Hence $\E( \vec \phi_n)\le \eta + o_n(1)$. This means that for large enough $n$ we have the uniform estimates $\E( \vec \phi_n) \le C < 2\E(Q)$. Therefore, by Theorem~\ref{main}, (which holds with exactly the same statement in $\HH_{1, 1}$ as in $\HH_0= \HH_{0, 0}$), we have that the wave map evolution $\vec \phi_n(t) \in \HH_{1,1}$ with initial data $\vec \phi_n$ is global in time and scatters  to $\pi$ as $t \to \pm \infty$. 
We define $\vec \phi =(\phi_0, \phi_1)\in \HH_{1, 1} $ by
\begin{align}
 \phi_0(r) := \begin{cases}  \pi \quad \textrm{if} \quad r=0\\ \phi_n(1-\bar t_n, r)  \quad \textrm{if} \quad r>2(1-\bar t_n)\end{cases}\\
 \phi_1(r):=\begin{cases}  0 \quad \textrm{if} \quad r=0\\ \dot{\phi}_n(1-\bar t_n, r)  \quad \textrm{if} \quad r>2(1-\bar t_n)\end{cases}
\end{align}
We need to check first that $ \vec \phi$ is well-defined. First recall  that by definition 
\begin{align*}
\vec \phi_n(r)= \vec \psi(\bar t_n, r) \quad \forall r \ge 1-\bar t_n
\end{align*}
since $r_n \le 1-\bar t_n$. Using the finite speed of propagation of the wave map flow, see e.g., \cite{SS},  we can then deduce that for all $t \in [0, 1)$ we have
\begin{align*}
\vec \phi_n(t-\bar t_n, r) = \vec\psi(t, r) \quad \forall\, r\ge 1-\bar t_n + \abs{t-\bar t_n}
\end{align*}
Now let $m >n$ and thus $\bar t_m>\bar t_n$. The above implies that 
\begin{align*}
\vec \phi_n(\bar t_m-\bar t_n, r) = \vec\psi(\bar t_m, r) = \vec \phi_m(r)\quad \forall\, r\ge 1-\bar t_n + \abs{\bar t_m-\bar t_n}
\end{align*}
Therefore, using the finite speed of propagation again we can conclude that 
\begin{align*}
\vec \phi_n(1-\bar t_n, r) = \vec \phi_m(1-\bar t_m, r) \quad \forall\, r >2(1-\bar t_n)
\end{align*}
proving that $\vec \phi$ is well-defined. Next we claim that 
\begin{align}\label{en phi}
\E( \vec \phi) \le \eta
\end{align}
Indeed, observe that by monotonicity of the energy on light cones, see e.g. \cite{SS},  we have 
\begin{align*}
\E_{2(1-\bar t_n)}^{\infty}(\vec \phi)= \E_{2(1-\bar t_n)}^{\infty}(\vec \phi_n(1-\bar t_n)) \le \E_{1-\bar t_n}^{\infty}(\vec \phi_n(0))\le \E(\vec \phi_n(0)) \le \eta + o_n(1)
\end{align*}
and then \eqref{en phi} follows by taking $n \to \infty$ above. Now, let $\vec \phi(t) \in \HH_{1,1}$ denote the wave map evolution of $\vec \phi$. Since $ \vec \phi \in \HH_{1,1}$ and $\E(\vec \phi) \le \eta <2\E(Q)$ we can deduce by Theorem \ref{main} that $\vec \phi(t)$ is global in time and scatters  as $t \to \pm \infty$.  Our final observation regarding $\phi(t)$ is that for all $t \in[0, 1)$ we have 
\begin{align*}
\vec \phi(t, r) = \vec \psi(t, r) \quad \forall\, r>1-t
\end{align*}
This follows immediately from the definition of $\vec \phi$ and the finite speed of propagation. To be specific, fix $t_0 \in [0, 1)$ and $r_0>1-t$. Since $\bar t_n \to 1$ we can choose $n$ large enough so that $r_0>2(1-\bar t_n)+1-t_0$. Then observe that by finite speed of propagation and the fact that $\vec \phi(r)= \vec \phi_n(1-\bar t_n, r)$ for all $r >2(1-\bar t_n)$ we have
\begin{align*}
\vec \phi(t_0, r) = \phi_n(t_0-\bar t_n, r) = \vec \psi(t_0, r) \quad \forall\, r>r_0>2(1-\bar t_n)+1-t_0
\end{align*}
and in particular for $r=r_0$. 

Finally, we define our radiation term $\vec{\fy}=(\fy_0, \fy_1) \in \HH_0$ by setting 
\begin{align} \label{fy def}
&\fy_0(r) := \phi_0 -\pi\\ \label{fy def1}
&\fy_1(r):= \phi_1.
\end{align}
We denote by $\vec \fy(t) \in \HH_0$ the global wave map evolution of $\vec \fy$. %We gather the results established  above in the following lemma:
 \end{proof}

%\begin{lem}\label{psi fy} Let $\vec \fy$ be defined as in \eqref{fy def}, \eqref{fy def1}. Then, $\fy \in \HH_0$ and $\E(\vec \fy) \le \eta< 2\E(Q)$. Denote by $\vec \fy(t)$ the  wave map evolution of $\vec \fy$. Then $\vec \fy(t) \in \HH_0$ is global in time and scatters to zero as $t \to \pm \infty$ and we have
%\begin{align}\label{same} 
%\vec \fy(t, r) +\pi= \vec \psi(t, r) \quad \forall \, \, (t, r) \in \{ (t, r) \, \vert\, t \in[0, 1),  r\in (1-t, \infty)\}
%\end{align}
%\end{lem}

Now define 
\begin{align} \label{a def}
\vec a(t, r) := \vec\psi(t, r)- \vec \fy(t, r).
\end{align} 
We use Lemma~\ref{psi fy} to show that $\vec a(t)$ has the following properties:
\begin{lem}\label{a lem} Let $\vec a(t)$ be defined as in \eqref{a def}. Then $a(t) \in \HH_1$ for all $t \in [0, 1)$ and 
\begin{align} \label{supp a}
\supp( a_r(t), \dot{a}(t))  \in [0, 1-t).
\end{align}
Moreover we have 
\begin{align} \label{a en}
\lim_{t \to 1} \E(\vec a(t)) = \E( \vec \psi) - \E(\vec \fy).
\end{align}
\end{lem} 
\begin{proof} 
First observe that \eqref{supp a} follows immediately from \eqref{same}.  Next we prove \eqref{a en}. First observe since $\vec \fy(t)\in \HH_0$ is a global wave map with $\E(\vec \fy) < 2\E(Q)$ we have 
\begin{align*} 
\sup_{t\in[0, 1]} \| \vec \fy(t) \|_{H \times L^2( r \le \de)} \to 0 \quad \textrm{as} \quad \de \to 0,
\end{align*}
which implies in particular that 
\begin{align}\label{fyreg}
 \| \vec \fy(t) \|_{H \times L^2( r \le 1-t)} \to 0
\end{align}
 as $ t \to 1$. Next we see that 
\begin{align*} 
\E(\vec a (t)) &= \int_0^{1-t} \left(\abs{ \psi_t(t) - \fy_t(t)}^2 + \abs{ \psi_r(t) - \fy_r(t)}^2 + \frac{\sin^2( \psi(t)- \fy(t))}{r^2}\right) \, r\, dr\\
&= \E_0^{1-t}( \vec \psi(t) + \int_0^{1-t} \left(-2 \psi_t(t) \fy(t) -2\psi_r(t) \fy_r(t)\right) \, r\, dr \\
&\, + \int_0^{1-t} \left( \fy_t^2(t) + \fy_r^2(t) \right) \, r\, dr + \int_0^{1-t} \frac{\sin^2( \psi(t)- \fy(t))- \sin^2 (\psi(t))}{r} \, dr\\
& = \E_0^{1-t}(\vec\psi(t))  + C\E(\vec \psi) \|\vec \fy(t) \|_{H \times L^2( r\le1-t)}  +C \|\vec \fy(t)\|_{H \times L^2( r\le1-t)}^2\\
&= \E_0^{1-t}(\vec\psi(t)) + o(1) \quad \textrm{as}\quad t \to 1,
\end{align*}
where on the last line two lines we used \eqref{fyreg} and the fact that 
\begin{align} \label{sin 3}
\abs{\sin^2(x+y) - \sin^2(x) } \le 2 \abs{\sin(x)} \abs{y} +  2\abs{y}^2.
\end{align}
Finally, by Lemma~\ref{psi fy} we observe that for all $t \in [0, 1)$ we have
\begin{align*} 
\E_{1-t}^{\infty}( \vec \psi(t)) = \E_{1-t}^{\infty}( \vec \fy(t)).
\end{align*} 
Hence, 
\begin{align*} 
\E(\vec a(t)) = \E(\vec \psi(t)) - \E_{1-t}^{\infty}( \vec \fy(t)) + o(1) \quad \textrm{as}\quad t \to 1,
\end{align*}
which completes the proof.
\end{proof}
 
 \subsection{Extraction of the blow-up profile} 
 Next, we use Struwe's result, Theorem~\ref{str}, to extract a sequence of  properly rescaled  harmonic maps.  At this point we note that we can, after a suitable rescaling and time translation, assume, without loss of generality, that the scale $\la_0$ in Theorem~\ref{str} satisfies $\la_0=1$. We prove the following result:

\begin{prop} \label{a prop}
Let $\vec a(t) \in \HH_1$ be defined as in \eqref{a def}. There exists a sequence $\al_n$ with $\al_n \to \infty$,  a sequence of times $\tau_n \to 1$ and a sequence of scales $\la_n = o(1-\tau_n)$ and $\al_n \la_n <1-\tau_n$ such that 
\begin{itemize} 
\item[($a$)] As $n \to \infty$ we have 
\begin{align} \label{a dot}
\int_0^{\infty} \dot{a}^2(\tau_n, r) \, r\, dr \le \frac{1}{n}.
\end{align}
\item[($b$)] As $n \to \infty$  we have
\begin{align} \label{a-Q}
\int_0^{\al_n \la_n} \left(\abs{a_r(\tau_n, r) - \frac{Q_r(r/ \la_n)}{\la_n}}^2 + \frac{\abs{a(\tau_n, r)- Q(r/ \la_n)}^2}{r^2}\right) \, r\, dr \le \frac{1}{n}.
\end{align}
\item[($c$)]As $n \to \infty$ we also have 
\begin{align} \label{b 2eq}
\E(\vec a(\tau_n) - (Q(\cdot/\la_n), 0)) \le \eta + o_n(1),
\end{align}
which implies that for large enough $n$ we have $$\E(\vec a(\tau_n) - (Q(\cdot/\la_n), 0)) \le C <2\E(Q).$$
\end{itemize}
\end{prop}

\begin{proof} 
We begin by establishing \eqref{a dot} and \eqref{a-Q}. The basis for the argument is Theorem~\ref{str}.  Indeed, by  Theorem~\ref{str} and Corollary~\ref{pi seq lem} there exists a sequence of times $t_n \to 0$ and a sequence of scales $\la_n = o(1-t_n)$ such that for any $B \ge 0$ we have 
\begin{align*} 
&\frac{1}{\la_n} \int_{t_n}^{t_n + \la_n} \int_0^{1-t}  \dot \psi^2(t,  r)  \, r \, dr \, dt \to 0\\
&\frac{1}{\la_n} \int_{t_n}^{t_n + \la_n} \int_0^{B \la_n}\left( \abs{ \psi_r(t, r) - \frac{ Q_r(r/ \la_n)}{\la_n}}^2 + \frac{\abs{\psi(t, r)- Q(r/ \la_n)}^2}{r^2}\right) \, r \, dr \, dt \to 0
\end{align*}
as $n \to \infty$. Next observe that since $\vec \fy(t) \in \HH_0$ is a global wave map with $\E(\vec \fy)<2 \E(Q)$, we can use the monotonicity of the energy on light cones to deduce that 
\begin{align} \label{fy reg3}
\sup_{ t_n\le t \le 1}\E_0^{1-t}(\vec \fy(t)) \to 0 \quad \textrm{as} \quad n \to \infty.
\end{align} 
The above then implies that
\begin{align} \label{fy reg2}
\sup_{ t_n \le t \le 1} \|\vec \fy(t)\|_{H\times L^2(r \le 1-t)} \to 0 \quad \textrm{as} \quad n \to \infty.
\end{align} 
By \eqref{a def}, Lemma~\ref{a lem} we then have 
\begin{align*}
\frac{1}{\la_n} \int_{t_n}^{t_n + \la_n} \int_0^{\infty}  \dot a^2(t,  r)  \, r \, dr\, dt  &=  \frac{1}{\la_n} \int_{t_n}^{t_n + \la_n}\int_0^{1-t}  \abs{\dot \psi(t,  r)- \dot{\fy}(t, r)}^2  \, r \, dr\, dt\\
 &\lesssim \frac{1}{\la_n} \int_{t_n}^{t_n + \la_n}\int_0^{1-t}  \dot \psi^2(t,  r)  \, r \, dr\, dt \\
 &\quad+\frac{1}{\la_n} \int_{t_n}^{t_n + \la_n} \int_0^{1-t}  \dot \fy^2(t,  r)  \, r \, dr\, dt \to 0.
 \end{align*}
 Using \eqref{fy reg2} it is also immediate that 
 \begin{align*} 
 \frac{1}{\la_n} \int_{t_n}^{t_n + \la_n} \int_0^{B \la_n}\left( \abs{ a_r(t, r) - \frac{ Q_r(r/ \la_n)}{\la_n}}^2 + \frac{\abs{a(t, r)- Q(r/ \la_n)}^2}{r^2}\right) \, r \, dr \, dt \to 0.
\end{align*}
Now, define
\begin{multline*}
s(B,n):= \frac{1}{\la_n} \int_{t_n}^{t_n + \la_n} \int_0^{\infty}  \dot a^2(t,  r)  \, r \, dr\, dt  \\+ \frac{1}{\la_n} \int_{t_n}^{t_n + \la_n} \int_0^{B \la_n}\left( \abs{ a_r(t, r) - \frac{ Q_r(r/ \la_n)}{\la_n}}^2 + \frac{\abs{a(t, r)- Q(r/ \la_n)}^2}{r^2}\right) \, r \, dr \, dt.
\end{multline*}
We know that for all $B \ge 0$ we have $s(B, n) \to 0$ as $n \to \infty$. Let  $\al_n \to \infty$.  Then there exists a subsequence  $\sigma(n)$ such  that $s( \al_n , \sigma(n)) \to 0$ as $n \to \infty$ with $\al_n \la_{\sigma(n)}<1-t_{\sigma(n)}$. To see this let $N(B, \de)$ be defined so that  for $n \ge N(B, \de)$ we have $s(B, n) \le \de$ and then set  $ \sigma(n):=N(\al_n, 1/n)$. Note that we necessarily have $\al_n\la_{\sigma(n)}<1-t_{\sigma(n)}$. Then  we can extract $\tau_{\sigma(n)} \in [t_{\sigma(n)}, t_{\sigma(n)} + \la_{\sigma(n)}]$ so that after relabeling we have 
\begin{multline*} 
 \int_0^{\infty}  \dot a^2(\tau_n,  r)  \, r \, dr  \\+  \int_0^{\al_n \la_n}\left( \abs{ a_r(\tau_n, r) - \frac{ Q_r(r/ \la_n)}{\la_n}}^2 + \frac{\abs{a(\tau_n, r)- Q(r/ \la_n)}^2}{r^2}\right) \, r \, dr  \le \frac{1}{n}
 \end{multline*}
 for every $n$ which proves \eqref{a dot} and \eqref{a-Q}. 
 
Lastly, we establish \eqref{b 2eq}. To see this, let $\tau_n$ and $\la_n$ be as in  \eqref{a dot} and \eqref{a-Q}. Observe that 
\begin{align*} 
\E( \vec a(\tau_n) -(Q(\cdot/ \la_n), 0) &= \E_0^{\al_n \la_n}( \vec a(\tau_n) -(Q(\cdot/ \la_n), 0)) \\
&+ \E_{\al_n \la_n}^{1-\tau_n}( \vec a(\tau_n) -(Q(\cdot/ \la_n), 0) )\\
&+ \E_{1-\tau_n}^{\infty}( \vec a(\tau_n) -(Q(\cdot/ \la_n), 0) ).
\end{align*}
 First, observe that  \eqref{a dot} and \eqref{a-Q} directly imply that 
 \begin{align} \label{e a-q}
  \E_0^{\al_n \la_n}( \vec a(\tau_n) -(Q(\cdot/ \la_n), 0)) = o_n(1)
  \end{align} 
  as $n \to \infty$. Next we observe that 
  \begin{align} \label{Q tail}
  \E_{\al_n \la_n}^{\infty}(Q(\cdot/ \la_n)) = \E_{\al_n}^{\infty}(Q) = o_n(1).
  \end{align}
 Using \eqref{Q tail} and the fact that $\vec a(\tau_n, r) = (\pi, 0)$ for every $r \in[1-\tau_n, \infty)$, we have that 
  \begin{align*} 
  \E_{1-\tau_n}^{\infty}( \vec a(\tau_n) -(Q(\cdot/ \la_n), 0) ) &= \E_{1-\tau_n}^{\infty}(( \pi, 0)- (Q(\cdot/ \la_n), 0) )\\
  & \le \E_{\al_n \la_n}^{\infty}( Q(\cdot/ \la_n) ) = o_n(1).
  \end{align*} 
  Hence it suffices to show that 
  \begin{align} \label{a nlan}
   \E_{\al_n \la_n}^{1-\tau_n}( \vec a(\tau_n) -(Q(\cdot/ \la_n), 0) ) \le \eta + o_n(1).
   \end{align}
   Applying \eqref{Q tail} again we see that the above reduces to showing that  
   \begin{align*} 
    \E_{\al_n \la_n}^{1-\tau_n}( \vec a(\tau_n)) \le \eta + o_n(1).
   \end{align*}
Now combine the following two facts. One the one hand, for large $n$, \eqref{a en} implies that 
   \begin{align*} 
   \E(\vec a(\tau_n)) \le \E(\vec \psi)  + o_n(1).
   \end{align*} 
   On the other hand,  \eqref{a dot} and \eqref{a-Q} give that $\E_0^{\al_n \la_n}(\vec a(\tau_n)) = \E(Q) - o_n(1)$. Putting this all together we obtain \eqref{a nlan}. 
      \end{proof}

In the next section we will also need the following consequence of Proposition~\ref{a prop}. 

\begin{lem}\label{tech lem} Let $\al_n, \la_n, $ and $ \tau_n$ be defined as in Proposition~\ref{a prop}. Let $\be_n \to \infty$ be any other sequence such that $\be_n \le c_0\al_n$ for all $n$, for some $c_0 <1$. Then for every $0<c_1<C_2$ such that $C_2 c_0 \le 1$ there exists $\ti \be_n$ with $c_1 \be_n \le \ti \be_n\le C_2 \be_n$ such that 
\begin{align}
&\psi(\tau_n, \ti \be_n \la_n) \to \pi \quad \textrm{as} \quad n \to \infty \label{psi be to pi}
\end{align}
\end{lem}

\begin{proof}We first observe that we can combine \eqref{a-Q} and \eqref{fyreg} to conclude that 
\begin{align}\label{psi-Q in H}
\|\vec \psi(\tau_n)- (Q(\cdot/ \la_n), 0)\|_{H \times L^2(r \le \al_n \la_n)} \to 0
\end{align}
as $n \to \infty$. Now,  suppose \eqref{psi be to pi} fails. Then there exists $\de_0>0$, $\be_n \to  \infty$ with $\be_n\le c_0 \al_n$, and $c_1<C_2$, and a subsequence so that 
\begin{align*}
 \forall n \quad \psi( \tau_n,  \la_n r) \not\in [\pi- \de_0, \pi+ \de_0] \quad \forall r\in[c_1  \be_n, C_2 \be_n]
\end{align*} 
Now, since $\be_n \to \infty$ we can choose $n $ large enough so that 
\begin{align*}
Q(r) \in [\pi- \de_0/2, \pi) \quad \forall r \in[c_1  \be_n, C_2 \be_n]
\end{align*}
Putting this together we have that  
\begin{align*}
\int_{c_1\be_n}^{C_2 \be_n} \frac{\abs{\psi(\tau_n, \la_nr)- Q(r)}^2}{r} \, dr \ge \left( \frac{C_2-c_1}{2c_1}\right)^2 \de_0^2
\end{align*}
But this directly contradicts \eqref{psi-Q in H} since $C_2 \be_n \le \al_n$ for every $n$.  
\end{proof}

\subsection{Compactness of the error}

For the remainder of this section, $\al_n, \tau_n$ and $\la_n$ will all be defined by Proposition~\ref{a prop}. Next, we define $\vec b_n \in \HH_0$ as follows:
\begin{align} \label{b def}
&b_{n, 0}(r) = a(\tau_n, r) - Q(r/\la_n)\\
&b_{n, 1}(r)=  \dot{a}(\tau_n, r)= o_n(1) \quad \textrm{in} \quad L^2 \label{b def 1}.
\end{align}

Our goal in this section is to complete the proof of Proposition~\ref{bu} by showing that $\vec b_n \to 0$ in the energy space. Indeed we prove the following result: 
\begin{prop} \label{b comp}
Define $\vec b_n= (b_{n, 0}, b_{n, 1})$ as in \eqref{b def}, \eqref{b def 1}. Then 
\begin{align} \label{b to 0}
\|\vec b_n\|_{H \times L^2} \to 0
\end{align} 
as $n \to \infty$. 
\end{prop}

To begin, we observe that by Proposition~\ref{a prop} we have $$\E(\vec b_n) \le C <2\E(Q)$$ for $n$ large enough.  Denote by $\vec b_n(t) \in \HH_0$ the wave map evolution with data $\vec b_n \in \HH_0$. Since $\E(\vec b_n)\le C <2\E(Q)$ for large $n$,  we know from Theorem~\ref{main} that $\vec b_n(t) \in \HH_0$ is global  and scatters to zero as $t \to \pm \infty$.

The proof of Proposition~\ref{b comp}  proceeds in several steps. We give a brief outline of the approach below to give the reader a general sense of the strategy.
\begin{itemize}
\item 
The first step in the proof of Proposition~\ref{b comp} is to show that the sequence $\vec b_n$ does not contain any nonzero profiles. The proof of this step is reminiscent of an argument given in \cite[Section $5$]{DKM1} and in particular  \cite[Proposition $5.1$]{DKM1}. Here the situation has been simplified as we have already extracted the large profile $Q(\cdot/ \la_n)$ by means of Struwe's theorem. This result is achieved in Proposition~\ref{no prof} below. 
\item Next,  we proceed by contradiction. If we assume that the conclusion of Proposition~\ref{b comp} fails, then up to extracting a subsequence we have $$\| \vec b_n\|_{H \times L^2} \ge \de_0.$$ As we have shown in the first step that $\vec b_n$ has no nonzero profiles, the nonlinear wave map evolution $\vec b_n(t)$ is well approximated by the corresponding linear flow with the same initial data, which we denote by $\vec b_{n, L}(t)$.  Using the exterior linear estimates, Corollary~\ref{linear ext en}, we can then deduce that the wave map flow $\vec b_n(t)$ maintains a fixed amount of energy exterior to the light cone, viz.~\eqref{b ext}.

\item The final step is to show that~\eqref{b ext} forces the original wave map, $\vec \psi(t)$ to concentrate energy on the boundary of the cone {\em before} the blow-up time $t=1$.  This is of course impossible, both by equivariance -- the only concentration point can be $r=0$ -- and by our assumptions that the map blows up at time $t=1$. The proof is a delicate argument that  is based on showing that  the backwards evolutions of $\vec \psi(\tau_n)$ and $\vec b_n$ remain close on an exterior region. %the decomposition, 
%\begin{align}\label{dec0}
%\vec \psi(\tau_n, r) = (Q(r/ \la_n), 0) +  \vec \fy(\tau_n, r) +  \vec b_n(r) 
%\end{align}
%is maintained after being evolved {\em nonlinearly} backwards in time up to harmless error terms. 
 Several technical difficulties arise from the fact that a Bahouri-Gerard type nonlinear profile decomposition cannot be directly applied, since $\vec \psi(\tau_n) \in \HH_1$, and not in $\HH_0$. The pieces of the decomposition must then be evolved separately, in two different steps,  and we rely on finite speed of propagation, as well as on Corollary~\ref{Cote} to show that $\vec \psi(\tau_n +t)$ and $\vec b_n(t)$ remain close throughout this process near the boundary of the cone. In particular, the nonzero amount of energy that $\vec b_n(t)$ maintains on the exterior cone forces $\vec \psi(\tau_n +t)$ to concentrate energy there as well, which will give us a  contradiction. 
 \end{itemize} 

 We now begin with the proof. 

\begin{prop}\label{no prof}
Let $\vec b_n \in \HH_0$ and the corresponding global wave map $\vec b_n(t)\in \HH_0$  be defined as above. Then there exists a decomposition 
\begin{align} 
\vec b_n(t, r) = \vec b_{n, L}(t, r) + \vec \theta_n(t, r) 
\end{align}
where $\vec b_{n, L}(t, r)$ satisfies the linear wave equation 
\begin{align} \label{linear 2d wave}
\p_{tt}b_{n, L} - \p_{rr}b_{n, L} -\frac{1}{r}\p_rb_{n, L} + \frac{1}{r^2}b_{n, L}=0
\end{align}
with initial data $\vec b_{n, L}(0, r)= (b_{n, 0}, 0)$. Moreover, $b_{n, L}$ and $\vec \theta_n$ satisfy
\begin{align}
&\left\|\frac{1}{r} b_{n, L}\right\|_{L^3_t(\R; L^6_x(\R^4))} \longrightarrow 0\label{bL to 0}\\
&\|\vec \theta_n\|_{L^{\infty}_t(\R; H\times L^2)} + \left\| \frac{1}{r} \theta_n\right\|_{L^3_t(\R ; L^6_x(\R^4))} \longrightarrow 0 \label{theta to 0}
\end{align}
as $n \to \infty$. 
\end{prop}

Before beginning the proof of Proposition~\ref{no prof} we deduce the following  corollary which will be an essential ingredient in the proof of Proposition~\ref{b comp}.

\begin{cor} \label{ext en est}
Let $\vec b_n(t)$ be defined as in Proposition~\ref{no prof}. Suppose that there exists a constant $\de_0$ and a subsequence in $n$ so that $\|b_{n, 0}\|_H \ge \de_0$.  Then there exists $\al_0 >0$ such that for all $t>0$ and all $n$ large enough we have
\begin{align} \label{b ext}
\|\vec b_n(t) \|_{H \times L^2(r \ge t)} \ge \al_0 \de_0
\end{align}
\end{cor}
\begin{proof} First note that since $\vec b_{n, L} $ satisfies the  linear wave equation \eqref{linear 2d wave} with initial data $\vec b_{n, L}(0) =(b_{n, 0}, 0)$ we know by Corollary~\ref{linear ext en} that there exists a constant $\be_0 >0$ so that for each $t\ge 0$ we have
\begin{align*} 
\|\vec b_{n, L}(t)\|_{H \times L^2(r \ge t)} \ge \be_0\|b_{n, 0}\|_H
\end{align*}
On the other hand, by Proposition~\ref{no prof} we know that 
\begin{align*} 
\|\vec b_{n}(t) - \vec b_{n, L}(t)\|_{H \times L^2(r\ge t)} \le \|\vec \theta_n(t)\|_{H \times L^2} = o_n(1)
\end{align*}
Putting these two facts together gives 
\begin{align*} 
\|\vec b_n(t) \|_{H \times L^2(r \ge t)} &\ge \|b_{n, L}(t)\|_{H \times L^2(r \ge t)}  -o_n(1)\\
&\ge  \be_0\|b_{n, 0}\|_H -o_n(1)
\end{align*}
This yields \eqref{b ext} by passing to a suitable subsequence and taking $n$ large enough.\end{proof} 

To prove Proposition~\ref{no prof} we will first pass to the standard $4d$ representation in order to perform a profile decomposition on the sequence $\vec b_n$. Up to extracting a subsequence, $\vec b_n \in \HH_0$ forms a uniformly bounded sequence with $\E(\vec b_n) \le C <2\E(Q)$. By Lemma~\ref{ckm lem2} and the right-most equality in \eqref{2-4}, the sequence $\vec u_n =(u_{n, 0}, u_{n,1})$ defined by 
\begin{align}
&u_{n, 0}(r) = \frac{b_{n, 0}(r)}{r}\\
&u_{n, 1}(r) = \frac{b_{n, 1}(r)}{r} = o_n(1) \quad \textrm{in} \quad L^2(\R^4)\label{L2}
\end{align} 
is uniformly bounded in $\dot{H}^1 \times L^2 (\R^4)$. By Theorem~\ref{BaG} we can perform the following profile decomposition on the sequence $\vec u_n$: 
\begin{align}
&u_{n, 0}(r) = \sum_{j \le k}\frac{1}{\la_n^j} V^j_L\left(\frac{-t_n^j}{\la_n^j}, \frac{r}{\la_n^j}\right) + w_{n, 0}^k(0, r) \label{bg0}\\
&u_{n, 1}(r) = \sum_{j \le k}\frac{1}{(\la_n^j)^2} \dot V_L^j\left(\frac{-t_n^j}{\la_n^j}, \frac{r}{\la_n^j}\right) + w_{n, 1}^k(0, r)\label{bg1}
\end{align}
where each $\vec V^j_L$ is a free radial wave in $4d$ and where we have for $j \neq k$:
\begin{align} \label{ort}
\frac{\la_n^j}{\la_n^k} + \frac{\la_n^k}{\la_n^j} + \frac{\abs{t_n^j -t_n^k}}{\la_n^k} +  \frac{\abs{t_n^j -t_n^k}}{\la_n^j} \to \infty \quad \textrm{as} \quad n \to \infty
\end{align} 
Moreover, if we denote by $\vec w_{n, L}^k(t)$ the free evolution of $\vec w_n^k$ we have for $j\le k$ that 
\begin{align} 
&\left(\la_n^j w_{n,L}^k(\la_n^jt_n^j, \la_n^j \cdot), (\la_n^j)^2 \dot w_{n,L}^k(\la_n^j t_n^j, \la_n^j \cdot)\right) \rightharpoonup 0 \in \dot{H}^1 \times L^2 \quad \textrm{as} \quad n \to \infty \label{weak}\\
&\limsup_{n \to \infty} \|w_{n,L}^k\|_{L^3_tL^6_x} \to 0 \quad\textrm{as} \quad k \to \infty\label{stric}
\end{align}
Finally,   
\begin{align} \label{pyt}
\|\vec u_n \|_{\dot{H}^1 \times L^2}^2 = \sum_{j \le k}\left\| \vec V_L^j\left(\frac{-t_n^j}{\la_n^j}\right)\right\|_{\dot{H}^1 \times L^2}^2 + \| \vec w_{n}^k(0)\|_{\dot{H}^1 \times L^2}^2 + o_n(1)
\end{align} 
%\begin{align} \label{pyt}
%&\|u_{n, 0}\|_{\dot{H}^1}^2 = \sum_{j \le k}\left\| V_L^j\left(\frac{-t_n^j}{\la_n^j}\right)\right\|_{\dot{H}^1}^2 + \|w_{n, 0}^k(0)\|_{\dot{H}^1}^2 + o_n(1)\\
%&\|u_{n, 1}\|_{L^2}^2 = \sum_{j \le k}\left\| \dot V_L^j\left(\frac{-t_n^j}{\la_n^j}\right)\right\|_{L^2}^2 + \|w_{n, 1}^k(0)\|_{L^2}^2 + o_n(1) \label{pyt2}.
%\end{align}
It  is  also convenient to rephrase the above profile decomposition in the $2d$ formulation. We have 
\begin{align}
&b_{n, 0}(r)= \sum_{j \le k}\fy^j_L\left(  \frac{ -t_n^j}{ \la_n^j} , \frac{r}{\la_n^j}\right) + \ga_{n, 0}^k( r)\\
&b_{n, 1}(r) = \sum_{j \le k} \frac{1}{\la_n^j}   \dot \fy^j_L\left(  \frac{ -t_n^j}{ \la_n^j} , \frac{r}{\la_n^j}\right) + \ga_{n, 1}^k( r),
\end{align}
where 
\begin{align*} 
&\fy^j_L\left( \frac{-t_n^j}{\la_n^j}, \frac{r}{\la_n^j}\right):=\frac{r}{\la_n^j} V_L^j\left( \frac{-t_n^j}{\la_n^j}, \frac{r}{\la_n^j}\right)    \\
&\gamma_{n}^k(r):=rw_{n, 0}^k( r) .
\end{align*}
and similarly for the time derivatives. 

We  make the following crucial observation about the scales $\la_n^j$. By Proposition~\ref{a prop} we have as $n \to \infty$ that
\begin{align} \label{aln b}
\E_0^{\al_n\la_n}(b_{n, 0}, 0) \to 0,\\
\E_{1-\tau_n}^{\infty}(b_{n, 0}, 0) \to 0\label{1-tn b}.
\end{align}
Note that we also have that if $\be_n \to \infty$ is any other sequence with $\be_n \le \al_n$ then  
\begin{align} \label{ben b}
\E_0^{\be_n\la_n}(b_{n, 0}, 0) \to 0.
\end{align}
We can combine ~\eqref{aln b} and \eqref{1-tn b} with  Proposition~\ref{c16} to conclude that for each scale $\la_n^j$ corresponding to a nonzero profile $\fy^j$ we have 
\begin{align} \label{diff scales}
\la_n \ll \la_n^j \le 1-\tau_n
\end{align} 
at least for $n$ large. In particular,  
\begin{align}\label{la to 0}
\la_n^j \to 0 \quad \textrm{as} \quad n \to \infty \quad \textrm{for every} \, j.
\end{align}

The proof of Proposition~\ref{no prof} will consist of a sequence of steps designed to show that each of the profiles $ \vec V^j_L$ (or equivalently the $\vec \phi^j_L$)  must be identically zero. 

Our first goal is to show that all of the time sequences $\{t_{n, j}\}$ can be taken to be $\equiv 0$ and that then the initial velocities of the profiles vanish, i.e., $\dot V_L^j(0, r) \equiv 0$ for each $j$. This is an easy consequence of the  following lemma:
\begin{lem} \label{l:tnj} 
In the decomposition~\eqref{bg0},~\eqref{bg1} we must have 
\EQ{
 \limsup_{n \to \infty} \abs{ \frac{t_{n}^j}{\la_n^j}}  < \infty \quad \forall \, j  \in \N.
 }
 
\end{lem}

\begin{cor}  \label{c:vt0} 
 In the decomposition \eqref{bg0}, \eqref{bg1} we can assume, without loss of generality, that $t_n^j =0$ for every $n$ and for every $j$. And, in addition we then have 
\begin{align*} 
\dot V_L^j(0, r) \equiv 0 \quad \textrm{for every} \quad  j.
\end{align*} 
\end{cor}
\begin{proof}[Proof of Corollary~\ref{c:vt0}]
Since all of the sequences $t_n^j/ \la_n^j$ are bounded, we can assume (by translating the profiles) that $t_{n}^j  \equiv 0$ for all $j$ and for all $n$. In the case when $t_n^j = 0$ for all $j$, it is easy to see that, besides~\eqref{pyt} the following Pythagorean expansion also
holds%It then follows from the Pythagorean expansion~\eqref{pyt} that 
\EQ{
o_n(1)  = \|u_{n, 1}\|_{L^2}^2 = \sum_{j \le k}\left\| \dot V_L^j(0)\right\|_{L^2}^2 + \|w_{n, 1}^k(0)\|_{L^2}^2 + o_n(1)    \label{pyt2}, 
}
%By \eqref{L2} and \eqref{pyt2} we have for all $j$ that 
%\begin{align} \label{dot V}
%\lim_{n \to \infty} \|\dot{V}_L^j(-t_n^j/ \la_n^j)\|_{L^2} = 0.
%\end{align} 
from which it is immediate that 
%From this we shall now deduce that for any $V_L^j \neq 0$ the sequence $\{-t_n^j/ \la_n^j\}$ is bounded. If not,  there exists a $j$ such that $V_L^j \neq 0$ and the sequence $\{-t_n^j/ \la_n^j\} \to \pm \infty$. But, by the equipartition of energy for free waves we would then have 
%\begin{align*} 
%\frac{1}{2}( \|V^j_L(0)\|_{\dot H^1}^2 + \|\dot V_L^j(0) \|_{L^2}^2 ) = \lim_{n \to \infty}\|\dot{V}_L^j(-t_n^j/ \la_n^j)\|_{L^2}^2= 0,
%\end{align*} 
%which contradicts the fact that we assumed $V_L^j \neq 0$. Therefore, translating the profiles in time,  we are free to assume that 
%\begin{align*} 
%t_n^j = 0 \,\, \,  \forall j, \, \forall n.
%\end{align*}
%But then \eqref{dot V} implies that 
$V_1^j:= \dot V^j_L(0) =0$ for every $j$. 
\end{proof} 

We now move to the proof of  Lemma~\ref{l:tnj}. We follow closely the argument in~\cite{DKMe}, however since there are a few technical differences, we reproduce the proof here. 

Note that one way of viewing Corollary~\ref{c:vt0} is that, under the hypothesis,  one has ability to pass from~\eqref{pyt} to~\eqref{pyt2}. For a profile decomposition of a general sequence $(v_{n, 0}, v_{n, 1})$ in $\dot H^1 \times L^2(\R^4)$ with $\|v_{n, 1}\|_{L^2}   = o_n(1)$ this is not possible due to the following example: Let $ \vec V_L(t)$ be any nonzero free wave and let $s_n \to \infty$ be any sequence of times. Let $v_{n, 0}:= 2 V_L(s_n)$ and $v_{n, 1} = 0$. Then 
\EQ{
v_{n, 0} = V_L^1( - s_n^1) + V_L^2( -s_{n}^2),  \quad %\quad  V_L^1(t) = V_L(t), \, \, s_n^1 = s_n \quad V_L^2(t) = V_L(-t)\, \, s_n^2 = - s_n \\
 v_{n, 1} = 0
}
where 
\begin{align*}
 V_L^1(t) := V_L(t), \, \, s_n^1 := -s_n,  \quad V_L^2(t) := V_L(-t), \, \, s_n^2 :=  s_n
\end{align*}
is a profile decomposition which does not satisfy 
\begin{align*}
 0 = \| u_{n, 1}\|_{L^2}^2 \neq   \|  \dot V_L^1(-s_n^1) \|_{L^2}^2  + \|  \dot V_L^2(-s_n^2) \|_{L^2}^2  + o_n(1).
 \end{align*}
%the Pythagorean expansion for the $\dot{H}^1$ norm of $v_{n, 0}$  and $L^2$ norm  $v_{n, 1}$ separately. 

With this example in mind, the first step towards proving Lemma~\ref{l:tnj} is to show that such time-symmetric profiles are the only type that can arise with diverging parameters $t_{n}^j/ \la_{n}^j  \to \pm \infty$, for a sequence $(v_{n, 0}, v_{n, 1})$ in $\dot H^1 \times L^2(\R^4)$ with $\|v_{n, 1}\|_{L^2}   = o_n(1)$. 

We begin by  establishing the following claim. Denote by $ \vec S(t)$ the free wave propagator in $\R^{1+4}$, i.e., for data $(f, g)$ we set 
\begin{align*}
&S(t) (f, g) =  \cos (t \sqrt{ - \De}) f + \frac{ \sin( t\sqrt{- \De})}{  \sqrt{ -\De}} g,  \\
& \vec S(t)(f, g) :=  ( S(t)(f, g) ,\,   \p_t S(t)(f, g)).
\end{align*} 
\begin{claim}\cite[Claim 2]{DKMe} \label{c:2} 
 Let $\{f_{n}, g_{n}\}$ be a bounded sequence of radial functions in $\dot{H}^1 \times L^2 (\R^4)$ and let $A_n>0$ be any sequence so that 
 \EQ{
 \| g_n \|_{L^2( r \ge A_n)} \to 0 \mas  n \to \infty.
 }
 Let $t_n$ be a time sequence so that $ \abs{t_n}/ A_n  \to \infty $ as $n \to \infty$. If 
\begin{align*}
  \vec S(-t_n)( f_n, g_n) \rightharpoonup  (V_0, V_1) \in \dot{H}^1 \times L^2,
\end{align*}
then, 
\begin{align*}
  \vec S(t_n)( f_n, g_n) \rightharpoonup  (V_0,  - V_1) \in \dot{H}^1 \times L^2.
\end{align*} 
\end{claim} 
 \begin{proof}
 The proof follows closely the argument given in~\cite{DKMe}, but here we crucially use~\cite[Theorem 4]{CKS} in place of~\cite[Lemma 4.1]{DKM1}.  Denote by $\ang{\cdot, \cdot}_{\dot H^1 \times L^2}$ the inner product in $\dot {H}^1 \times L^2$. Given any radial $(h_0, h_1) \in C^{\infty}_0 \times C^\infty_0 (\R^4)$ we have 
 \begin{multline*}
  \ang{  \vec S( - t_n)(f_n, g_n), \, ( h_0, h_1)}_{\dot H^1 \times L^2}  =   \ang{  (f_n, g_n), \, \vec S( t_n)( h_0, h_1)}_{\dot H^1 \times L^2} \\
   = \ang{  (f_n, - g_n), \, \vec S( t_n)( h_0, h_1)}_{\dot H^1 \times L^2} + o_n(1) \mas n \to \infty.
 \end{multline*} 
We note that the last inequality above is due our assumptions on $g_n$. Indeed, by~\cite[Theorem 4]{CKS} (which says roughly that radial free waves radiate most of their energy near the light cone)  and since $\abs{t_n}/ A_n \to  \infty$, we have  
\begin{align*}
 \ang{  (0, g_n), \, \vec S( t_n)( h_0, h_1)}_{\dot H^1 \times L^2}  =  o_n(1) \mas n \to \infty.
\end{align*}  
Using the fact that for any data $(f, g)$ we have $$ \vec S(- t) (f, -g)  = ( S(t)(f, g),  -  \p_t S(t)(f, g))$$ we obtain
\begin{multline*}
\ang{  \vec S( - t_n)(f_n, g_n), \, ( h_0, h_1)}_{\dot H^1 \times L^2}  = \ang{  \vec S( - t_n)(f_n, - g_n), \, ( h_0, h_1)}_{\dot H^1 \times L^2}  + o_n(1) \\
 =  \ang{  \vec S(  t_n)(f_n, g_n), \, ( h_0,  - h_1)}_{\dot H^1 \times L^2} + o_n(1)  \mas n \to \infty,
\end{multline*}
which completes the proof. 
 \end{proof}
 
 \begin{claim}\label{c:3} 
 Let $(v_{n, 0}, v_{n, 1})$ be a bounded sequence of radial functions in $\dot H^1 \times L^2 (\R^4)$ such that 
 \EQ{ \label{vnto0} 
  \| v_{n, 1} \|_{L^2} \to 0 \mas n \to \infty.
 }
 Then, after passing to a subsequence, there exists a profile decomposition with free waves $V_{L}^j$ and parameters $\{t_n^j, \la_n^j\}$ so that for any fixed $j \in \N$ we have either 
 \EQ{
 t_{n}^j =  0 , \quad \forall n  \mand \dot V_L^j(0) = 0,
 }
 or 
 \EQ{ \label{tnjI} 
  \frac{t_n^j}{\la_n^j} \to  \pm \infty \mas n \to \infty
  }
  and 
 there exists $k \neq j$ so that 
 \EQ{ \label{kth prof} 
 V_{L}^k(t)  = V_L^j(-t) \mand  \, \, \forall n \, \,  t_n^j = - t_{n}^k, \quad  \la_{n}^k = \la_n^j.
 }
 \end{claim} 
 \begin{proof}
 Fix and $j \in \N$. Recall from~\cite{BG} that the profile $ \vec V_L^j$ with parameters $\{t_n^j, \la_n^j\}$ is defined by the weak limit 
 \EQ{ \label{profile}
 \vec S(t_n^j/ \la_n^j)(  \la_n^j v_{n, 0}( \la_n^j \cdot), (\la_n^j)^2 v_{n, 1}( \la_n^j \cdot)) \rightharpoonup  \vec V_L^j(0) \in \dot H^1 \times L^2.
 }
Now, we can assume without loss of generality that either $t_n^j  = 0$ for all $n$ or that~\eqref{tnjI}  holds. If $t_{n, j} =  0$ then~\eqref{vnto0} and~\eqref{profile} show that 
 $\p_t V_L (0) = 0$. In the latter case, we can use Claim~\ref{c:2} to extract the weak limit  
 \EQ{
  \vec S(- t_n^j/ \la_n^j)(  \la_n^j v_{n, 0}( \la_n^j \cdot), (\la_n^j)^2 v_{n, 1}( \la_n^j \cdot)) \rightharpoonup  ( V_L^j(0), - \p_t V_L^j(0)) \in \dot H^1 \times L^2.
  }
  This gives us the existence of the $k$th profile $V^k_L$  precisely as in~\eqref{kth prof}.  \end{proof} 
 % \begin{align*}
 % V^k_L(0) = V^j_L(0), \quad 
 % \end{align*} 
 We can now prove Lemma~\ref{l:tnj}. 
\begin{proof}[Proof of Lemma~\ref{l:tnj}] We argue by contradiction. Passing to the $2d$ formulation,  assume that there exists a $j_0 \ge1$ so that $\fy^j_L \neq 0$ and $- t_n^{j_0}/ \la_{n}^{ j_0} \to + \infty$. By Claim~\ref{c:3} and after reordering the profiles we can assume that 
\begin{align*}
 \fy^{j_0+1}_L(t)  =  \fy_L^{j_0}(-t) \mand  t_{n}^{j_0+1} = - t_n^{j_0}, \, \, \, \la_n^{j_0+1} = \la_n^{j_0}.
\end{align*} 
Recall that in Proposition~\ref{a prop} the time sequence $\tau_n$ was chosen so that for every $n$ we have 
\begin{align*}
 \int_0^\I  \dot a^2( \tau_n, r) \, r \, dr \le \frac{1}{n}.
\end{align*}   
Our first observation is that there is considerable flexibility in the choice of $\tau_n$ in Proposition~\ref{a prop}. In fact, we claim that there exists a number $\tau_0 \in (0, 1]$ so that 
\EQ{ \label{tau0}
 \int_0^\I  \dot a^2( \tau_n + \la^{j_0}_n \tau_0, r) \, r \, dr  \to 0 \mas n \to \infty.
 }
To prove~\eqref{tau0}, we first show that there exists a sequence $\e_n  \to 0$ so that 
\EQ{
 \frac{1}{\la_n^{j_0}} \int_{\tau_n}^{\tau_n + \la^{j_0}_n} \int_0^\infty \dot a^2(t, r) \, r \, dr \, dt  = \e_n.
 }
 Recalling that $\vec a(t)  = \vec \psi(t) - \vec \fy(t)$ and using the global regularity of $\fy$ we see that it suffices to show that 
 \EQ{
  \frac{1}{\la_n^{j_0}} \int_{\tau_n}^{\tau_n + \la^{j_0}_n} \int_0^{1-t} \dot  \psi^2(t, r) \, r \, dr \, dt  =  o_n(1) \mas n \to \infty.
 }
 Note from the proof of Proposition~\ref{a prop} that $ \tau_n \in [t_n, t_n + \la_n]$, where $t_n$ is as in Corollary~\ref{t dec extract}. We also have $ \tau_n + \la_n^{j_0} < 1$. From this we infer that 
\begin{align*}
 \tau_n + \la_n^{j_0} \le t_n + \min \{ 1- t_n, \la_n^{j_0} + \la_n\}.
 \end{align*} 
 Setting $\sigma_n = \min \{ 1-t_n, \la^{j_0}_n + \la_n\}$ we see that 
 \begin{align*}
   \frac{1}{\la_n^{j_0}} \int_{\tau_n}^{\tau_n + \la^{j_0}_n} \int_0^{1-t} \dot  \psi^2(t, r) \, r \, dr \, dt & \le \frac{1}{\la_n^{j_0}} \int_{t_n}^{t_n + \sigma_n} \int_0^{1-t} \dot  \psi^2(t, r) \, r \, dr \, dt \\
   & \lesssim \frac{1}{\sigma_n} \int_{t_n}^{t_n + \sigma_n} \int_0^{1-t} \dot  \psi^2(t, r) \, r \, dr \, dt = o_n(1)  ,
   \end{align*}
   where we have used Corollary~\ref{t dec extract} and~\eqref{diff scales} in the last line. 
   
   Next, let 
\begin{align*}
E_n:=   \Big\{  \tau \in [0, 1]   \,\,   | \, \,  \int_0^\I  \dot a^2( \tau_n + \la^{j_0}_n \tau, r) \, r \, dr   \ge  \e^{\frac{1}{4}}_n \Big \}.
\end{align*} 
 We have 
 \begin{align*}
  \e_n   &= \frac{1}{\la_n^{j_0}} \int_{\tau_n}^{\tau_n + \la^{j_0}_n} \int_0^\infty \dot a^2(t, r) \, r \, dr \, dt  
    =  \int_0^1  \int_0^\infty  \dot a^2( \tau_n + \la_n^{j_0} t, r) \, r \, dr \, dt . \\
  & \ge \abs{E_n} \e^{\frac{1}{2}}_n
  \end{align*}
  This implies that $ \abs{E_n} \to 0 $ as $n \to \infty$. Passing to a subsequence, we can assume that $| E_n| \le 2^{-n-2}$ so that 
  \EQ{
   \abs{ \bigcup_{n \ge 0} E_n }  \le \frac{1}{2} .
   }
   It follows that $50\%$ of all  $ \tau_0 \in (0, 1]$ satisfy~\eqref{tau0}. Choosing any such $\tau_0$ proves~\eqref{tau0}. 
   
   Now, recall the from the definition of $\vec b_n$ we have 
 \begin{align} \label{prof for psi1} 
\vec  \psi(\tau_n)  = Q( \cdot/ \la_n) + \vec \fy( \tau_n) + \sum_{j \le k}\fy^j_{L, n}(  0 ) +  \vec \ga_{n}^k
 \end{align}
 where we write $\vec \fy_{n, L}$ for the modulated linear profiles, i.e., 
 \begin{align*}
  \vec \fy_{L, n}^j(t, r) =  \left(  \fy^j_L \left( \frac{t - t_n^j}{ \la_n^j},  \frac{ r}{ \la_n^j} \right), \, \frac{1}{\la_n^j} \dot \fy_L^j \left( \frac{t - t_n^j}{ \la_n^j},  \frac{ r}{ \la_n^j} \right) \right)
  \end{align*}
    Using~\eqref{aln b},~\eqref{diff scales} and  \cite[Appendix $B$]{DKM1}, choose a sequence $\ba \la_n \to 0$ such that 
\begin{align*}
&\ba \la_n \ll \al_n \la_n, \quad 
\la_n  \ll  \ba\la_n \ll \la_n^{j_0}\\
&\ba \la_n \ll \la_n^j \, \, \textrm{or} \,\, \la_n^j \ll \ba\la_n \quad \forall j>1.
\end{align*}
Set
$
\ba \be_n= \frac{\ba \la_n}{\la_n} \to \infty
$
and we note that $\ba\be_n  \ll \al_n$ and $\ba\la_n = \ba \be_n  \la_n$. Therefore, up to replacing $\ba \be_n$ by a sequence $\ti \be_n \simeq  \ba \be_n$  and $\ba \la_n$ by $\ba{\ba \la}_n:= \ti \be_n \la_n$, we have by Lemma~\ref{tech lem} and a slight abuse of notation that
\begin{align}\label{psi ti la to pi1}
\psi(\tau_n, \ba \la_n) \to \pi \quad \textrm{as} \quad n \to \infty.
\end{align}
We define the set 
\begin{align*} 
\J_{\textrm{ext}}^1:= \{ j \ge 1 \, \, \vert \, \,  \ba\la_n \ll \la_n^{j}\}.
\end{align*}
Note that by construction $j_0 \in \J_{\textrm{ext}}^1$. 

Next, with $\ba \la_n$ as above we define $ (\ba f_{n, 0}, \ba f_{n, 1}) $ as follows: 
\begin{align*} 
& \ba f_{n, 0}( r) := \begin{cases}  \pi - \frac{\pi- \psi(\tau_n,  \ba \la_n)}{ \ba{\la}_n} r  \quad\textrm{if} \quad 0 \le r \le \ba\la_n\\ \psi(\tau_n, r)  \quad\textrm{if} \quad  \ba\la_n \le r\end{cases}\\
& \ba f_{n, 1}(r) :=  \dot{\psi}(\tau_n, r)
\end{align*}
Then $(\ba f_{n, 0}, \ba f_{n, 1}) \in \HH_{1, 1}$. Now let $\chi \in C^{\infty}_0$ be defined so that $\chi(r) \equiv 1$ for all $r \in [2, \infty)$ and $\textrm{supp}( \chi )\subset [1, \infty)$. We define $\vec{\ba{\psi}}_n = ( \ba{\psi}_{n, 0}, \ba{\psi}_{n, 1}) \in \HH_0$ as follows:
\begin{align*}
&\ba{\psi}_{n, 0} := \chi(2r/ \ti \la_n)( \ba f_{n, 0}( r) - \pi)\\
& \ba{\psi}_{n, 1} : =  \chi(2r/ \ti \la_n)  \ba f_{n, 1}(r)
\end{align*}
By construction for  $n $ large enough we have $\E(\vec{\ti{\psi}}_n) \le C < 2 \E(Q)$ (for a proof of this fact we refer the reader to the proof of Lemma~\ref{2 lim} for a similar arguement which applies verbatim here). %To see this, observe that 
%\begin{align}\label{e ti psi dec}
%\E( \vec{\ti{\psi}}_n) = \E_{\ti\la_n/2}^{\ti \la_n}( \vec{\ti{\psi}}_n) + \E_{\ti \la_n}^{\infty}( \vec \psi(\tau_n)).
%\end{align}
%Using \eqref{psi ti la to pi} and \eqref{b R},  we note that we have $\E_{0}^{\ti \la_n}( \vec \psi(\tau_n)) \ge \E(Q)- o_n(1)$ which in turn implies that  
%\begin{align*}
%  \E_{\ti \la_n}^{\infty}( \vec \psi(\tau_n)) \le \eta + o_n(1). 
%  \end{align*}
%  We can again use the fact that $\psi(\ta_n, \ti \la_n) \to \pi$ and \eqref{aln b} to deduce that $ \E_{\ti\la_n/2}^{\ti \la_n}( \vec{\ti{\psi}}_n)= o_n(1)$. Putting these facts into \eqref{e ti psi dec} we obtain the claim since, by assumption, $\eta<2\E(Q)$.  
   %Now, since  $\vec{\ti{\psi}}_n \in \HH_0$ satisfies $\E(\vec{\ti{\psi}}_n) \le C < 2 \E(Q)$,
   It follows from Theorem~\ref{main} that for each $n$, the wave map evolution $\vec{\ba{\psi}}_n(t) \in \HH_{0 }$ of the data $\vec{\ba{\psi}}_n$ is global in time and scatters to zero as $t \to \pm \infty$.  And by the finite speed of propagation, it is immediate that 
for all $t$ such that $0 \le \tau_n+t <1$ we have
\begin{align}\label{outside ela11}
\vec{\ba{\psi}}_n(t, r) + (\pi, 0)= \vec \psi(\tau_n+t, r) \quad \forall r \ge  \ba \la_n + \abs{t}. 
\end{align}
We also define 
\begin{align*}
\vec{\ba{ \ga}}_{n, L}^k(0, r):= \chi(2r/ \ti \la_n) \vec \ga_{n, L}^k(0, r)
\end{align*}
Now observe that we can combine \eqref{prof for psi1} and Proposition~\ref{c16} to  obtain the following decomposition: 
\begin{align}\label{prof for ti psi1}
\vec{\ba{\psi}}_{n}(r) =  \vec \fy(\tau_n, r) +  \sum_{j \in \J_{\textrm{ext}}^1\, , \, j \le k}  \vec \fy^j_{L, n}(0) + \vec{\ba{\ga}}_{n, L}^k(0, r) + o_n(1)
\end{align}
where the $o_n(1)$ above is in the sense of $H\times L^2$. Using Proposition~\ref{nonlin profile},  Lemma~\ref{pert}, and Lemma~\ref{en orth lem} we can find a corresponding nonlinear profile decomposition 
\EQ{ \label{npbap}
 \vec{\ba \psi}_n(t, r)  =  \vec \fy(\tau_n + t, r)  + \sum_{j \in \J_{\textrm{ext}}^1\, , \, j \le k}  \vec \fy^j_{n}(t, r) + \vec{\ba{\ga}}_{n, L}^k(0, r) +  \vec{ \ba \theta}^k_n(t, r)
}
where 
\begin{align*}
  \lim_{k \to \infty} \limsup_{n \to \infty}   \left\| \vec{\ba{\theta}}^k_n \right\|_{ L^{\infty}(H \times L^2)}  = 0
\end{align*}
For the precise details on how to deduce~\eqref{npbap} we again refer to the proof of Lemma~\ref{2 lim}. 

Next, we evaluate~\eqref{npbap} at the time $t =   \la_{n}^{j_0} \tau_0$ note that one can  extract a~\emph{linear} profile decomposition $( \vec{\ba{V}}_{L}^j, \ba t_n^j, \ba \la_n^j)$ from the sequence $\vec{\ba{\psi}}(  \la_n^{j_0} \tau_0)$ where the parameters are given by  
\EQ{
  \ba t_n^j  = t_n^j  - \la_n^{j_0} \tau_0, \quad \ba \la^j_n = \la_n^j
  }
  Note that the profiles corresponding to the indices $j_0$ and $j_0+1$ are precisely $ \ba \fy^{j_0}_L(t) = \fy^{j_0}_L(t)$ and $ \ba \fy^{j_0+1}_L(t) = \fy^{j_0+1}_L(t) = \fy^{j_0}_L(-t)$. In addition to this we note that by~\eqref{outside ela11} 
  \begin{align*}
  \vec{\ba{\psi}}_n( \la^{j_0}_n \tau_0, r) + (\pi, 0)= \vec \psi(\tau_n+ \la^{j_0}_n \tau_0, r) \quad \forall r \ge  \ba \la_n +  \la^{j_0}_n \tau_0.  
\end{align*}
Next we apply  Claim~\ref{c:2} with $A_n =  \ba \la_n/  \la_{n}^{j_0} + \tau_0$ and $t_n = t_n^{j_0}/ \la_{n}^{j_0}$ and 
\begin{align*}
(f_n, g_n) =  (\ba \psi( \la_n^{j_0} \tau_0, \la_n \cdot), \frac{1}{\la_n^{j_0}}  \p_t \ba \psi( \la^{j_0}_n \tau_0,  \la^{j_0}_n)). 
\end{align*}
By our choice of $\ba \la_n$ we see that $ \abs{t_n}/ A_n \to \infty$ and hence 
\begin{align*}
 \textrm{weak}-\lim_{n \to \infty} \vec S(t_n^{j_0}/ \la_{n}^{j_0})(f_n, g_n)   &=  \textrm{weak}-\lim_{n \to \infty}  \vec S( \tau_0) \vec S( \ba t_n^{j_0}/  \ba \la_{n}^{j_0})(f_n, g_n)  \\ &= ( \fy_L^{j_0}( \tau_0),  \p_t \fy_L^{j_0}( \tau_0))  
\end{align*}
as well as 
\begin{multline*}
 \textrm{weak}-\lim_{n \to \infty} \vec S(- t_n^{j_0}/ \la_{n}^{j_0})(f_n, g_n)   =  \textrm{weak}-\lim_{n \to \infty}  \vec S( \tau_0) \vec S( \ba t_n^{j_0+1}/  \ba \la_{n}^{j_0+1})(f_n, g_n)  \\ 
 =  ( \ba \fy^{j_0+1}_L( \tau_0),   \p_t  \ba \fy^{j_0+1}_L( \tau_0)) = ( \fy_L^{j_0}(- \tau_0),  -\p_t \fy_L^{j_0}( -\tau_0))  
\end{multline*}
But the above implies that 
\begin{align*}
\fy^{j_0}_L(t) = \fy^{j_0}_L( t+ 2 \tau_0). 
\end{align*}
Since $\vec \fy^j_L$ is a solution to the linear wave equation the above implies that $\fy^{j_0}_L$ can only be identically $0$, which contradicts the assumption that $\fy_L^{j_0}$ is nonzero.
\end{proof} 

Now, using Corollary~\ref{c:vt0} we can rewrite our profile decomposition %as follows: 
%\begin{align}
%&u_{n, 0}(r) = \sum_{j \le k}\frac{1}{\la_n^j} V_L^j\left(0, \frac{r}{\la_n^j}\right) + w_{n, 0}^k( r)\\
%&u_{n, 1}(r) = o_n(1) \quad \textrm{in} \quad L^2(\R^4)
%\end{align}
%At this point it is convenient to rephrase the above profile decomposition 
in the $2d$ formulation as follows. %We have 
\begin{align}\label{2d bg1}
&b_{n, 0}(r)= \sum_{j \le k}\fy^j\left(0, \frac{r}{\la_n^j}\right) + \ga_{n}^k( r)\\
&b_{n, 1}(r) = o_n(1) \quad \textrm{in} \quad L^2\label{2d bg2},
\end{align}
where 
\begin{align*} 
&\fy^j\left(0, \frac{r}{\la_n^j}\right):=\frac{r}{\la_n^j} V_L^j\left(0, \frac{r}{\la_n^j}\right)    \\
&\gamma_{n}^k(r):=rw_{n, 0}^k( r) .
\end{align*}
Note that in addition to the Pythagorean expansions given in \eqref{pyt} we also have the following almost-orthogonal decomposition of the nonlinear energy given by Lemma~\ref{en orth lem}:
\begin{align} \label{nonlin en}
\E(\vec b_n) = \sum_{j\le k} \E(\fy^j(0), 0) + \E(\ga_n^k, 0) + o_n(1).
\end{align}
Note that $\fy^j, \ga_n^k \in \HH_0$ for every $j$, for every $n$, and for every $k$.  Using the fact that $\E(\vec b_n) \le C< 2\E(Q)$, \eqref{nonlin en} and Theorem~\ref{main} imply that,  for every $j$, the nonlinear wave map evolution of the data $(\fy^j(0, r/ \la_n^j), 0)$  given by
\begin{align} 
\vec \fy_n^j(t, r) = \left (\fy^j\left(\frac{t}{\la_n^j}, \frac{r}{\la_n^j}\right)\, , \, \frac{1}{\la_n^j} \dot{\fy}^j\left(\frac{t}{\la_n^j}, \frac{r}{\la_n^j}\right)\right)
\end{align}
is global in time and scatters as $t\to \pm \infty$. Moreover we have the following nonlinear profile decomposition given by  Proposition~\ref{nonlin profile}:
\begin{align} \label{nonlin prof}
\vec b_n(t, r)  = \sum_{j \le k} \vec \fy_n^j(t, r) + \vec \ga_{n, L}^k(t, r) + \vec \theta_n^k(t, r)
\end{align} 
where $\vec b_n(t, r)$ are the global wave map evolutions of the data $\vec b_n$ and  $\vec \ga_{n,  L}^k(t, r)$ is the linear evolution of $(\ga_n^k, 0)$.  Finally, by  \eqref{nonlin error}, we have
\begin{align} 
& \limsup_{n\to \infty}\left(\|\vec \theta_n^k\|_{L^{\infty}_t(H\times L^2)} + \left\| \frac{1}{r} \theta_n^k\right\|_{L^3_t(\R ; L^6_x(\R^4))}\right) \to 0 \quad \textrm{as} \quad k \to \infty\label{theta 0}.
\end{align}
Now, recall that our goal is to prove that $\fy^j=0$ for every $j$. %We can make the following crucial observation about the scales $\la_n^j$. By Proposition~\ref{a prop} we have as $n \to \infty$ that
%\begin{align} \label{aln b}
%\E_0^{\al_n\la_n}(b_{n, 0}, 0) \to 0,\\
%\E_{1-\tau_n}^{\infty}(b_{n, 0}, 0) \to 0\label{1-tn b}.
%\end{align}
%Note that we also have that if $\be_n \to \infty$ is any other sequence with $\be_n \le \al_n$ then  
%\begin{align} \label{ben b}
%\E_0^{\be_n\la_n}(b_{n, 0}, 0) \to 0.
%\end{align}
%We can combine ~\eqref{aln b} and \eqref{1-tn b} with  Proposition~\ref{c16} to conclude that for each scale $\la_n^j$ corresponding to a nonzero profile $\fy^j$ we have 
%\begin{align} \label{diff scales}
%\la_n \ll \la_n^j \le 1-\tau_n
%\end{align} 
%at least for $n$ large. In particular,  
%\begin{align}\label{la to 0}
%\la_n^j \to 0 \quad \textrm{as} \quad n \to \infty \quad \textrm{for every} \, j.
%\end{align}
 Let $k_0$ be the index corresponding to the first nonzero profile $\fy^{k_0}$. Without loss of generality, we can assume that $k_0=1$. Using \eqref{aln b},  \eqref{diff scales} and  \cite[Appendix $B$]{DKM1}, we can find a sequence $\ti \la_n \to 0$ such that 
\begin{align*}
&\ti \la_n \ll \al_n \la_n\\
&\la_n  \ll  \ti \la_n \ll \la_n^1\\
&\ti \la_n \ll \la_n^j \, \, \textrm{or} \,\, \la_n^j \ll \ti \la_n \quad \forall j>1.
\end{align*}
Now define
\begin{align*}
\be_n= \frac{\ti \la_n}{\la_n} \to \infty
\end{align*}
and we note that $\be_n  \ll \al_n$ and $\ti\la_n = \be_n  \la_n$. Therefore, up to replacing $\be_n$ by a sequence $\ti \be_n \simeq \be_n$  and $\ti \la_n$ by $\ti{\ti \la}_n:= \ti \be_n \la_n$, we have by Lemma~\ref{tech lem} and a slight abuse of notation that
\begin{align}\label{psi ti la to pi}
\psi(\tau_n, \ti \la_n) \to \pi \quad \textrm{as} \quad n \to \infty.
\end{align}
We define the set 
\begin{align*} 
\J_{\textrm{ext}}:= \{ j \ge 1 \, \, \vert \, \,  \ti\la_n \ll \la_n^j\}.
\end{align*}
Note that by construction $1 \in \J_{\textrm{ext}}$. The next step consists of establishing the following claim: 
\begin{lem} \label{2 lim} Let $\fy^1$, $\la_n^1$ be defined as above. Then for all $\e>0$ we have 
\begin{align}\label{p1}
&\frac{1}{\la_n^1} \int_0^{\la_n^1} \int_{ \e \la_n^1 + t}^{\infty}  \abs{  \sum_{ j \in \J_{\textrm{ext}} \, , j \le k }  \dot\fy_n^j(t, r) + \dot \ga_{n, L}^k(t, r)}^2 \, r\, dr\, dt= o^k_n 
\end{align}
where $\ds{\lim_{k\to \infty} \limsup_{n \to \infty} o^k_n =0}$. 
Also, for all $j>1$ and for all $\e>0$ we have 
\begin{align} \label{p2}
\frac{1}{\la_n^1} \int_0^{\la_n^1} \int_{ \e \la_n^1 + t}^{\infty} ( \dot \fy_n^j)^2(t, r) \, r\, dr dt \to 0 \quad \textrm{as} \quad n \to \infty.
\end{align}
\end{lem}
Note that \eqref{p1} and \eqref{p2} together directly imply the following result: 
\begin{cor}\label{p3}Let $\fy^1$ be as in Lemma~\ref{2 lim}. Then for all $\e>0$ we have 
\begin{align} 
\frac{1}{\la_n^1} \int_0^{\la_n^1} \int_{\e \la_n^1 + t}^{\infty}  \abs{   \dot\fy_n^1(t, r) + \dot \ga_{n, L}^k(t, r)}^2 \, r\, dr\, dt =o_n^k
\end{align}
where  $\ds{\lim_{k\to \infty} \limsup_{n \to \infty} o^k_n =0}$. 
\end{cor}

\begin{proof}[Proof of Lemma~\ref{2 lim}]
We begin by proving \eqref{p1}. First recall that by the definition of $\vec b_n$ we have the following decomposition 
\begin{align}\label{prof for psi}
\vec \psi(\tau_n, r) = (Q(r/ \la_n), 0) +  \vec \fy(\tau_n, r) +  \sum_{j \le k} (\fy^j(0, r/ \la_n^j), 0) + \vec \ga_{n, L}^k(0, r)
\end{align}
Next, with $\ti \la_n$ as above we define $\vec f_n= (f_{n, 0},f_{n, 1}) $ as follows: 
\begin{align*} 
& f_{n, 0}( r) := \begin{cases}  \pi - \frac{\pi- \psi(\tau_n,  \ti \la_n)}{ \ti{\la}_n} r  \quad\textrm{if} \quad 0 \le r \le \ti\la_n\\ \psi(\tau_n, r)  \quad\textrm{if} \quad  \ti\la_n \le r\end{cases}\\
& f_{n, 1}(r) :=  \dot{\psi}(\tau_n, r)
\end{align*}
Then $\vec f_n \in \HH_{1, 1}$. Now let $\chi \in C^{\infty}_0$ be defined so that $\chi(r) \equiv 1$ for all $r \in [2, \infty)$ and $\textrm{supp}( \chi )\subset [1, \infty)$. We define $\vec{\ti{\psi}}_n = ( \ti{\psi}_{n, 0}, \ti{\psi}_{n, 1}) \in \HH_0$ as follows:
\begin{align*}
&\ti{\psi}_{n, 0} := \chi(2r/ \ti \la_n)( f_{n, 0}( r) - \pi)\\
& \ti{\psi}_{n, 1} : =  \chi(2r/ \ti \la_n) f_{n, 1}(r)
\end{align*}
We claim that for $n $ large enough we have $\E(\vec{\ti{\psi}}_n) \le C < 2 \E(Q)$. To see this, observe that 
\begin{align}\label{e ti psi dec}
\E( \vec{\ti{\psi}}_n) = \E_{\ti\la_n/2}^{\ti \la_n}( \vec{\ti{\psi}}_n) + \E_{\ti \la_n}^{\infty}( \vec \psi(\tau_n)).
\end{align}
Using \eqref{psi ti la to pi} and \eqref{b R},  we note that we have $\E_{0}^{\ti \la_n}( \vec \psi(\tau_n)) \ge \E(Q)- o_n(1)$ which in turn implies that  
\begin{align*}
  \E_{\ti \la_n}^{\infty}( \vec \psi(\tau_n)) \le \eta + o_n(1). 
  \end{align*}
  We can again use the fact that $\psi(\ta_n, \ti \la_n) \to \pi$ and \eqref{aln b} to deduce that $ \E_{\ti\la_n/2}^{\ti \la_n}( \vec{\ti{\psi}}_n)= o_n(1)$. Putting these facts into \eqref{e ti psi dec} we obtain the claim since, by assumption, $\eta<2\E(Q)$.  
  
  Now, since  $\vec{\ti{\psi}}_n \in \HH_0$ satisfies $\E(\vec{\ti{\psi}}_n) \le C < 2 \E(Q)$, Theorem~\ref{main} implies that for each $n$, the wave map evolution $\vec{\ti{\psi}}_n(t) \in \HH_{0 }$ of the data $\vec{\ti{\psi}}_n$ is global in time and scatters to zero as $t \to \pm \infty$.  And by the finite speed of propagation, it is immediate that 
for all $t$ such that $0 \le \tau_n+t <1$ we have
\begin{align}\label{outside ela1}
\vec{\ti{\psi}}_n(t, r) + (\pi, 0)= \vec \psi(\tau_n+t, r) \quad \forall r \ge \e \la_n^1 + \abs{t}
\end{align}
as long as  $n $ is large enough to ensure that $\ti \la_n \le \e \la_n^1$. We also define 
\begin{align*}
\vec{\ti{ \ga}}_{n, L}^k(0, r):= \chi(2r/ \ti \la_n) \vec \ga_{n, L}^k(0, r)
\end{align*}
Now observe that we can combine \eqref{prof for psi} and Proposition~\ref{c16} to to obtain the following decomposition: 
\begin{align}\label{prof for ti psi}
\vec{\ti{\psi}}_{n}(r) =  \vec \fy(\tau_n, r) +  \sum_{j \in \J_{\textrm{ext}}\, , \, j \le k} (\fy^j(0, r/ \la_n^j), 0) + \vec{\ti{\ga}}_{n, L}^k(0, r) + o_n(1)
\end{align}
where the $o_n(1)$ above is in the sense of $H\times L^2$. By Lemma~\ref{local scat lem} we have that 
\begin{align*}
\limsup_{n \to \infty}\left\| \frac{1}{r} \ti{\ga}_{n, L}^k\right\|_{L^3_t L^6_x(  \R^{1+4})}  \to 0 \quad \textrm{as} \quad k  \to \infty
\end{align*}
since if the above did not hold we could find subsequences $n_{\ell}$ and $k_{\ell}$ such that for all $\ell$ we have
\begin{align*}
\left\| \frac{1}{r} \ti{\ga}_{n_{\ell}, L}^{k_{\ell}}\right\|_{L^3_t L^6_x(  \R^{1+4})}  \ge \e  \quad \textrm{and} \quad \lim_{\ell \to\infty}\left\| \frac{1}{r} \ga_{n_{\ell}, L}^{k_{\ell}}\right\|_{L^3_t L^6_x(  \R^{1+4})}  =0
\end{align*}
which would directly contradict Lemma~\ref{local scat lem}. Hence, if we ignore the $o_n(1)$ term,  the right-hand side of \eqref{prof for ti psi} is a profile decomposition in the sense of Corollary~\ref{bg wm}. Therefore, by Proposition~\ref{nonlin profile}, and Lemma~\ref{pert},  we can find $\vec{\ti{\theta}}_n^k(t, r)$ with 
\begin{align*}
\lim_{k \to \infty}\limsup_{n \to \infty}\left\|\vec{\ti{\theta}}_n^k(t, r)\right\|_{L^{\infty}_t (H \times L^2)} =0
\end{align*}
such that the following nonlinear profile decomposition holds: 
\begin{align}\label{nonlin prof ti psi}
\vec{\ti{\psi}}_{n}(t, r) =  \vec \fy(\tau_n+t, r) +  \sum_{j \in \J_{\textrm{ext}}\, , \, j \le k} \vec \fy_n^j(t, r)+ \vec{\ti{\ga}}_{n, L}^k(t, r) + \vec{\ti{\theta}}_n^k(t, r)
\end{align}
To be precise, \eqref{nonlin prof ti psi} is proved as follows: Define 
\begin{align}\label{prof for breve psi}
\vec{\breve{\psi}}_{n}(r) =  \vec \fy(\tau_n, r) +  \sum_{j \in \J_{\textrm{ext}}\, , \, j \le k} (\fy^j(0, r/ \la_n^j), 0) + \vec{\ti{\ga}}_{n, L}^k(0, r) 
\end{align}
As mentioned above, this is a profile decomposition in the sense of Corollary~\ref{bg wm} and $\E(\vec{\breve{\psi}}_{n}) < C \le 2\E(Q)$. By Proposition~\ref{nonlin profile} we then have the following nonlinear profile decomposition for the wave maps evolutions $\vec{\breve{\psi}}_{n}(t, \cdot) \in \HH_0$: 
\begin{align*}
&\vec{\breve{\psi}}_{n}(t, r)= \vec \fy(\tau_n+t, r) +  \sum_{j \in \J_{\textrm{ext}}\, , \, j \le k} \vec \fy_n^j(t, r)+ \vec{\ti{\ga}}_{n, L}^k(t, r) + \vec{\breve{\theta}}_n^k(t, r)\\
&\lim_{k \to \infty}\limsup_{n \to \infty}\left\|\vec{\breve{\theta}}_n^k(t, r)\right\|_{L^{\infty}_t (H \times L^2)} =0
\end{align*}
Now, by our perturbation theory, i.e., Lemma~\ref{pert}, we can deduce~\eqref{nonlin prof ti psi} since $\|\vec{\breve{\psi}}_{n}(0)- \vec{\ti{\psi}}_{n}(0)\|_{H \times L^2} = o_n(1)$. 

Next, we combine \eqref{nonlin prof ti psi} with \eqref{outside ela1} to conclude that
\begin{align*}
\vec{\psi}(\tau_n+t, r) -(\pi, 0) - \vec \fy(\tau_n+t, r) = \sum_{j \in \J_{\textrm{ext}}, \,  j \le k} \vec \fy_n^j(t, r)+ \vec{\ga}_{n, L}^k(t, r) + \vec{\ti{\theta}}_n^k(t, r)
\end{align*}
for all $t +\tau_n<1 $ and $r \ge \e \la_n^1+t$ for $n$ large enough so that $\ti \la_n \le \e \la_n^1$. Using the above we can finally conclude that
\begin{multline}
\frac{1}{\la_n^1} \int_0^{\la_n^1} \int_{ \e \la_n^1 + t}^{\infty}  \abs{  \sum_{ j \in \J_{\textrm{ext}} \, , j \le k }  \dot\fy_n^j(t, r) + \dot \ga_{n, L}^k(t, r)}^2 \, r\, dr\, dt\\
\le \frac{1}{\la_n^1} \int_0^{\la_n^1} \int_{ \e \la_n^1 + t}^{\infty}  \dot{a}^2(\tau_n+t, r)\, r\, dr\, dt + o_n^k\\
\le  \frac{1}{\la_n^1} \int_0^{\la_n^1} \int_0^{\infty}  \dot{a}^2(\tau_n+t, r)\, r\, dr\, dt + o_n^k \\
=  \frac{1}{\la_n^1} \int_{\tau_n}^{\tau_n+\la_n^1} \int_0^{\infty}  \dot{a}^2(t, r)\, r\, dr\, dt + o_n^k\\
 \le \frac{1}{\la_n^1} \int_{\tau_n}^{\tau_n+\la_n^1} \int_0^{1-t}  \dot{\psi}^2(t, r)\, r\, dr\, dt +  \sup_{t \ge \tau_n} \E_0^{1-t}(\vec \fy(t)) + o_n^k=o_n^k.
\end{multline}
To justify the last line above we need to show that 
\begin{align*}
 \frac{1}{\la_n^1} \int_{\tau_n}^{\tau_n+\la_n^1} \int_0^{1-t}  \dot{\psi}^2(t, r)\, r\, dr\, dt =o_n(1)
 \end{align*}
 On the one hand,  by our construction in the proof of Proposition~\ref{a prop} we have $\tau_n \in [t_n, t_n + \la_n]$ where $t_n$ is as in  Corollary~\ref{t dec extract} and Theorem~\ref{str}. On the other hand, note that $\tau_n + \la_n^1 <1$. Putting these facts together we infer that 
 \begin{align*} 
 \tau_n + \la_n^1  \le t_n + \min\{1-t_n,  \la_n^1 + \la_n\}
 \end{align*}
 Therefore, if we define $\sigma:= \min\{1-t_n,  \la_n^1 + \la_n\}$ we have
 \begin{align*}
 \frac{1}{\la_n^1} \int_{\tau_n}^{\tau_n+\la_n^1} \int_0^{1-t}  \dot{\psi}^2(t, r)\, r\, dr\, dt & \le \frac{1}{\la_n^1} \int_{t_n}^{t_n+ \sigma} \int_0^{1-t}  \dot{\psi}^2(t, r)\, r\, dr\, dt \\
 &\lesssim \frac{1}{\sigma} \int_{t_n}^{t_n+ \sigma} \int_0^{1-t}  \dot{\psi}^2(t, r)\, r\, dr\, dt  
 = o_n(1)
 \end{align*}
 where the last line above follows from Corollary~\ref{t dec extract}. Note that we have used  the fact that  $\la_n \ll \la_n^1$ in the second inequality above.
This proves~\eqref{p1}.

Next we prove \eqref{p2}. Recall that for $j \neq 1$ we have either $\mu_n^j:=\frac{\la_n^1}{\la_n^j} \to 0$ or $\mu_n^j \to \infty$. Suppose the former occurs. Then 
\begin{align*} 
 \frac{1}{\la_n^1} \int_0^{\la_n^1} \int_{ 0}^{\infty} ( \dot \fy_n^j)^2(t, r) \, r\, dr dt&= \frac{1}{\la_n^1} \int_0^{\la_n^1} \int_{ 0}^{\infty} \frac{1}{(\la_n^j)^2}( \dot \fy^j)^2\left(\frac{t}{\la_n^j},\frac{r}{\la_n^j}\right) \, r\, dr dt\\
 &= \frac{1}{\la_n^1} \int_0^{\la_n^1} \int_{ 0}^{\infty}( \dot \fy^j)^2\left(\frac{t}{\la_n^j},r \right) \, r\, dr dt\\
 &=\frac{1}{\mu_n^j}\int_0^{\mu_n^1} \int_{ 0}^{\infty}( \dot \fy^j)^2\left(t,r \right) \, r\, dr dt\\
 &\longrightarrow  \int_{ 0}^{\infty}( \dot \fy^j)^2\left(0,r \right) \, r\, dr dt =0
 \end{align*}
 Now suppose that $\mu_n^j \to \I$.  Then, changing variables  as above, we have  
 \begin{align} \label{p2int}
  \frac{1}{\la_n^1} \int_0^{\la_n^1} \int_{\e \la_n^1 + t}^{\infty} ( \dot \fy_n^j)^2(t, r) \, r\, dr dt&=\frac{1}{\mu_n^j}\int_0^{\mu_n^1} \int_{\e \mu_n^j + t}^{\infty}( \dot \fy^j)^2\left(t,r \right) \, r\, dr dt
  \end{align}
Now note that by monotonicity of the energy on exterior cones we have that  for all $\de>0$ there exists $M>0$ such that for all $t \in [0, \infty)$ we have 
 \begin{align*} 
 \int_{M+t}^{\infty} ( \dot \fy^j)^2\left(t,r \right) \, r\, dr  \le \de
 \end{align*} 
 This implies that the right-hand side of \eqref{p2int} tends to $0$ as $n \to \infty$. 
 \end{proof}
We can now conclude the proof Proposition~\ref{no prof}. 
\begin{proof}[Proof of Proposition~\ref{no prof}]
We first show that all of the profiles $\fy^j$ in the decomposition \eqref{2d bg1} must be identically $0$. We argue by contradiction. As above we assume that $\fy^1 \neq 0$. By Corollary~\ref{p3} we know that for all $\e>0$ we have 
\begin{align*} 
\frac{1}{\la_n^1} \int_0^{\la_n^1} \int_{\e \la_n^1 + t}^{\infty}  \abs{   \dot\fy_n^1(t, r) + \dot \ga_{n, L}^k(t, r)}^2 \, r\, dr\, dt =o_n^k
\end{align*}
as $n \to \infty$ for any $k>1$. Changing variables this implies that 
\begin{align}\label{sum lim}
 \int_0^{1} \int_{\e + t}^{\infty}  \abs{   \dot\fy^1(t, r) +\la_n^1 \dot \ga_{n, L}^k(\la_n^1 t, \la_n^1r)}^2 \, r\, dr\, dt =o^k_n
\end{align}
Now consider the mapping $H \times L^2 \to \R$ defined by 
\begin{align*}
(f_0, f_1) \mapsto \int_0^1 \int_{\e +t}  \dot\fy^1(t, r) \dot{f}(t, r) \, r \, dr \, dt
\end{align*} 
where $\vec f(t, r)$ is the solution to the linear wave equation 
\begin{align*} 
f_{tt} -f_{rr}-\frac{1}{r} f_{r} + \frac{1}{r^2}f =0
\end{align*}
with initial data $(f_0, f_1)$. This is a continuous linear functional on $H \times L^2$. Now, by \eqref{weak} we have 
\begin{align*} 
\left(\gamma_{n, L}^k( \la_n^1\cdot), \la_n^1 \dot \ga_{n,L}^k( \la_n^1 \cdot)\right) \rightharpoonup 0 \, \,\textrm{in}\, \, H \times L^2 \quad \textrm{as} \quad n \to \infty 
\end{align*}
Hence, for all $\e>0$ we have
\begin{align*} 
\lim_{n\to \infty}\int_0^{1} \int_{\e + t}^{\infty}     \dot\fy^1(t, r) \la_n^1 \dot \ga_{n, L}^k(\la_n^1 t, \la_n^1r) \, r\, dr\, dt = 0
\end{align*}
Combining the above line with \eqref{sum lim} we conclude that for all $\e>0$ we have 
\begin{align*}
 \int_0^{1} \int_{\e + t}^{\infty}  \abs{   \dot\fy^1(t, r)}^2  \, r\, dr\, dt =0
\end{align*}
Letting $\e$ tend to $0$ we obtain
\begin{align*}
 \int_0^{1} \int_{ t}^{\infty}  \abs{   \dot\fy^1(t, r)}^2  \, r\, dr\, dt =0
\end{align*}
Therefore $\dot \fy^1(t, r) = 0$ if $r \ge t$ and $t \in [0, 1]$. Let $\Om$ denote the region in $[0,1] \times\R^2$ exterior to the light cone 
\begin{align*} 
\Om = \{ (t, x) \in[0,1] \times \R^2 \, \vert \, \abs{x} \ge t\}
\end{align*}
If we let $U^1(t, x) = ( \fy^1(t, r), \om)$ denote the full equivariant wave map (here $x=(r, \om)$ in polar coordinates on $\R^2$) then we have $(t, x) \in \Om \Rightarrow  U^1(t, x) = U_0^1(x)$. Hence $U_0^1(x)$ is a finite energy equivariant harmonic map on $\R^2-\{0\}$. By Sacks-Uhlenbeck  \cite{SU} we can extend $U_0^1$ to a smooth equivariant harmonic map from $\R^2 \to S^2$. But since $\fy^1 \in \HH_0$, $U_0^1$ must  be identically equal to $0$, since $0$ is the unique harmonic map in the topological class $\HH_0$. But this contradicts the fact that we assumed $\fy^1 \neq 0$. 

To complete the proof of Proposition~\ref{no prof} we note that we have now concluded that all the profiles in the decomposition \eqref{2d bg1} must be identically zero. Hence, we have $\ga_n^k(r) = b_{n, 0}(r)$,  $\vec \ga_{n, L}^k=:b_{n, L}$, and $\vec \theta_n^k= \vec \theta_n$ and we can rewrite \eqref{nonlin prof} as follows:
\begin{align} 
\vec b_n(t, r) =  \vec b_{n, L}(t, r) + \vec \theta_n(t, r)
\end{align} 
 Finally,  \eqref{bL to 0} and \eqref{theta to 0} are satisfied because of \eqref{stric} and \eqref{theta 0}. 
\end{proof}

We can now prove Proposition~\ref{b comp}.

\begin{proof}[Proof of Proposition~\ref{b comp}]
 Assume that Proposition~\ref{b comp} fails. Then up to extracting a subsequence, we can find $\de_0>0$ so that 
 \begin{align}\label{contradict}
 \|b_{n, 0}\|_H \ge \de_0
 \end{align}
for every $n$.

We will show that this implies a concentration of energy at some point $r_0 >0$ and  time $t = 1-r_0 <1$, which is a contradiction with our assumed blow-up time as well as with the fact that equivariance prevents concentration away from  $r=0$.

The key idea is to use that the wave map evolution $\vec b_n(t)$ actually maintains a fixed amount of  energy outside the light cone, %as its linear counterpart does (
as shown in Corollary~\ref{ext en est}. We then prove that this forces $\vec \psi$ to concentrate energy on the boundary of the cone.  For this, we proceed in two steps, both requiring evolving a nonlinear profile decomposition backwards in time. First, we show that  the evolutions of $\vec b_n$ and $\vec \psi(\tau_n)$ remain close on an exterior region during a time-scale on which we can control the rescaled harmonic map -- in the sense of %Here we use 
Corollary~\ref{Cote} and Proposition \ref{nonlin profile}.
%(we cannot go beyond that time-scale with the rescaled harmonic map included in the decomposition  due to the conditions in Corollary~\ref{Cote} and Proposition \ref{nonlin profile}). For this step 
%and we evolve the entire decomposition 
%\begin{align}\label{dec0}
%\vec \psi(\tau_n, r) = (Q(r/ \la_n), 0) +  \vec \fy(\tau_n, r) +  \vec b_n(r) 
%\end{align}
%nonlinearly, backwards for a short time and prove that the essential features of the decomposition are preserved by the evolution.  

At this point, we focus the analysis outside the light cone: we need to evolve the decomposition past the time-scale on which we can control the harmonic map, but fortunately this large profile does not contribute in this exterior region. In fact, we evolve  the profile decomposition with the harmonic map removed for \emph{all time} exterior to the cone, and infer that some energy remains outside the light cone -- in fact it concentrates on the boundary. %From there, unscaling back to $\vec \psi$, we reach easily the aforementioned contradiction.
As $\vec b_n(t)$ actually maintains a fixed amount of  energy outside the light cone, this means that $\vec \psi(\tau_n+t)$ concentrates energy there as well, and gives us a contradiction.

We now begin with the details of the proof. It is convenient to carry out the argument in rescaled coordinates.
%With this assumption we seek  a contradiction. \Red{Indeed, we will use Corollary~\ref{ext en est} to show that \eqref{contradict} forces $\vec \psi$ to concentrate a nonzero amount of energy at a point $r_0>0$, and at time $t=1-r_0$, which is impossible both by equivariance and our assumption that the blow-up time is $t=1$. It is convenient to carry out the argument in rescaled coordinates. } %We begin by rescaling. 
Set 
 \begin{align*} 
 \mu_n:= \frac{\la_n}{1-\tau_n}.
 \end{align*}
 Since $\la_n =o(1-\tau_n)$ we have $\mu_n \to 0$ as $n \to \infty$. Now define the rescaled wave maps
\begin{align*}
&g_n(t, r):= \psi(\tau_n + (1-\tau_n)t, (1-\tau_n)r) \\
&h_n(t, r):= \fy(\tau_n + (1-\tau_n)t, (1-\tau_n)r).
\end{align*}
 Then $\vec g_n(t) \in \HH_1$ is a wave map defined on the interval $[-\frac{\tau_n}{1-\tau_n}, 1)$, and $ \vec h_n(t) \in \HH_0$ is global in time and scatters to $0$. We then have 
 \begin{align*}
 a(\tau_n+(1-\tau_n) t, (1-\tau_n)r) =g_n(t, r)-h_n(t, r).
 \end{align*}
 Similarly, define 
 \begin{align*} 
& \ti b_{n, 0}( r):= b_{n, 0}((1-\tau_n)r)\\
& \ti b_{n , 1}(r):=(1-\tau_n)b_{n, 1}((1-\tau_n)r)
 \end{align*}
 and the corresponding rescaled wave map evolutions 
 \begin{align*}
 &\ti b_n(t, r):= b_n((1-\tau_n)t, (1-\tau_n)r)\\
 &\p_t \ti{b}_n(t, r) := (1-\tau_n) \dot{b}_n((1-\tau_n) t, (1-\tau_n) r).
 \end{align*}
Observe that we have the decomposition
\begin{align}\label{ghqb}
&g_n(0, r)= h_n(0, r) + Q\left(\frac{r}{\mu_n}\right) +\ti b_{n, 0}(r)\\
&\dot{g}_n(0, r)= \dot{h}_n(0, r) + \ti b_{n,1}(r).
\end{align}
Note that by \eqref{supp a} we have $\ti b_{n, 0} =  \pi - Q(\cdot/ \mu_n)$ on $[1, \infty)$ and hence
\begin{align}\label{ti b ext}
\|\ti b_{n, 0}\|_{H(r\ge 1)} \to 0
\end{align}
as $n \to \infty$.

Now, observe that the regularity properties of $\fy(t)$ imply that  
\begin{align}\label{h reg}
\lim_{\rho \to 0}\sup_{n} \|\vec h_n(0)\|_{H \times L^2(r \le \rho/(1-\tau_n))} =0
\end{align}
Hence, for fixed large $K$, (to be chosen precisely later),  we can find $r_0>0$ so that 
\begin{align}\label{h small}
 \sup_n  \|\vec h_n(0)\|_{H \times L^2(r \le \frac{3r_0}{(1-\tau_n)})}  \le \frac{\de_0}{K},
 \end{align}
 where $\de_0$ is as in \eqref{contradict}. Now, recall that  $\al_n \to \infty$ has been fixed. % as in Proposition~\ref{a prop}. 
 Using Lemma~\ref{tech lem} we can choose $\ga_n \to \infty$ with $$\ga_n\ll  \al_n$$ such  that 
 \begin{align*}
 g_n(0, \ga_n \mu_n) \to \pi \quad \textrm{as} \quad n \to \infty
\end{align*}
Now define $\de_n \to 0$ by 
\begin{align*}
\abs{g_n(0, \ga_n \mu_n) -\pi } =: \de_n \to 0
\end{align*}
Finally we choose $\be_n \to \infty$ so that 
\begin{align}
& \be_n \le\min\{\sqrt{\ga_n}, \de_n^{-1/2}, \sqrt{n}\} \notag\\
& g_n(0, \be_n \mu_n/2) \to \pi \quad \textrm{as} \quad n \to \infty \label{gn bemu}
 \end{align}
 
  We make the following claims:
\begin{itemize}
\item[($i$)] As $n \to \infty$ we have 
\begin{align}\label{g-Q}
\|\vec g_n(-\be_n \mu_n/2) - (Q(\cdot/ \mu_n), 0)\|_{H \times L^2( r\le \be_n \mu_n)} \to 0
\end{align}
\item[($ii$)] For each $n$, on the interval $r \in [\be_n \mu_n, \infty)$ we have 
\begin{align}\label{ghbt}
&\vec g_n\left( -\frac{\be_n \mu_n}{2}, r\right) - (\pi, 0) = \vec h_n\left( -\frac{\be_n \mu_n}{2}, r\right) +  \vec{\ti{b}}_n\left( -\frac{\be_n \mu_n}{2}, r\right) \\ \notag
&\quad + \vec{\breve{\theta}}_n\left( -\frac{\be_n \mu_n}{2}, r\right),\\
&\|\vec{\breve{\theta}}_n\|_{L^{\infty}_t( H \times L^2)} \to 0 \notag
\end{align}

\end{itemize}

We first prove \eqref{g-Q}. Note that by Proposition~\ref{a prop} we have 
 \begin{align}\label{b in AB} 
 \|(\ti b_{n,0}, \ti b_{n, 1})\|_{ H \times L^2(r \le \al_n \mu_n)} \le \frac{1}{n} \to 0.
 \end{align}
 Using \eqref{h reg} together with  $\al_n \la_n \le 1- \tau_n \to 0$  as well as \eqref{b in AB} and  the decomposition \eqref{ghqb}  we can then deduce that 
 \begin{align*}
 \|\vec g_n(0)- (Q(\cdot/ \mu_n), 0) \|_{ H \times L^2(r \le \ga_n \mu_n)}  \le \frac{2}{n} \to 0.
 \end{align*}
  Unscale the above by setting $\ti g_n(t, r) =g_n(\mu_n t, \mu_nr)$ and observe that, 
 \begin{align*}
  \|( \ti{g}_n(0), \p_t \ti g_n(0))- (Q(\cdot), 0) \|_{ H \times L^2(r \le \ga_n)} \le \frac{2}{n} \to 0.
  \end{align*}
  Now using Corollary~\ref{Cote} and the finite speed of propagation we claim that we have
  \begin{align} \label{ti g-Q at be}
\|( \ti{g}_n(-\be_n /2), \p_t \ti g_n(-\be_n /2))- (Q(\cdot), 0) \|_{ H \times L^2(r \le \be_n)} =o_n(1) .
\end{align}
 To see this, we need to show that Corollary~\ref{Cote} applies. Indeed define 
 \begin{align*}
 &\hat{g}_{n,0}(r) := \begin{cases}   \pi \quad \textrm{if} \quad r \ge 2\ga_n  \\ \pi+ \frac{\pi - \ti{g}_n(0, \ga_n)}{\ga_n}(r- 2\ga_n) \quad \textrm{if} \quad \ga_n \le r \le 2\ga_n \\ \ti g_n(0, r) \quad \textrm{if} \quad r \le \ga_n
 \end{cases}\\
 &\hat{g}_{n, 1}(r) = \begin{cases} \p_t \ti g_n(0, r)\quad  \textrm{if} \quad r \le \ga_n \\ 0 \quad \textrm{if} \quad r \ge \ga_n\end{cases}
 \end{align*}
 Then, by construction we have $\vec{\hat{g}}_n \in \HH_1$, and since  
 \begin{align*}
 \| \vec{\hat{g}}_n-(\pi, 0)\|_{H \times L^2( \ga_n \le r \le 2\ga_n )} \le C \de_n
 \end{align*}
 we then can conclude that 
 \begin{align*} 
 \|\vec{\hat{g}}_n - (Q, 0)\|_{H \times L^2} &\le  \|\vec{\hat{g}}_n - (Q, 0)\|_{H \times L^2(r \le \ga_n)}+  \|\vec{\hat{g}}_n - ( \pi, 0)\|_{H \times L^2(\ga_n \le r \le 2\ga_n)} \\
 &\quad+  \|(\pi, 0) - (Q, 0)\|_{H \times L^2(r \ge \ga_n)}\\
 & \le C\left( \frac{1}{n} + \de_n + \ga_n^{-1}\right)
 \end{align*}
 Now, given our choice of $\be_n$,  \eqref{ti g-Q at be} follows from  Corollary~\ref{Cote} and the finite speed of propagation. 
 Rescaling  \eqref{ti g-Q at be}  we have 
 \begin{align*}
\|( {g}_n(-\be_n\mu_n/2), \p_t g_n(-\be_n\mu_n/2))- (Q(\cdot/ \mu_n), 0) \|_{ H \times L^2(r \le \be_n\mu_n)} \to 0.
\end{align*}
This proves \eqref{g-Q}. Also note that by monotonicity of the energy on interior cones and the comparability of the energy and the $H \times L^2$  in $\HH_0$ for small energies, we see that \eqref{b in AB} implies that 
\begin{align}\label{ti b small before be}
\|(\ti b_{n}(-\be_n \mu_n/2),  \p_t\ti b_{n}(-\be_n\mu_n/2))\|_{ H \times L^2(r \le \be_n \mu_n)} \to 0
 \end{align}

Next we prove \eqref{ghbt}. First we define
\begin{align*}
&\hat g_{n, 0}(r)=  \begin{cases} \pi - \frac{\pi- g_n(0, \mu_n \be_n/2)}{\frac{1}{2}\mu_n \be_n} r\quad \textrm{if} \quad r \le  \be_n\mu_n/2\\ g_n(0, r) \quad \textrm{if} \quad  r \ge  \be_n\mu_n/2  \end{cases}\\
& \hat g_{n, 1}(r)= \dot{g}_n(0, r)
\end{align*}
Then, let $\chi \in C^{\infty}([0, \infty))$ be defined so that $\chi(r) \equiv 1$ on the interval $[2, \infty)$ and $\textrm{supp} \chi \subset [1, \infty)$. Define 
\begin{align*}
&\vec{\breve{g}}_n(r) := \chi(4 r/ \be_n \mu_n)( \vec{\hat{g}}_n(r)- (\pi, 0))\\
& \vec{\breve{b}}_n(r):= \chi(4 r/ \be_n \mu_n) \vec{\ti{b}}_n(r)
\end{align*}
 and observe that we have the following decomposition 
 \begin{align*}
 \vec{\breve{g}}_n(r)  = \vec h_n(0, r) +  \vec{\breve{b}}_n(r) + o_n(1).
 \end{align*}
 where the $o_n(1)$ is in the sense of $ H \times L^2$ -- here we also have used that $\be_n \la_n \to 0$ together with \eqref{h reg}. Moreover,  the right-hand side above, without the $o_n(1)$ term, is a profile decomposition in the sense of Corollary~\ref{bg wm} because  of Proposition~\ref{no prof} and Lemma~\ref{local scat lem}. We can then consider the nonlinear profiles. Note that by construction we have $\vec{\breve{g}}_n \in \HH_0$ and as usual, we can use \eqref{gn bemu} to show that $\E(\vec{\breve{g}}_n) \le C < 2\E(Q)$ for large $n$. The corresponding wave map evolution $\vec{\breve{g}}_n(t) \in \HH_0$ is thus global in time and scatters as $t \to \pm \infty$ by Theorem~\ref{main}. We also need to check that $\E(\vec{\breve{b}}_n) \le C < 2 \E(Q)$. Note that by construction and the definition of $\ti b_n$, we have
 \begin{align*}
\E( \vec{\breve{b}}_n) &\le \E(\vec{\ti{b}}_n) + O\left( \int_{0}^{\infty} \frac{4r^2}{\be_{n, 0}^2 \mu_n^2} (\chi^{\prime})^2( 4r/ \be_n \mu_n) \frac{b_n^2((1-\tau_n)r)}{r} \, dr \right)\\
&\quad  + \int_{\be_n \mu_n/2}^{\be_n \mu_n} \frac{\sin^2( \chi(4r/ \be_n \mu_n) b_{n, 0}((1-\tau_n)r))}{r} \, dr \\
&\le  \E(\vec{\ti{b}}_n) + O\left( \int_{\be_n \la_n/2}^{\be_n \la_n} \frac{b_{n, 0}^2(r)}{r} \, dr \right)\\
&=\E(\vec{\ti{b}}_n) + o_n(1) \le C < 2\E(Q),
\end{align*}
 where the last line follows from Proposition~\ref{a prop} and the definition of $b_{n, 0}$, since $\be_n \ll \al_n$. 
 
 Arguing as in the proof of \eqref{nonlin prof ti psi}, we can use Proposition~\ref{no prof},  Proposition~\ref{nonlin profile} and Lemma~\ref{pert} to obtain
 the following nonlinear profile decomposition  \begin{align*}
& \vec{\breve{g}}_n(t, r)  = \vec h_n(t, r) +  \vec{\breve{b}}_n(t, r) + \vec{\breve{\theta}}_n(t, r)\\
& \|\vec{\breve{\theta}}_n\|_{L^{\infty}_t(H \times L^2)} \to 0
 \end{align*}
 Finally observe that by construction and the finite speed of propagation we have 
 \begin{align*}
& \vec{\breve{g}}_n(t, r)  = \vec g_n(t, r) - \pi\\
 &  \vec{\breve{b}}_n(t, r)  =\vec{\ti{b}}_n (t, r)
 \end{align*}
 for all  $t \in [-\tau_n/ (1-\tau_n), 1)$ and $ r \in [\be_n \mu_n/2 + \abs{t}, \infty)$.  Therefore, in particular we have 
 \begin{align*}
 & \vec{g}_n( - \be_n \mu_n/2, r) -(\pi, 0) = \vec h_n(- \be_n \mu_n/2, r) +  \vec{\ti{b}}_n(- \be_n \mu_n/2, r) + \vec{\breve{\theta}}_n( \be_n \mu_n/2, r)
 \end{align*}
 for all $r \in [\be_n \mu_n, \infty)$ which proves \eqref{ghbt}. 
 
 %------------------------------------------------------------------figure-----------------------------------------------------------------------%

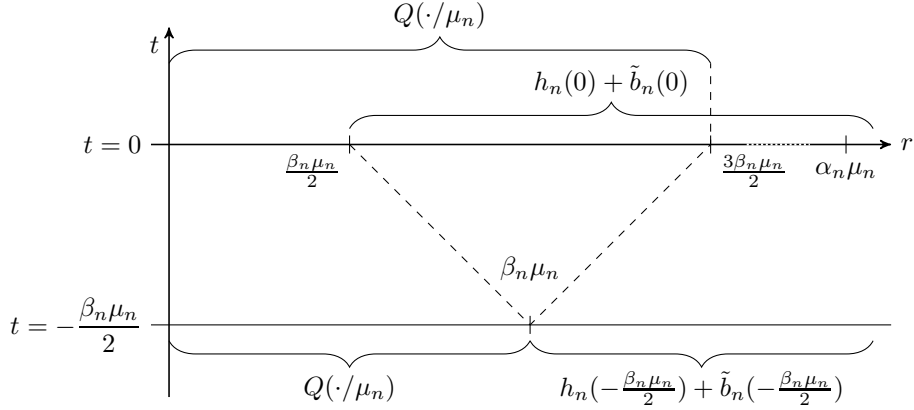
\begin{figure}
\begin{tikzpicture}[
	>=stealth',
	axis/.style={semithick, ->},
	coord/.style={dashed, semithick},
	yscale = 1.2,
	xscale = 1.2]
	\newcommand{\xmin}{0};
	\newcommand{\xmax}{8};
	\newcommand{\ymin}{-2.8};
	\newcommand{\xa}{1}
	\newcommand{\ya}{2};
	\newcommand{\yb}{7.5};
	\newcommand{\xb}{\xa+\yb-\ya}
	\newcommand{\ymax}{0.9};
	\newcommand{\delay}{1.5};
	\newcommand{\fsp}{0.2};
	\draw [axis] (\xmin-\fsp,0) node [left] {$t=0$} -- (\xmax,0) node [right] {$r$};
	\draw [axis] (0,\ymin) -- (0,\ymax+2*\fsp) node [below left] {$t$};
	\draw (-\fsp, -2) node [left] {$\ds t= - \frac{\beta_n \mu_n}{2}$} -- (\xmax,-2);
	\draw (\ya,-0.1) -- (\ya,0.1);
	\draw (\ya,0) node [below left] {$\frac{\beta_n \mu_n}{2}$};
	\draw (3*\ya,-0.1) -- (3*\ya,0.1);
	\draw (3*\ya,0) node [below right] {$\frac{3 \beta_n \mu_n}{2}$};
	\draw [thick, densely dotted, white] (3*\ya+2*\fsp,0) -- (\yb - 2*\fsp,0);
	\draw (\yb,0.1) -- (\yb,-0.1) node [below] {$\alpha_n \mu_n$};
	\draw (2*\ya, -\ya-0.1) -- (2*\ya, -\ya+0.1) node [yshift=18] {$\beta_n \mu_n$};
	\draw [dashed] (\ya,0) -- (2*\ya, -\ya) -- (3*\ya,0) --(3*\ya,\ymax);
	\draw [decorate,decoration={brace,amplitude=10}] (0, \ymax) -- (3*\ya,\ymax)  node [midway,yshift=18] {$Q(\cdot/\mu_n)$}; 	
	\draw [decorate,decoration={brace, amplitude=10, raise=6}] (\ya,0) -- (\xmax-\fsp,0)  node [midway,yshift=24] {$ h_n(0) + \ti{b}_n(0)$};
	\draw [decorate,decoration={brace,mirror,amplitude=10, raise=6}] (0, -\ya) -- (2*\ya,-\ya )  node [midway,yshift=-24] {$Q(\cdot/\mu_n)$}; 
	\draw [decorate,decoration={brace,mirror,amplitude=10, raise=6}] (2*\ya,-\ya ) -- ( (\xmax-\fsp, -\ya)  node [midway,yshift=-24] {$ h_n(- \frac{\beta_n \mu_n}{2}) + \ti{b}_n( - \frac{\beta_n \mu_n}{2})$}; 	
\end{tikzpicture}
\caption{\label{fig:3} A schematic depiction of the evolution of the decomposition \eqref{ghqb} from time $t=0$ up to $t = - \frac{\beta_n \mu_n}{2}$. At time $t = - \frac{\beta_n \mu_n}{2}$ the decomposition \eqref{g at be n mu n} holds.}
\end{figure}

%------------------------------------------------------------------figure-----------------------------------------------------------------------%

We can combine \eqref{g-Q},  \eqref{ghbt}, \eqref{ti b small before be},  and \eqref{h reg} together with the monotonicity of the energy on interior cones to obtain the decomposition 
\begin{align} \label{g at be n mu n}
\vec{g}_n( - \be_n \mu_n/2, r) &=  (Q(r/ \mu_n), 0)+ \vec h_n(- \be_n \mu_n/2, r) \\ \notag
& \quad +\vec{\ti{b}}_n(- \be_n \mu_n/2, r) + \vec{\ti{\theta}}_n( r)\\
&\| \vec{\ti{\theta}}_n\|_{H \times L^2} \to 0 
 \end{align}

Now define 
\begin{align*}
s_n:=- \frac{r_0}{1-\tau_n}.
\end{align*}
The next step is to prove the following decomposition at time $s_n$: 
\begin{align}\label{ghbz}
&\vec g(s_n, r)-( \pi, 0) = \vec h_n(s_n, r) + \vec{\ti{b}}_n(s_n, r) + \vec \zeta_n( r) \quad  \forall r \in [\abs{s_n}, \infty)\\
&  \|\vec \zeta_n\|_{H \times L^2}  \to 0\label{z to 0}
\end{align}

We proceed as in the proof of \eqref{ghbt}. By \eqref{g-Q} we can argue as in the proof of Lemma~\ref{tech lem} and find $\rho_n \to \infty$ with $\rho_n\ll \be_n$ so that 
\begin{align}\label{gn rho}
 g_n( - \be_n \mu_n/2,   \, \rho_n\mu_n ) \to \pi \quad \textrm{as} \quad n  \to \infty
 \end{align}
Define 
\begin{align*}
&\hat f_{n, 0}(r)=  \begin{cases} \pi - \frac{\pi- g_n(- \be_n \mu_n/2,  \,  \rho_n \mu_n)}{ \rho_n \mu_n} r\quad \textrm{if} \quad r \le   \rho_n\mu_n\\ g_n(- \be_n \mu_n/2, r) \quad \textrm{if} \quad  r \ge  \rho_n \mu_n  \end{cases}\\
& \hat f_{n, 1}(r)= \dot{g}_n(- \be_n\mu_n/2, r)
\end{align*}
Let  $\chi \in C^{\infty}$ be as above and set 
\begin{align*}
&\vec{f}_n(r) := \chi(2r/ \rho_n \mu_n)( \vec{\hat{f}}_n(r)- (\pi, 0))\\
& \vec{\hat{b}}_n(r):= \chi(2 r/ \rho_n \mu_n) \vec{\ti{b}}_n(-\be_n \mu_n/2, r)
\end{align*}
 Observe that we have the following decomposition: 
  \begin{align*}
 \vec{f}_n(r)  = \vec h_n(- \be_n \mu_n/2, r) +  \vec{\hat{b}}_n(r) + o_n(1).
 \end{align*}
 where the $o_n(1)$ above is in the sense of $H \times L^2$. Moreover, the  right-hand side above, without the $o_n(1)$ term, is a profile decomposition in the sense of Corollary~\ref{bg wm} because  of Proposition~\ref{no prof} and Lemma~\ref{local scat lem}. We can then consider the nonlinear profiles. Note that by construction we have $\vec{f}_n \in \HH_0$ and, as usual, we can use  \eqref{gn rho} to  show that $\E(\vec{f}_n) \le C < 2\E(Q)$ for large $n$. The corresponding wave map evolution $\vec{f}_n(t) \in \HH_0$ is thus global in time and scatters as $t \to \pm \infty$ by Theorem~\ref{main}. 
 
  As in the proof of  \eqref{ghbt} it is also easy to show that $\E(\vec{\hat{b}}_n) \le C<2\E(Q)$ where here we use \eqref{ti b small before be} instead of Proposition~\ref{a prop}. 
 
 %------------------------------------------------------------------figure-----------------------------------------------------------------------%
\begin{figure}
\begin{tikzpicture}[
	>=stealth',
	axis/.style={semithick, ->},
	coord/.style={dashed, semithick},
	yscale = 1.6,
	xscale = 1.6]
	\newcommand{\xmin}{-0};
	\newcommand{\xmax}{6};
	\newcommand{\ymin}{-3.8};
	\newcommand{\xa}{0.5}
	\newcommand{\ya}{1.5};
	\newcommand{\yb}{3};
	\newcommand{\xb}{\xa+\yb-\ya}
	\newcommand{\ymax}{0.2};
	\newcommand{\delay}{1.5};
	\newcommand{\fsp}{0.2};
	\draw [axis] (\xmin-\fsp,0) node [left] {$t=0$} -- (\xmax,0) node [right] {$r$};
	\draw [axis] (0,\ymin) -- (0,\ymax+\fsp) node [below right] {$t$};
	\draw [->] (0,0) -- (3.8,-3.8); 
	\draw (-\fsp, -\ya) node [left] {$\ds t=  - \frac{\beta_n \mu_n}{2}$} -- (\xmax,-\ya);
	\draw (\xa,-\ya+0.1) node [above] {$\rho_n \mu_n$} -- (\xa,-\ya-0.1);
	\draw (\ya,-\ya+0.1) -- (\ya,-\ya-0.1);
	\draw (\ya,-\ya)  node [above right] {$\frac{\beta_n \mu_n}{2}$};
	\draw [dashed] (\xa,-\ya) -- (\xb,-\yb);
	\draw (-\fsp, -\yb) node [left] {$\ds t= s_n$} -- (\xmax,-\yb);
	\draw (\xb,-\yb+0.1) -- (\xb,-\yb-0.1) node [below] {$\rho_n \mu_n + |\nu_n|$};
	\draw (\yb,-\yb+0.1) -- (\yb,-\yb-0.1);
	\draw (\yb,- \yb) node [above right] {$|s_n|$};
	\draw [decorate,decoration={brace,amplitude=10, raise=18}] (\xa,-\ya) -- (\xmax,-\ya) node [midway,yshift=36] {$\vec h_n( - \frac{\beta_n \mu_n}{2}) + \vec{\ti{b}}_n(-\frac{\beta_n \mu_n}{2})$}; 
	\draw  [decorate,decoration={brace, amplitude=10, raise=16}] (\xb,-\yb) -- (\xmax,-\yb) node [midway,yshift=36] {$\vec h_n( s_n) + \vec{\ti{b}}_n(s_n)$};
	\draw  [<->, xshift=-3] (0,-\ya) -- (0,-\yb) node [midway,left] {$\nu_n$};
%\draw [thick] plot[domain={\xmin+\delay}:{\xmax-\delay}] (\x,{(abs(\x)+\delay});	
\end{tikzpicture}
\caption{\label{fig:2} A schematic depiction of the evolution of the decomposition \eqref{g at be n mu n} up to time $s_n$. On the interval $[|s_n|, +\infty)$, the decomposition \eqref{ghbz} holds.}
\end{figure}
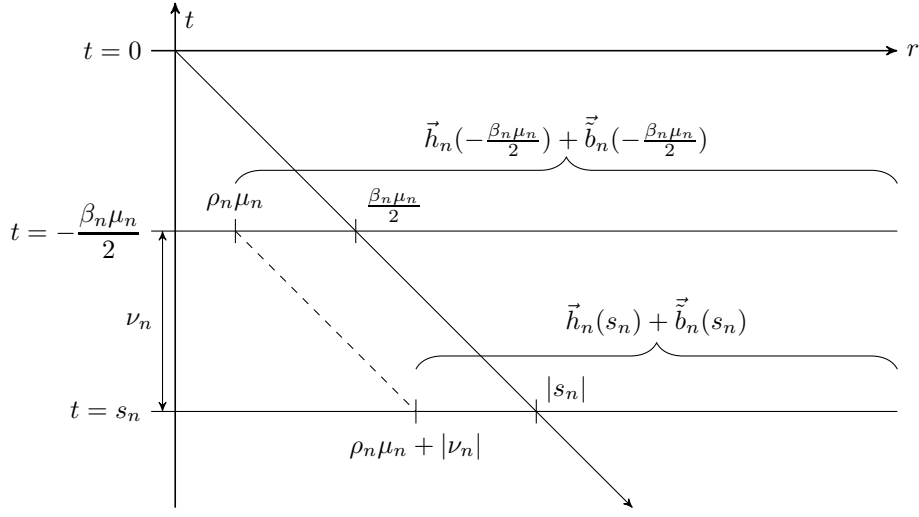
%------------------------------------------------------------------figure-----------------------------------------------------------------------%
 
 %-----------------------------------------------------Figure------------------------------------------------------------------% 
  %\begin{figure}
 %\labellist
 %\pinlabel $\rho_n\mu_n$ [t] at 150 933
 %\pinlabel $\be_n\mu_n/2$ [t] at 300 945
 %\pinlabel $\abs{s_n}$ [t] at 614 945
 %\pinlabel $-\be_n\mu_n/2$ [t] at -75 640
 %\pinlabel $s_n$ [t] at -25 310
 %\pinlabel $t$ [t] at -10 50
 %\pinlabel $r$ [t] at 900 933
 %\endlabellist
 %\includegraphics[scale=.3]{backcone} \caption{\label{fig}  }
 %\end{figure}

 %-----------------------------------------------------Figure------------------------------------------------------------------%

 Arguing as in the proof of \eqref{nonlin prof ti psi} we can use Proposition~\ref{nonlin profile} and Lemma~\ref{pert} to obtain
 the following nonlinear profile decomposition 
 \begin{align*}
& \vec{f}_n(t, r)  = \vec h_n( - \be_n \mu_n/2 +t, r) +  \vec{\hat{b}}_n(t, r) + \vec{\ti{\zeta}}_n(t, r)\\
& \|\vec{\ti{\zeta}}_n\|_{L^{\infty}_t(H \times L^2)} \to 0
 \end{align*}
In particular, for $$\nu_n:= s_n +\be_n \mu_n/2$$ we have 
 \begin{align*}
& \vec{f}_n(\nu_n, r)  = \vec h_n( s_n, r) +  \vec{\hat{b}}_n(\nu_n, r) + \vec{\ti{\zeta}}_n(\nu_n, r).
 \end{align*}
By the finite speed of propagation we have that
\begin{align*}
&\vec{f}_n(\nu_n, r) = \vec g_n(s_n, r) \\
& \vec{\hat{b}}_n(\nu_n, r) = \vec{\ti{b}}_n(s_n, r)
\end{align*}
as long as $r \ge \rho_n \mu_n + \abs{\nu_n}$. Using the fact that $\rho_n \ll \be_n$ we have that $\abs{s_n} \ge \rho_n \mu_n + \abs{\nu_n}$ and hence, 
 \begin{align*}
& \vec{g}_n(s_n, r)-(\pi, 0)  = \vec h_n( s_n, r) +  \vec{\ti{b}}_n(s_n, r) + \vec{\ti{\zeta}}_n(\nu_n, r) \quad \forall r \in [\abs{s_n}, \infty).
 \end{align*}
 Setting $\vec{\zeta}_n:= \vec{\ti{\zeta}}_n(\nu_n)$ we obtain \eqref{ghbz} and \eqref{z to 0}. Now, combine \eqref{z to 0},  \eqref{h small}, and the monotonicity of the energy on light cones for the evolution of $\vec h_n$, we obtain:
 \begin{align}\label{at sn}
 \|\vec g_n(s_n)-(\pi, 0)- \vec{\ti{b}}_n(s_n) \|_{ H \times L^2(\abs{s_n} \le r \le  2\abs{s_n})} \le  \frac{C\de_0}{K}
 \end{align}
 for $n$ large enough.
 By Corollary~\ref{ext en est} and \eqref{contradict}, there exists $\be_0 >0$ so that for all $t \in \R$ we have 
 \begin{align*}
  \|\vec{\ti{b}}_n(t)\|_{H \times L^2( r \ge \abs{t})} \ge  \be_0 \de_0
  \end{align*}
  By \eqref{ti b ext} and the monotonicity of the energy on cones we have $$\|\vec{\ti{b}}_n(t)\|_{H \times L^2(r \ge\abs{ t}+1)} \to 0$$ as $n \to \infty$. Therefore we have 
 \begin{align*}
  \|\vec{\ti{b}}_n(t)\|_{H \times L^2( \abs{t} \le r \le 1+\abs{t})} \ge  \frac{\be_0 \de_0}{2}
  \end{align*}
  for $n$ large enough and for all $t \in \R$. Hence setting $t= s_n$ we see that the above and \eqref{at sn} imply in particular that
  \begin{align*}
  \|\vec g_n(s_n) - (\pi, 0)\|_{H \times L^2( \abs{s_n} \le r \le 1+\abs{s_n})} \ge  \frac{\be_0 \de_0}{4} >0
  \end{align*}
  for $n, K$ large enough. Un-scaling this we obtain 
   \begin{align*}
  \|\vec \psi(\tau_n- r_0) - (\pi, 0)\|_{H \times L^2( r_0 \le r \le r_0+(1-\tau_n))} \ge  \frac{\be_0 \de_0}{4} >0.
  \end{align*}  
  However this contradicts the fact the $\psi(t, r)$ cannot concentrate any energy at the point $(1-r_0, r_0) \in [0, 1) \times [0, \infty)$ with $r_0>0$. This concludes the proof  of Propostion~\ref{b comp} and hence of Proposition~\ref{bu} as well. 
\end{proof}

 We can now finish the proof of Theorem~\ref{bu cont}. 
 \begin{proof}[Proof of Theorem \ref{bu cont}]
 Let $\vec a(t)$ be defined as in \eqref{a def}.  Recall that by Lemma~\ref{a lem} we have 
 \begin{align}\label{Epsi-Efy}
 \lim_{t \to 1} \E (\vec a(t)) = \E (\vec \psi)- \E(\vec \fy)
 \end{align}
 Over the course of the proof of Proposition~\ref{bu} we have found a sequence of times $\tau_n \to 1$ so that 
 \begin{align*}
 \E(\vec a(\tau_n)) \to \E(Q)
 \end{align*}
 as $n \to \infty$. Since $\E(\vec \psi) = \E(Q) + \eta$ this implies that $\E(\vec \fy) = \eta$ since the right hand side of \eqref{Epsi-Efy} is independent of $t$. This then implies that 
 \begin{align*}
 \lim_{t \to 1} \E (\vec a(t)) =\E(Q)
 \end{align*}
We now use the variational characterization of $Q$ to show that in fact $\|\dot a(t)\|_{L^2} \to 0$ as $t \to 1$. To see this observe that since $a(t) \in \HH_1$ we can deduce by \eqref{var char} that 
 \begin{align*}
  \E(Q) \leftarrow \E(a(t), \dot{a}(t)) \ge  \int_0^{\infty} \dot{a}^2(t, r) \, r\, dr + \E(Q)
 \end{align*}
 Next observe that the decomposition in Lemma~\ref{decomposition} provides us with a function $\la: (0, \infty) \to (0, \infty)$ such that 
 \begin{align*}
 \|a(t,  \cdot)- Q(\cdot/ \la(t))\|_H \le \de(\E( a(t), 0) - \E(Q)) \to 0
 \end{align*}
 This also implies that 
 \begin{align}\label{a-Q en to 0}
 \E(\vec a(t) - (Q(\cdot/ \la(t)), 0)) \to 0
 \end{align}
 as  $t \to 1$. Since $t \mapsto a(t)$ is continuous in $H$ for $t \in[0, 1)$ it follows from Lemma~\ref{decomposition} that $\la(t)$ is continuous on $[0, 1)$.  Therefore we have established that 
 \begin{align*}
 \vec \psi(t) - \vec \fy(t) - (Q(\cdot/ \la(t)), 0) \to 0 \quad \textrm{in} \quad H \times L^2 \quad \textrm{as} \quad t \to 1
 \end{align*}
 It remains to show that $\la(t) = o(1-t)$. This follows immediately from the support properties of $\nabla_{t, r} a$  and from \eqref{a-Q en to 0}. To see this observe that $a(t, r)-Q(r/ \la(t)) = \pi-Q(r/ \la(t))$ on $[1-t, \infty)$. Thus, 
 \begin{align*}
\E_{\frac{1-t}{\la(t)}}^{\infty}(Q) = \E_{1-t}^{\infty}( \pi- Q(\cdot/ \la(t))) \le  \E(\vec a(t) - (Q(\cdot/ \la(t)), 0)) \to 0.
\end{align*}
But this then implies that $\frac{1-t}{\la(t)} \to \infty$ as $t \to 1$. This completes the proof. 
  \end{proof}

\appendix

\section{Higher Equivariance classes and more general targets}
%\begin{rem} 

\subsection{$1$-equivariant wave maps to more general targets}
Theorem~\ref{main}, Theorem~\ref{rigidity}, and Theorem~\ref{bu cont}  can be extended to a larger class of equations, namely equivariant wave maps to general, rotationally symmetric compact targets. To be specific, each of these theorems holds in the case that the target manifold $M$ is a surface of revolution with the metric given in polar coordinates, $(\rho, \om) \in [0, \infty) \times S^1$, by $ds^2 = d\rho^2 + g^2(\rho)d\om^2$ where $g: \R \to \R$ is a smooth, odd, function with $g(0)=0$, $ g'(0)= 1 $. In addition, in order to ensure the existence of stationary solutions to the corresponding equivariant wave map equation we need to require that there exists $C>0$ such such that $g(C)=0$ and we let $C^*$ be minimal with this property. We also assume that $g'(C^*)=-1$ and that $g$ is periodic with period $2C^*$.  In this case, the nonlinear wave equation of interest is given by 
\begin{align} \label{cp gen1} 
&\psi_{tt} - \psi_{rr} - \frac{1}{r} \psi_r + \frac{f(\psi)}{r^2} = 0\\
&(\psi, \psi_t)\vert_{t=0} = (\psi_0, \psi_1) \notag
\end{align} 
where $f(\psi) = g(\psi)g'(\psi)$. The conserved energy for this problem is given by 
\begin{align*}
\E(\vec \psi(t)) = \int_0^{\infty} \left( \psi_t^2 + \psi_r^2 + \frac{g^2(\psi)}{r^2} \right)\, r \, dr = \textrm{const}. 
\end{align*}

To see how this extension works, we note that the small data well-posedness theory for  \eqref{cp gen1} is given in \cite[Theorem $2$]{CKM}. One then needs replacements for the estimates involving the $\sin$ function in the proof of the orthogonality of the nonlinear
energy, the proof of the nonlinear perturbation theory, and later in estimates involving the energy of $\vec a(t)$, namely \eqref{sin},  \eqref{sin 2}, and \eqref{sin 3}. But, the same type of estimates for $g$ are easily established using the assumptions we have made on $g$ and its derivatives and simple calculus.

For more details regarding more general metrics we refer the reader to \cite{CKM}. Note that since we do not rely on \cite[Lemma $7$]{CKM} we are able to eliminate their condition~ \cite[(A$3)$]{CKM}. 
%\end{rem}

\subsection{Higher equivariance classes and the $4d$-equivariant Yang-Mills system} 
%\begin{rem} 
We can also consider higher equivariance classes,  $\ell >1$. Restricting our attention again to the case $g(\rho)= \sin( \rho)$, the Cauchy problem for $\ell$ equivariant wave maps reduces to 
\begin{align} \label{cp ell} 
&\psi_{tt} - \psi_{rr} - \frac{1}{r} \psi_r  + \ell^2 \frac{\sin(2\psi)}{2r^2} = 0\\
&(\psi, \psi_t)\vert_{t=0} = (\psi_0, \psi_1) \notag
\end{align} 
For $\ell$-equivariant wave maps of topological degree zero we can, as in the $1$-equivariant case, consider the reduction $\psi =: r^{\ell} u$ and we obtain the following Cauchy problem for $u$: 
\begin{align}\label{semi lin ell}
u_{tt} - u_{rr} - \frac{2 \ell + 1}{r} u_r = u^{1+ 2/ \ell} Z(r^{\ell}u)
\end{align}
with $$\ds{Z(\rho):=\frac{\ell^2}{2} \frac{\sin(2 \rho)- 2\rho}{\rho^{1 + 2/ \ell}}}$$ a bounded function. In \cite[Theorem $2$]{CKM} a suitable  local well-posedness/small data theory for such a nonlinearity is addressed when $\ell =2$ and thus Theorem~\ref{main} follows from the same arguments in this paper. For $\ell>2$, one would need to develop a suitable well-posedness theory for \eqref{semi lin ell}. This presents some difficulties due the fractional power, $1+2/ \ell$, in the nonlinearity. 

One can also consider the $4d$ equivariant Yang-Mills system: 
\begin{align*}
&F_{\al \be} = \p_{\al} A_{\be} - \p_{\be} A_{\al} +[A_{\al}, A_{\be}]\\
&\p_{\be} F^{\al \be} + [A_{\be}, F^{\al \be}]= 0, \quad \al, \be = 0, \dots, 3
\end{align*}
for the connection form $A_{\al}$ and the curvature $F_{\al \be}$. After, making the equivariant ansatz: 
\begin{align*}
A_{\al}^{ij}= (\de^{i}_{\al}x^j -\de^j_{\al} x^i)\frac{1-\psi(t, r)}{r^2}
\end{align*} 
one obtains the following equation for $\psi$: 
\begin{align*}
\psi_{tt} - \psi_{rr}- \frac{1}{r} \psi_r - \frac{ 2\psi(1-\psi^2)}{r^2} = 0
\end{align*}
which can be written in the form 
\begin{align} \label{cp ym} 
&\psi_{tt} - \psi_{rr} - \frac{1}{r} \psi_r  + \ell^2 \frac{f(\psi)}{r^2} = 0\\
&(\psi, \psi_t)\vert_{t=0} = (\psi_0, \psi_1) \notag
\end{align} 
for $f(\rho)= g(\rho) g^{\prime}(\rho)$ and $g(\rho) = 1/2(1- \rho^2)$ and $\ell =2$. This equation is of the same form as \eqref{cp ell} with $\ell =2$ and a more general metric $g$. The local well-posedness/small data scattering theory for \eqref{cp ym} is addressed in \cite[Theorem $2$]{CKM}. The proof and conclusions of Theorem $1.1$ thus hold for solutions of this equation with suitable modifications as in the case of $1$-equivariant wave maps to more general targets addressed above. 

As we mentioned in the introduction, modulo a suitable local well-posedness/small data theory,  one should be able to apply our methods to prove the analog of Theorem~\ref{bu cont} for the odd higher equivariance classes, $\ell= 3, 5, 7, \dots, $. The reason is  that if $\ell$ is odd, the linearized version of equation \eqref{cp ell} is a $2 \ell +2$ dimensional free radial wave equation with $2 \ell +2 = 0 \mod 4$ for $\ell$ odd, and in these dimensions Proposition~\ref{lin ext estimate} holds, see \cite[Corollary~$5$]{CKS}. 

However, as demonstrated in \cite{CKS}, Proposition~\ref{lin ext estimate} {\em fails} for $\ell= 2, 4, 6, \dots$, since  $2 \ell+2= 2 \mod 4$ for $\ell$ even. Therefore it is impossible to prove Corollary~\ref{ext en est} in these cases and our contradiction argument for the compactness of the error term $\vec b_n$ does not go through. So our method is not suited to prove the complete conclusions of Theorem~\ref{bu cont} for either the even equivariance classes or the $4d$ Yang-Mills system, which  corresponds roughly to the case $\ell=2$. However, the rest of the argument preceding the proof of Proposition~\ref{bu} should go through and in particular one should be able to deduce Proposition~\ref{no prof}. This would allow one to conclude that the error terms $\vec b_n$ contain  no profiles  and converge to zero in a Strichartz norm adapted to the nonlinearity in \eqref{cp ell}. This is a slightly weaker result than showing that the $\vec b_n$'s vanish in the energy space, but on its own, it is already  quite strong.

\section{Shatah and Tahvildar-Zadeh's theorem in $\HH$:  \\ an appendix by Jacek Jendrej}  \label{a:STZ}

%\section{Shatah and Tahvildar-Zadeh's theorem in the energy space: an appendix by Jacek Jendrej}  \label{a:STZ} 

In this section we prove that a self-similar blow-up rate is impossible for wave maps $\vec \psi(t) \in \HH$, i.e., we prove Lemma~\ref{ext en decay} for maps in the energy class. This section was not included in the published version of this paper. This gap was pointed out to the authors by Jacek Jendrej, who also suggested the following proof. We also note that a direct analog of the results proved in this section were established by Jia and the second author in~\cite[Lemma 2.1]{JK-AJM} via an alternative argument. 

We restate the goals of this section. 

\begin{prop} \label{p:stzH} 
Let~$\vec \psi(t) \in \HH$ be a solution to~\eqref{cp} on the time interval $[0, T_+( \vec \psi))$ with $T_+ < \infty$. Let $\la \in (0, 1)$. Then, 
\EQ{
 \lim_{t \to T_+} \int_{ \la( T_+ - t)}^{T_+ - t} \Big(\psi_t^2(t) + \psi_r^2(t) + \frac{\sin^2 \psi(t)}{r^2}  \Big) \, r \, \ud r = 0 
}
\end{prop} 

We will also prove the following corollary. 
\begin{cor} \label{c:stzdt} 
Let~$\vec \psi(t) \in \HH$ be a solution to~\eqref{cp} on the time interval $[0, T_+( \vec \psi))$ with $T_+< \infty$. Then 
\EQ{
\lim_{t \to T_+} \frac{1}{T_+-t}\int_t^{T_+} \int_0^{T_+-t} \psi_t^2(t', r) \, r \, \ud r \, \ud t' = 0 
}
\end{cor} 

%A naive approach would be to combine the proof will proceed via an approximation argument. However, the argument from~\cite{STZ1} cannot simply be combined with an approximation 

The proof below will use an approximation argument. However, it is not enough to simply combine the scheme from~\cite[Proof of Lemma 2.2]{STZ1} with an approximating sequence $\vec \psi_n(t)$ of smooth wave maps,  since several of the crucial estimates from~\cite{STZ1} would not be uniform in $n$. We thus give a different argument below. We note that we do make use of some key ingredients from the argument in~\cite{STZ1},   such as~\eqref{eq:pApB}.

Let $\vec \psi_n(0)$ be a sequence of smooth functions  in $\HH$ with 
%We will prove these statements using  elements of the basic scheme from~\cite{STZ1} together with an approximation argument. Let $\vec \psi_n(0)$ be a sequence of smooth functions  in $\HH$ with 
\EQ{
 \vec \psi_n(0) \to \vec \psi(0)
}
strongly in $\HH$. Let $\vec \psi_n(t)$ denote the unique solution to~\eqref{cp} with data $\vec \psi_n(0)$. Note that 
\EQ{
\liminf_{n \to \infty} T_+( \vec \psi_n) \ge T_+( \vec \psi)
}
and thus (passing to a subsequence) we may assume the above holds element-wise and below we will write $T_+ = T_+( \vec \psi)$. By the local Cauchy theory for~\eqref{cp} (i.e., continuous dependence) we see that the proof of Proposition~\ref{p:stzH} reduces to proving the following lemma. 
\begin{lem}  \label{l:undec}
Fix $\la >0$ and let $\vec\psi_n$ be as above. Then for any $\epsilon >0$ there exists a time $\tau = \tau(\epsilon)< T_+( \vec \psi)$ such that for all $t \in [\tau, T_+)$ we have 
\EQ{
\liminf_{n \to \infty}  \int_{\la ( T_+ - t)}^{T_+ - t}  \Big( (\p_t \psi_n(t))^2 + (\p_r \psi_n(t))^2 + \frac{\sin^2 \psi_n(t)}{r^2}  \Big) \, r \, \ud r  \le \epsilon
}
where we emphasize that $\tau$ is independent of $n$. 
\end{lem}

To prove Lemma~\ref{l:undec} we introduce some notation from~\cite{STZ}. Denote 
\EQ{
e_n(t, r) &:= \frac{1}{2} \Big(\p_t (\psi_n(t))^2 + (\p_r \psi_n(t))^2 + \frac{\sin^2 \psi_n(t)}{r^2}\Big) \\
m_n(t, r) &:= \p_t \psi_n(t) \p_r\psi_n \\ 
\A_n^2(t, r) &:= r (e_n(t, r) + m_n(t, r)) , \quad \B_n^2(t, r) := r ( e_n(t, r) - m_n(t, r)) 
} 
We also denote by $e, m, \A^2, \B^2$ the corresponding quantities for $\vec \psi(t)$. 

We begin with the following key lemma. 
\begin{lem}[Uniformity of the flux decay]\label{l:unflux} For any $\epsilon>0$ there exists $\tau = \tau(\epsilon) < T_+$ such that if $ \tau \le t_1 \le t_2 < T_+$, then, 
\EQ{
\limsup_{n \to \infty} \int_{t_1}^{t_2} \B_n^2(t, T_+ - t) \, \ud t \le \epsilon. 
}
\end{lem} 
\begin{proof}[Proof of Lemma~\ref{l:unflux}]
From the energy-flux identity we know that for each $n$, and $t_1 \le  t_2 < T_+$, 
\EQ{
\int_0^{T_+- t_1} e_n(t_1, r) \, r \, \ud r \ge  \int_{0}^{T_+- t_2} e_n(t_2, r) \,r \, \ud r
}
Taking the limit as $n \to \infty$ we obtain, 
\EQ{
\int_0^{T_+- t_1} e(t_1, r) \, r \, \ud r \ge  \int_{0}^{T_+- t_2} e(t_2, r) \,r \, \ud r
}
In other words, the function, 
\EQ{
t \mapsto \int_{0}^{T_+ - t} e(t, r) \, r \, \ud r
}
is decreasing, and thus it has a limit, 
\EQ{
\E_0 :=  \lim_{t \to T_+} \int_0^{T_+ - t} e(t, r) \, r \, \ud r
}
Now, fix $\epsilon >0$ and let $\tau = \tau(\epsilon)$ be such that 
\EQ{
 \int_0^{T_+ - \tau(\epsilon)} e( \tau(\epsilon), r) \,r  \, \ud r  \le \E_0 + \frac{1}{3}\eps
}
Next, consider any $\tau(\eps) \le t_1 \le t_2 < T_+$.  For large enough $n$ we can invoke the local Cauchy theory to deduce that 
\EQ{
\int_0^{T_+ - t_1} e_n(t_1, r)  \, r \, \ud r \le \E_0 + \frac{2}{3} \eps \\
\int_0^{T_+ - t_2} e_n(t_2, r)  \, r \, \ud r \ge \E_0 - \frac{1}{3} \eps 
}
By the energy-flux identity and the above we conclude that 
\EQ{
\int_{t_1}^{t_2} \B_n^2(t, T_+ - t) \, \ud t  = \int_0^{T_+ - t_1} e_n(t_1, r)  \, r \, \ud r - \int_0^{T_+ - t_2} e_n(t_2, r)  \, r \, \ud r   \le \eps 
}
as desired. 
\end{proof} 

Before stating the next lemma, we first introduce null coordinates, 
\EQ{
 \eta := (T_+ - t )- r , \quad \xi := (T_+ - t) + r
}
A computation in~\cite[p. 954]{STZ1} shows that 
\EQ{ \label{eq:pApB}
 \abs{ \p_\xi \A_n} \lesssim \frac{1}{r} \B_n \mand  \abs{ \p_\eta \B_n } \lesssim \frac{1}{r} \A_n
}
where the implicit constants above are universal (in particular independent of $n$). 

Define 
\EQ{
 \ti \la := \frac{1 - \la}{1+ \la} < 1
} 

\begin{lem}\label{l:hsmall} 
For any $\eps>0$ there exists $\xi_1>0$ such that if $\xi  \le \xi_1$ then, 
\EQ{
\limsup_{n \to \infty} \int_0^{\ti \la \xi} \A_n^2( \eta', \xi) \, \ud \eta' \le \eps. 
}
\end{lem}

Next, let $0 < \xi_0 \ll 1$ to be determined below and consider $\xi$ with $ 0 < \xi  \le  \xi_0 \ll 1$.  For all $\xi'$ with $ \xi  \le  \xi' \le \xi_0$ we define
\EQ{
h_n( \xi', \xi) := \int_0^{ \ti \la \xi} \A_n^2( \eta', \xi') \, \ud \eta' 
}
With this notation, Lemma~\ref{l:hsmall} is the statement that $\limsup_{n \to \infty} h_n( \xi, \xi) \to 0$ as $\xi \to 0$. 
The first  ingredient in the proof of Lemma~\ref{l:hsmall} is the following claim.  
\begin{claim} \label{c:hnxi0} 
For any fixed $0<\xi_0 \ll 1$, we have 
\EQ{ \label{hnxi0} 
\limsup_{n \to \infty} h_n( \xi_0, \xi)   \to 0 \mas  \xi \to 0
}
\end{claim} 
\begin{proof}[Proof of Claim~\ref{hnxi0}]
The proof is similar in spirit to the proof of Lemma~\ref{l:unflux}. Let $\xi_0$ be fixed and consider any $0< \xi < \xi_0$.  Consider in $(1+1)$-dimensions, the triangle with vertices $A = (0, \xi_0)$, $B_\xi= (\ti \la \xi, \xi_0)$, and  $C_\xi= (0,  \xi_0 + \ti \la \xi)$.  Note that while boundary of this triangle does not correspond to a truncated light cone in the original $(t, x)$ variables, it does correspond to a surface made up of constant $t$-slices and null hypersurfaces. 
Indeed, for each $\xi < \xi_0$ the line segments connecting $B_\xi$ to $C_\xi$ correspond to constant $t$-slices in the original coordinates. The line segment joining $A$ to $C_\xi$ is a piece of the backwards light cone emanating from $(T_+, 0)$, i.e, it is a constant-$\eta= 0$ slice. And the line segment joining $A$ to $B_\xi$ is also manifestly null, i.e, it is a constant-$\xi = \xi_0$ slice. It follows from finite speed of propagation that for any $\xi_1< \xi_2 < \xi_0$ we have 
\EQ{
\int_{[B_{\xi_1}, C_{\xi_1}]} e_n(t, r) \,r \, \ud r \le \int_{[B_{\xi_2}, C_{\xi_2}] }e_n(t, r) \,r \, \ud r 
}
where $[B_\xi, C_{\xi}]$ denotes the line segment connecting $B_\xi$ to $C_\xi$. Passing to the limit as $n \to \infty$ yields, 
\EQ{
\int_{[B_{\xi_1}, C_{\xi_1}]} e(t, r) \,r \, \ud r \le \int_{[B_{\xi_2}, C_{\xi_2}] }e(t, r) \,r \, \ud r 
}
Hence the energy over these slices is decreasing as $\xi \to 0$ and thus there exists a limit, 
\EQ{
\E_{\xi_0} := \lim_{\xi \to 0} \int_{[B_{\xi}, C_{\xi}]} e(t, r) \,r \, \ud r.
}
In fact, $\E_{\xi_0} = 0$ for any $\xi_0>0$ since any positive limit would correspond to a blow-up via concentration of energy before the time $T_+$. Thus for any fixed $\eps>0$ we can find  $\xi_1= \xi_1(\eps)$ small enough so that for any $\xi \le \xi_1$ we have 
\EQ{
 \int_{[B_{\xi}, C_{\xi}]} e(t, r) \,r \, \ud r \le   \frac{1}{2} \eps
}
Thus for large enough $n$ we can ensure that 
\EQ{
 \int_{[B_{\xi}, C_{\xi}]} e_n(t, r) \,r \, \ud r \le  \eps \label{eq:eneps1} 
}
for all $\xi \le \xi_1$. Finally, the positive fluxes along the segments joining $C_\xi$ to $A$ and joining $B_\xi$ to $A$ are given by 
\EQ{
\int_{\xi_0}^{\xi_0 + \ti \la \xi} \B_n^2( 0, \xi') \, \ud \xi', \quad  h_n( \xi_0, \xi) = \int_{0}^{\ti \la \xi} \A^2_n( \eta', \xi_0) \, \ud \eta' 
}
Since the latter flux above is precisely $h_n( \xi_0, \xi)$ we have by~\eqref{eq:eneps1} that 
\EQ{
h_n( \xi_0, \xi) \le \eps
}
for all $n$ large enough, as desired. 
\end{proof} 

Proving Lemma~\ref{l:hsmall} requires showing that $\limsup_n h_n( \xi, \xi)$ is small. So we try to propagate the smallness given by Claim~\ref{c:hnxi0} from $\xi_0$ all the way down to $\xi$ using~\eqref{eq:pApB} and an argument based on Grownwall's inequality. 

\begin{claim} \label{c:hn'} 
For any $\eps>0$ there exists $\xi_0, \xi_1$, with  $0< \xi_1 < \xi_0 \ll 1$ such that if  $\xi  \le \xi_1$, then for all $0 < \xi \le \xi' \le \xi_0$ we have 
\EQ{
h_n( \xi', \xi) \le  \eps + C  \eta \int_{\xi'}^{\xi_0} \frac{h_n( \xi'', \xi)}{( \xi'')^2} \, \ud \xi''
}
for all $n$ large enough. Above $C>0$ is a universal constant (which may depend on $\la$) and from now on, given $\xi$ we set 
\EQ{
\eta:= \ti \la \xi
}
\end{claim} 
\begin{proof}[Proof of Claim~\ref{c:hn'}]
Given $\xi_0$, which will be fixed below, find $\xi_1$ with $\xi_1 < \xi_0$ small enough so that by Claim~\ref{c:hnxi0} we have 
\EQ{
h_n( \xi_0, \xi) \le \eps
}
for all $\xi \le \xi_0$. Now fix any such $\xi \le \xi_1$ and let $\eta := \ti \la \xi$. 
By~\eqref{eq:pApB}, for $\xi \le \xi'  \le \xi_0$ we have 
\EQ{
\abs{ h_n'( \xi', \xi)}  \lesssim  \int_0^\eta  \frac{1}{\xi'}  \A_n( \eta', \xi') \B_n(\eta', \xi')\,  \ud \eta'
}
where $h_n'$ means the derivative in the first slot, i.e., in $\xi'$. We have used above that in the region $r \ge \la (T_+ - t)$ we have $\xi' \lesssim r$, with a constant depending on $\la$. Integrating back to $\xi_0$ yields, 
\EQ{ \label{eq:hnxi'xi0} 
h_n( \xi', \xi) &\le h_n(\xi_0, \xi)  + C \int_{\xi'}^{\xi_0}\int_0^\eta  \frac{1}{\xi''}  \A_n( \eta', \xi'') \B_n(\eta', \xi'') \, \ud \eta' \, \ud \xi'' \\
& = h_n(\xi_0, \xi)  + C \int_0^\eta  \int_{\xi'}^{\xi_0}\frac{1}{\xi''}  \A_n( \eta', \xi'') \B_n(\eta', \xi'')\, \ud \xi''  \, \ud \eta' 
}
We pause for a moment and consider for a given $\xi'$ the quantity 
\EQ{
g_n(\eta', \xi', \xi_0):= \int_{\xi'}^{\xi_0} \B_n^2(\eta', \xi'')\, \ud \xi'' 
}
Again using~\eqref{eq:pApB}, and writing $g_n(\eta') = g_n( \eta', \xi', \xi_0)$ for simplicity, we deduce that 
\EQ{
\abs{ g_n'( \eta')} &\lesssim \int_{\xi'}^{\xi_0} \frac{1}{\xi''} \A_n(\eta', \xi'') \B_n(\eta', \xi'')\, \ud \xi''  \\
& \le \left(\int_{\xi'}^{\xi_0} \Big(\frac{1}{\xi''}\Big)^2 \A_n^2(\eta', \xi'')\, \ud \xi''\right)^{\frac{1}{2}} \sqrt{ g_n(\eta')} 
}
It follows that for $\eta' \le \eta = \ti \la \xi$ we have 
\EQ{
\sqrt{ g_n(\eta')} &\le \sqrt{ g_n(0)} + C \int_0^{\eta'} \left(\int_{\xi'}^{\xi_0} \Big(\frac{1}{\xi''}\Big)^2 \A_n^2(\eta'', \xi'')\, \ud \xi''\right)^{\frac{1}{2}}  \ud \eta'' \\
& \le \sqrt{ g_n(0)} + C \int_0^{\eta} \left(\int_{\xi'}^{\xi_0} \Big(\frac{1}{\xi''}\Big)^2 \A_n^2(\eta'', \xi'')\, \ud \xi''\right)^{\frac{1}{2}}  \ud \eta'' \\
&  \le \sqrt{ g_n(0)} + C \sqrt{\eta} \left( \int_{\xi'}^{\xi_0}\int_0^{\eta} \Big(\frac{1}{\xi''}\Big)^2 \A_n^2(\eta'', \xi'')\, \ud \eta''  \ud \xi''\right)^{\frac{1}{2}}  \\
& \le \sqrt{ g_n(0)} + C \sqrt{\eta} \left( \int_{\xi'}^{\xi_0} \Big(\frac{1}{\xi''}\Big)^2 h_n( \xi'', \xi)  \ud \xi''\right)^{\frac{1}{2}}
}
Now, applying Cauchy-Schwarz to the right-hand-side of~\eqref{eq:hnxi'xi0} and applying the above we obtain, 
\begin{multline} 
h_n( \xi', \xi) \le h_n(\xi_0, \xi)  \\
  + C \int_0^\eta \left( \int_{\xi'}^{\xi_0}   \frac{\A_n^2(\eta', \xi'') }{(\xi'')^2}\ud \xi'' \right)^{\frac{1}{2}} \Bigg(  \sqrt{ g_n(0)} + C \sqrt{\eta} \left( \int_{\xi'}^{\xi_0}  \frac{h_n( \xi'', \xi)}{(\xi'')^2}  \ud \xi''\right)^{\frac{1}{2}} \Bigg) \ud \eta' 
\end{multline} 
which after another application of Cauchy-Schwarz becomes 
\EQ{
&h_n(\xi', \xi) \le h_n( \xi_0, \xi)   \\
&\quad + C\sqrt{\eta}  \left( \int_{\xi'}^{\xi_0} \frac{h_n(\xi'', \xi)}{(\xi'')^2}  \, \ud \xi'' \right)^{\frac{1}{2}} \Bigg(  \sqrt{ g_n(0)} + C \sqrt{\eta} \left( \int_{\xi'}^{\xi_0}  \frac{h_n( \xi'', \xi)}{(\xi'')^2}  \ud \xi''\right)^{\frac{1}{2}} \Bigg)
}
where we note above that $C$ is independent of $n$ and depends only on $\la$. Now, note that $g_n( 0) = g_n(0, \xi', \xi_0)$ is part of the flux controlled by Lemma~\ref{l:unflux} and hence we can choose (and fix) $\xi_0$ small enough so that for all $\xi'>0$ we have 
\EQ{
\limsup_{n \to \infty} g_n(0) \le \frac{1}{2} \eps 
}
Now that we have fixed $\xi_0$ we can now find $\xi_1>0$ small enough so that for any $\xi \le \xi_1$ we have by Claim~\ref{c:hnxi0} 
\EQ{
\limsup_{n \to \infty} h_n(\xi_0, \xi) \le \frac{1}{2} \eps 
}
Inserting these estimates above completes the proof of the claim. 
\end{proof} 

To conclude we show how Lemma~\ref{l:hsmall} follows from Claim~\ref{c:hn'} together with Grownwall's inequality, which we recall below. 
\begin{lem}[Gronwall's inequality] 
Suppose $f, g$ are continuous non-negative functions on the interval $[x_0, x]$ and for all $x' \in [x_0, x]$ we have 
\EQ{
f(x')  \le \eps + \int_{x_0}^{x'} g(x'') f(x'') \, \ud x'' 
}
Then, 
\EQ{
f(x) \le \eps + \eps \int_{x_0}^x g(x') \exp \Big( \int_{x'}^x g( x'') \, \ud x'' \Big) \, \ud x' 
}
\end{lem} 
\begin{proof}[Proof of Lemma~\ref{l:hsmall}]
By Claim~\ref{c:hn'} and Gronwall with $f(\xi') = h_n(\xi', \xi)$ and $g(\xi') = C \eta (\xi')^{-2}$, recalling that $\eta = \ti \la \xi$  we have 
\EQ{
h_n( \xi, \xi) &\le \eps + \eps \int_{\xi}^{\xi_0} \frac{C \eta}{( \xi')^{2}} \exp \Big( \int_{\xi}^{\xi'}  \frac{C \eta}{( \xi'')^{2}}  \, \ud \xi'' \Big) \ud \xi'  \\
& \lesssim_\la  \eps 
} 
for all $n$ large enough. This completes the proof. 
\end{proof} 

We can now sketch the proofs of Proposition~\ref{p:stzH} and Corollary~\ref{c:stzdt}. 
\begin{proof}[Sketch of the proof of Proposition~\ref{p:stzH}]
The proposition is a consequence of the standard energy-flux identify over the triangle (in $(\eta, \xi)$ coordinates) with vertices, $A = (\xi, 0)$, $B = ( \xi, \ti \la \xi)$, $C = ( \xi + \ti \la \xi, 0)$. Note that the line segment connecting $B$ to $C$ is precisely the constant $t$ slice between $r = \la( T_+-t)$ and $r = T_+ - t$. The flux energy computation yields 
\EQ{
\int_{\la( T_+- t)}^{T_+ - t} e_n(t, r) \, r \, \ud r  \le h_n(\xi, \xi) + g_n(0, \xi, \xi+ \ti \la \xi).
}
We now conclude that the right-hand-side above can be made small uniformly in $n$ using Lemma~\ref{l:unflux} and Lemma~\ref{l:hsmall}. This proves Lemma~\ref{l:undec} and thus also Proposition~\ref{p:stzH}. 
\end{proof} 
\begin{proof}[Sketch of the proof of Corollary~\ref{c:stzdt}]
This follows by invoking the corresponding proof in Shatah, Tahvildar-Zadeh, i.e, ~\cite[Proof of Corollary 2.3]{STZ1} for the approximating sequence $\vec \psi_n(t)$ together with Proposition~\ref{p:stzH} and the uniform decay of the flux from Lemma~\ref{l:unflux}. 
\end{proof}

 \bigskip

\centerline{\scshape Rapha\"el C\^{o}te }
\medskip
{\footnotesize
\begin{center}
CNRS and \'{E}cole Polytechnique \\
Centre de Math\'ematiques Laurent Schwartz UMR 7640 \\
Route de Palaiseau, 91128 Palaiseau cedex, France \\
\email{cote@math.polytechnique.fr}
\end{center}
} 

\medskip

\centerline{\scshape Carlos Kenig, Andrew Lawrie, Wilhelm Schlag}
\medskip
{\footnotesize
% please put the address of the first author
 \centerline{Department of Mathematics, The University of Chicago}
\centerline{5734 South University Avenue, Chicago, IL 60615, U.S.A.}
\centerline{\email{cek@math.uchicago.edu, alawrie@math.uchicago.edu, schlag@math.uchicago.edu}}
} % Do not forget to end the {\footnotesize by the sign }


\begin{thebibliography}{10}

 
 \bibitem{BG} Bahouri, H.,  G\'{e}rard, P.  {\em High frequency approximation of solutions to critical nonlinear wave equations}. Amer.\ J.\ Math., 121 (1999), 131--175.
 
 
  \bibitem{BKT} Bejenaru, I., Krieger, J., Tataru, D.  {\em A codimension two stable manifold of near soliton equivariant wave maps}. Preprint 2012  arXiv:1109.3129.
 
 
 \bibitem{Bu} Bulut, A. {\em Maximizers for the Strichartz inequalities for the wave equation}. Differential Integral Equations 23 (2010), no.~11-12, 1035--1072.

  \bibitem{BCLPZ} Bulut, A.,  Czubak, M.,  Li, D., Pavlovi\'c, N.,  Zhang, X. {\em Stability and Unconditional Uniqueness of Solutions for Energy Critical Wave Equations in High Dimensions}. Preprint 2009
  
  
  \bibitem{Co} C\^{o}te, R.  {\em Instability of nonconstant harmonic maps for the $(1+2)$-dimensional equivariant wave map system}. Int.\ Math.\ Res.\ Not.\ 2005, no.~57, 3525--3549.
 
 \bibitem{CKM} C\^{o}te, R.,  Kenig, C.,Merle, F. {\em 
Scattering below critical energy for the radial 4D Yang-Mills equation and for the $2D$ corotational wave map system.}  
Comm.\ Math.\ Phys.\ 284 (2008), no.~1, 203--225. 
 
  \bibitem{CKLS2} C\^{o}te, R.,  Kenig, C., Lawrie, A.,  Schlag, W. {\em Characterization of large energy solutions of the equivariant wave map problem: II}. To appear in Amer. J. Math. Preprint 2012. arxiv: 1209.3684
 
 \bibitem{CKS} C\^{o}te, R.,  Kenig, C., Schlag, W. {\em Energy partition for the linear radial wave equation}. To appear in Math. Ann. Preprint 2012. arXiv:1209.3678 
 
 \bibitem{CTZ}  Christodoulou, D.,  Tahvildar-Zadeh, A.\ S.  {\em On the regularity of spherically symmetric wave maps}. Comm.\ Pure Appl.\ Math.\ 46 (1993), no.~7, 1041--1091.
 
 \bibitem{DKM1} Duyckaerts, T., Kenig, C.,  Merle, F. {\em Universality of blow-up profile for small radial type II blow-up solutions of the energy-critical wave equation}. J.\ Eur.\ Math.\ Soc.~(JEMS) 13 (2011), no.~3, 533--599. 
 
  \bibitem{DKM2} Duyckaerts, T., Kenig, C.,  Merle, F. {\em Universality of the blow-up profile for small type II blow-up solutions of energy-critical wave equation: the non-radial case}.   J.\ Eur.\ Math.\ Soc.\ (JEMS)  14 (2012) no. 5 1389--1454.	
  
    \bibitem{DKM3} Duyckaerts, T., Kenig, C.,  Merle, F. {\em Profiles of bounded radial solutions of the focusing, energy-critical wave equation}.  Geom. Funct. Anal. 22 \ (2012) no. 3, 639--698.
    
    \bibitem{DKMe} Duyckaerts, T., Kenig, C.,  Merle, F. {\em Erratum to ``Profiles of bounded radial solutions of the focusing, energy-critical wave equation"}. arXiv e-print.  arXiv:1201.4986v2
      
   \bibitem{DKM4} Duyckaerts, T., Kenig, C.,  Merle, F. {\em Classification of radial solutions of the focusing, energy-critical wave equation}.  Preprint 2012. 	arXiv:1204.0031v1.
   
 \bibitem{G} Grillakis, M. {\em Classical solutions for the equivariant wave maps in $1+2$ dimensions}. Preprint, 1991
 
 \bibitem{Hel} H$\acute{\textrm{e}}$lein, F. {\em Harmonic maps, conservation laws and moving frames}. Translated from the 1996 French original. With a foreword by James Eells. Second edition. Cambridge Tracts in Mathematics, 150. Cambridge University Press, Cambridge, 2002.
 
 \bibitem{JK-AJM} Jia, H., Kenig, C.\ E. {\em Asymptotic decomposition for semilinear wave and equivariant wave map equations} Amer. J. Math. 139 (2017) no. 6, 1521--1603.

 \bibitem{KM06} Kenig, C.\ E., Merle, F.  {\em Global well-posedness, scattering and blow-up for the energy-critical, focusing, non-linear Schr\"odinger equation in the radial case.} Invent.\ Math.\ 166 (2006), no.~3, 645--675. 

 
 \bibitem{KM08} Kenig, C.\ E., Merle, F. {\em Global well-posedness, scattering and blow-up for the energy-critical focusing non-linear wave equation}. Acta Math.\ 201 (2008), no.~2, 147--212. 
 
 \bibitem{Kr} Krieger, J. {\em Global regularity and singularity development for wave maps}. Surveys in differential geometry. Vol.\ XII. Geometric flows, 167--201. Surveys in Differential Geometry, 12. International, Somerville, Mass., 2008.
 
 \bibitem{KS} Krieger, J., Schlag, W.  {\em Concentration compactness for critical wave maps}. EMS Monographs in Mathematics. European Mathematical Society (EMS), Z\"{u}rich, 2012.  
 
 \bibitem{KST}Krieger, J.,  Schlag, W.,  Tataru, D. {\em Renormalization and blow up for charge one equivariant critical wave maps}. Invent.\ Math.\ 171 (2008), no.~3, 543--615.
 
\bibitem{KST2} Krieger, J.,  Schlag, W.,  Tataru, D. {\em Renormalization and blow up for the critical Yang-Mills problem}. Adv.\ Math.\ 221 (2009), no.~5, 1445--1521.
 
 
  \bibitem{LS} Lawrie, A.,  Schlag, W. {\em Scattering for wave maps exterior to a ball}.  Adv.\  Math.\  232 (2013) no.1 57--97.
 
 \bibitem{LinS}  Lindblad, H.,  Sogge, C.\ D. {\em On existence and scattering with minimal regularity for semilinear wave equations.} J.\ Funct.\ Anal.\ 130 (1995), no.~2, 357--426. 
 
\bibitem{RR} Rapha\"{e}l, P., Rodnianski, I.  {\em Stable blow up dynamics for the critical corotational Wave map and equivariant Yang-Mills problems}.  Publi.\ I.H.E.S., in press.

\bibitem{RS} Rodnianski, I., Sterbenz, J. {\em On the formation of singularities in the critical $O(3)$ $ \sigma$-model}. Ann.\ of Math.\ 172, 187--242 (2010) 

 \bibitem{SU} Sacks, J.,  Uhlenbeck, K. {\em The existence of minimal immersions of $2$-spheres}.  Ann.\ of Math.\ (2) 113 (1981), no.~1, 1--24. 

\bibitem{Sh} Shatah, J. {\em Weak solutions and development of singularities of the $SU(2)$ $\sigma$-model}. Comm.\ Pure Appl.\ Math. 41 (1988), no.~4, 459--469.
 
 \bibitem{SS} Shatah, J.,  Struwe, M.  {\em Geometric wave equations}. Courant Lecture Notes in Mathematics, 2. New York University, Courant Institute of Mathematical Sciences, New York; American Mathematical Society, Providence, RI, 1998.  
 
  \bibitem{SS2} Shatah, J.,  Struwe, M. {\em The Cauchy problem for wave maps.} Int.\ Math.\ Res.\ Not.~2002, no.~11, 555--571.
 
 \bibitem{STZ1} Shatah, J., Tahvildar-Zadeh, A.\ S. {\em Regularity of harmonic maps from the Minkowski space into rotationally symmetric manifolds}. Comm.\ Pure Appl.\ Math.\ 45 (1992), no.~8, 947--941.

 
 \bibitem{STZ} Shatah, J.,  Tahvildar-Zadeh, A.\ S. {\em On the Cauchy problem for equivariant wave maps}. Comm.\ Pure Appl.\ Math.\ 47 (1994), no.~5, 719--754.

\bibitem{Sogge}  Sogge, C.\ D.  {\em Lectures on non-linear wave equations.} Second edition. International Press, Boston, MA, 2008.

 \bibitem{ST1} Sterbenz, J., Tataru, D. {\em Energy dispersed large data wave maps in $2+1$ dimensions}. Comm.\ Math.\ Phys.\ 298 (2010), no.~1, 139--230
 
 \bibitem{ST2}Sterbenz, J.,Tataru, D.  {\em Regularity of wave-maps in dimension $2+1$}. Comm.\ Math.\ Phys.\ 298 (2010), no.~1, 231--264.
 

  \bibitem{St} Struwe, M. {\em Equivariant wave maps in two space dimensions}. Comm.\ Pure Appl.\ Math.\ 56 (2003), no.~7, 815--823.

 
 \bibitem{T2} Tao, T. {\em Global regularity of wave maps II. Small energy in two dimensions}. Comm.\ Math.\ Phys.\ 224 (2001), no.~2, 443--544. 

\bibitem{T3}Tao, T. {\em Global regularity of wave maps III-VII}. Preprints 2008-2009.

\bibitem{Tat} Tataru, D.  {\em On global existence and scattering for the wave maps equation}. Amer.\ J.\ Math.\ 123 (2001), no.~1, 37--77. 

  \end{thebibliography}
 \end{document}